\documentclass[a4paper,11pt,reqno]{amsart}

\usepackage{amssymb}
\usepackage{epsfig}
\usepackage{amsfonts}
\usepackage{amsmath}
\usepackage{euscript}
\usepackage{amscd}
\usepackage{amsthm}
\usepackage{enumitem}
\DeclareMathAlphabet{\mathpzc}{OT1}{pzc}{m}{it}
\usepackage{enumitem}
\usepackage{color}
\usepackage{mathtools}
\usepackage[pagebackref, hypertexnames=false, colorlinks, citecolor=red, linkcolor=blue, urlcolor=red]{hyperref}
\usepackage{marginnote}

\usepackage[normalem]{ulem}
\newcommand{\marginextend}[1]{ \addtolength{\oddsidemargin}{-#1}  \addtolength{\evensidemargin}{-#1}
  \addtolength{\textwidth}{#1}\addtolength{\textwidth}{#1}}
\newcommand{\updownextend}[1]{ \addtolength{\topmargin}{-#1}  \addtolength{\textheight}{#1}
\addtolength{\textheight}{#1}}
\marginextend{1cm}
\updownextend{0cm}
\allowdisplaybreaks[4]

\usepackage{mathrsfs}
\DeclareFontFamily{OT1}{pzc}{}
\DeclareFontShape{OT1}{pzc}{m}{it}{<-> s * [1.10] pzcmi7t}{}
\DeclareMathAlphabet{\mathpzc}{OT1}{pzc}{m}{it}

\DeclareSymbolFont{SY}{U}{psy}{m}{n}
\DeclareMathSymbol{\emptyset}{\mathord}{SY}{'306}

\theoremstyle{plain}

\newtheorem{thm}{Theorem}[section]
\newtheorem*{thm*}{Theorem}
\newtheorem{cor}[thm]{Corollary}
\newtheorem{lem}[thm]{Lemma}
\newtheorem{prop}[thm]{Proposition}
\newtheorem{defn}[thm]{Definition}
\newtheorem{rem}[thm]{Remark}
\newtheorem{ex}[thm]{Example}

\newtheoremstyle{named}{}{}{\itshape}{}{\bfseries}{.}{.5em}{#1 \thmnote{#3}}
\theoremstyle{named}

\numberwithin{equation}{section}
\def\C{{\mathbb C}}

\def\norm#1{\left\|{#1}\right\|}

\def\R{\mathbb{R}}

\def\v{\varphi}
\def\l{\lambda}

\def\ra{\rightarrow}
\def\ov{\overline}

\def\m{\mathcal}
\def\mb{\mathbb}
\def\mr{\mathrm}
\def\a{\alpha}
\def\b{\beta}

\def\wi{\widetilde}

\def\w{\widehat }

\def\beq{\begin{eqnarray}}
\def\eeq{\end{eqnarray}}
\def\beqa{\begin{eqnarray*}}
\def\eeqa{\end{eqnarray*}}

\def\del{\partial}

\def\ov{\overline}
\def\bl{\boldsymbol}

\def\i{\prime}

\def\bl{\boldsymbol}
\newcommand{\be}{\begin{equation}}
\newcommand{\ee}{\end{equation}}
\newcommand{\bea}{\begin{eqnarray}}
\newcommand{\eea}{\end{eqnarray}}
\newcommand{\Bea}{\begin{eqnarray*}}
\newcommand{\Eea}{\end{eqnarray*}}
\newcommand{\inner}[2]{\langle #1,#2 \rangle }%

\newcounter{cnt1}
\newcounter{cnt2}
\newcounter{cnt3}
\newcommand{\blr}{\begin{list}{$($\roman{cnt1}$)$}
 {\usecounter{cnt1} \setlength{\topsep}{0pt}
 \setlength{\itemsep}{0pt}}}
\newcommand{\bla}{\begin{list}{$($\alph{cnt2}$)$}
 {\usecounter{cnt2} \setlength{\topsep}{0pt}
 \setlength{\itemsep}{0pt}}}
\newcommand{\bln}{\begin{list}{$($\arabic{cnt3}$)$}
 {\usecounter{cnt3} \setlength{\topsep}{0pt}
 \setlength{\itemsep}{0pt}}}
\newcommand{\el}{\end{list}}

\newcommand {\Poincare} {Poincar\'e\ }

\DeclareMathOperator  {\Aut} {Aut}

\DeclareMathOperator  {\tr} {tr}
\DeclareMathOperator  {\adj}{adj}

\begin{document}
%\title {Action of finite pseudoreflection groups on holomorphic functions}
\title[Reducing submodules and CST theorem] {Reducing submodules of Hilbert Modules and Chevalley-Shephard-Todd Theorem}
\author[Biswas]{Shibananda Biswas}
\author[Datta]{Swarnendu Datta}
\author[Ghosh]{Gargi Ghosh}
\author[Shyam Roy]{Subrata Shyam Roy}
\address[Biswas, Datta, Ghosh and Shyam Roy]{Department of Mathematics and Statistics, Indian Institute of Science Education and Research Kolkata, Mohanpur 741246, Nadia, West Bengal, India}
\email[Biswas]{shibananda@iiserkol.ac.in}
\email[Datta]{swarnendu.datta@iiserkol.ac.in}
\email[Ghosh]{gg13ip034@iiserkol.ac.in}
\email[Shyam Roy]{ssroy@iiserkol.ac.in}
\thanks{The work of S. Biswas is partially supported by Inspire Faculty Fellowship (IFA-11MA-06)
funded by DST, India at IISER Kolkata. The work of G. Ghosh is supported by Senior Research Fellowship funded by CSIR}

\date  {}

\begin{abstract}
%If $D_1,\, D_2$ are domains in $\mb C^n$ and $f:D_1\to D_2$ is a proper holomorphic map which  is factored by automorphisms  $G\subseteq {\rm Aut}(D_1).$
%such that
%\Bea
%f^{-1}f(z)=\bigcup_{\rho\in G}\{\rho(z)\} \,\, for \,\, z\in D_1.
%\Eea

Let $G$ be a finite pseudoreflection group, $\Omega\subseteq \mathbb C^n$ be a bounded domain which is a $G$-space and $\mathcal H\subseteq\mathcal O(\Omega)$ be an analytic Hilbert module possessing a $G$-invariant reproducing kernel. We study the structure of joint reducing subspaces of the multiplication operator $\mathbf M_{\boldsymbol\theta}$ on $\mathcal H,$ where $\{\theta_i\}_{i=1}^n$ is a homogeneous system of parameters associated to $G$ and $\boldsymbol\theta = (\theta_1, \ldots, \theta_n)$ is a polynomial 
map of $\mathbb C^n$.  We show that it admits a family $\{\mb P_\varrho\m H:\varrho\in\w G\}$
%$\{\mb P_\varrho^{ii}\m H:\varrho\in\w G,\, 1\leq i\leq {\rm deg}\,\varrho\}$ 
of non-trivial joint reducing subspaces, where $\w G$ is the set of all equivalence classes of irreducible representations of $G.$ We prove a generalization of Chevalley-Shephard-Todd theorem for the algebra $\mathcal O(\Omega)$ of holomorphic functions on $\Omega$. As a consequence, we show that for each $\varrho\in \w G,$ the multiplication operator $\mathbf M_{\bl\theta}$ on the reducing subspace $\mb P_\varrho \m H$
%= \oplus_{i=1}^{\deg \varrho}\mb P_\varrho^{ii}\m H$ 
can be realized as multiplication by the coordinate functions on a reproducing kernel Hilbert space of $\mb C^{({\deg}\,\varrho)^2}$-valued holomorphic functions on $\bl\theta(\Omega)$.  This, in turn, provides a description of the structure of joint reducing subspaces of the multiplication operator induced by a representative of a proper holomorphic map from a domain $\Omega$ in $\C^n$ which is factored by automorphisms $G\subseteq {\rm Aut}(\Omega).$  
\end{abstract}

\maketitle

%%%%%%%%%%%%%%%%%%%%%%%%%%%%%%%%%%%%%%%%%%%%%%%%%%%%%%%%%%%%%%%%%%%%%%%%%%%%%%%%%%%%%%%%

\section {Introduction} \label{intro}
For $n\geq 1,$ suppose that $\Omega$ is a bounded domain in $\mb C^n$ and $\m H$ is a reproducing kernel Hilbert space consisting of holomorphic functions on $\Omega$. We study joint reducing subspaces of (bounded) multiplication operators by holomorphic functions on $\m H$. A {\it joint reducing subspace} for a  tuple 
$\mathbf T=(T_1,\ldots,T_n)$ of operators on a Hilbert space $\m H$ is a closed subspace $\m M$ of $\m H$ such that 
\Bea
T_i\m M\subseteq\m M \text{~ and~} T_i^*\m M\subseteq\m M \text{~for~ all~} i=1,\ldots,n.
\Eea 
It is easy to see that $\mathbf T$ admits a  joint reducing subspace of  if and only if there exists an orthogonal projection $P$ ($P:\m H\to\m H, P^2=P=P^*$) such that 
\Bea
PT_i=T_iP \text{~for~all ~} i=1,\ldots, n.
\Eea 
Here $\m M=P\m H$ is the corresponding joint reducing subspace of $\mathbf T.$ A joint reducing subspace $\m M$ of $\mathbf T$  is called {\it minimal} or {\it irreducible} if only joint reducing subspaces contained in $\m M$ are $\m M$ and $\{0\}.$ In particular, the operator tuple $\mathbf T$ is called ${\it irreducible}$ if only  joint reducing subspaces contained in $\m H$ are $\m H$ and $\{0\}.$ For instance, suppose that $\Omega$ is a bounded domain in $\mb C^n$ and $\mb A^2(\Omega)$ denotes the Bergman space on $\Omega,$ that is, the subspace of  holomorphic functions in $L^2(\Omega)$ with respect to the Lebesgue measure on $\Omega.$ We put $\mathbf M:=(M_1,\ldots, M_n),$ where $M_i$ denotes the multiplication by the $i$-th coordinate function of $\Omega$ on $\mb A^2(\Omega)$ for $i=1,\ldots, n$. $\mathbf M$ is called the {\it Bergman operator} on $\mb A^2(\Omega)$ or the Bergman operator {\it associated to} a domain $\Omega.$ It is well-known that $\mathbf M$ is irreducible.

It is natural to start with a biholomorphic multiplier $\bl f$ on $\mb A^2(\Omega).$   It is easy to see that $\mathbf M_{\bl f}$ is irreducible as follows. Let $\Omega_1,\,\Omega_2$ be bounded domains in $\mb C^n$ and $\bl f=(f_1,\ldots, f_n):\Omega_1\to\Omega_2$ be a biholomorphism. Let $J_{\bl f}=\det\big(\!\!\big(\frac{\partial f_i}{\partial z_j}\big)\!\!\big)_{i,j=1}^n$ denote the jacobian of $\bl f$. The linear map $U_{\bl f}:\mb A^2(\Omega_2)\to\mb A^2(\Omega_1)$ defined by 
\Bea
U_{\bl f}\v=(\v\circ \bl f)J_{\bl f} \text{~for~} \v\in \mb A^2(\Omega_2). 
\Eea
%From change of variable formula, it follows that $U_f$ is an isometry. 
is a unitary satisfying $U_{\bl f}M_i=M_{f_i}U_{\bl f}$ for $i=1,\ldots,n$. Therefore, $\mathbf M_{\bl f}$ is unitarily equivalent to the Bergman operator on $\mb A^2(\Omega_2).$ Since the Bergman operator associated to any domain is irreducible, it follows that $\mathbf M_{\bl f}$ is irreducible.

This is the motivation for considering the well behaved class of bounded multipliers induced by proper holomorphic maps. It is shown in \cite{G, HZ} that for a bounded domain $\Omega$ in $\mb C^n$ and a bounded proper holomorphic multiplier, $\bl f:\Omega\to \bl f(\Omega)$   on the Bergman space  $\mb A^2(\Omega)$  always admits a non-trivial joint reducing subspace. The problem of understanding the nature of joint reducing subspaces of proper holomorphic multipliers and parametrizing the number of minimal joint reducing subspaces by a topological invariant (for example, the number of connected components of some Riemann surface) has been studied extensively in the literature for one and several variables. On the Hardy space of the unit disc $\mb D$ in the complex plane $\mb C,$ for the multiplication operator by a finite Blaschke  product (the class of all proper holomorphic self-maps of $\mb D$) the aforementioned problem was studied in \cite{C, JT, T}. Analogous problem was investigated in the context of Bergman space $\mb A^2(\mb D)$ of the unit disc $\mb D$ by several authors, see \cite{DPW,GHu, DSZ, Z} and references therein. For a bounded domain $\Omega$ in $\mb C^n,$ similar problem was studied in \cite{G, HL, HZ}. In this paper, we study the structure of joint reducing subspaces of the multiplication operator induced by a representative (cf. \cite{BD}) of a proper holomorphic map $\bl f$ from a domain $\Omega$ in $\C^n$ which is factored by automorphisms $G\subseteq {\rm Aut}(\Omega).$ A representative determines the proper holomorphic map up to a biholomorphism.

%We apply the technique outlined so far to find joint reducing subspaces of a bounded proper holomorphic multiplier on a bounded domain $\Omega$ in $\mb C^n$ which is factored by  automorphisms. 
Following \cite{BD, DS}, we say that a proper holomorphic map $ \bl f:\Omega_1\to \Omega_2$ is {\it factored by automorphisms} if there exists a finite subgroup $G\subseteq {\rm Aut}(\Omega_1)$ such that
\Bea
\bl f^{-1}\bl f(\bl z)=\bigcup_{\rho\in G}\{\rho(\bl z)\} \,\, \text{~for~} \bl z\in \Omega_1.
\Eea
Equivalently,  $\bl f$ is  a (branched) Galois covering  with $G$ as the group of deck transformations of $\bl f.$ We know from \cite[Theorem 2.1, Theorem 2.2]{BD}, \cite[Lemma 2.2, Theorem 2.5]{DS}, \cite[Theorem 1.6]{R}, \cite[p.506]{BB} that such a group $G$ is either a group generated by pseudoreflections or conjugate to a pseudoreflection group. In  Section \ref{last}, we provide a list of known sufficient conditions for a large class of domains $\Omega_1$ and $\Omega_2$ so that every proper holomorphic map $\bl f:\Omega_1\to\Omega_2$ is factored by automorphisms. These observations naturally lead to the study of finite groups generated by pseudoreflections on $\mb C^n$ and joint reducing subspaces of proper holomorphic multipliers induced by them.
%ametrize the number of joint reducing subspaces of a proper holomorphic multiplier on a Hilbert space consisting of holomorphic functions on $\Omega$ with $G$-invariant reproducing kernel which is factored by automorphisms $G\subseteq {\rm Aut}(\Omega).$ In this process, we obtain some new results which are of independent interest.

%We begin by describing a procedure of obtaining a family of  arising naturally from the class of finite pseudoreflection groups.  
A {\it pseudorelfection} on $\mb C^n$ is a linear mapping $\rho:\mb C^n\to\mb C^n$ such that $\rho$ has finite order in $GL(n,\C)$ and the rank of $I_n-\rho$ is $1,$ that is, $\rho$ is not the identity map and fixes a hyperplane pointwise. A group $G$ generated by pseudoreflections is called a {\it pseudoreflection group}. Let $G$ act on a set of functions on $\mb C^n$ by 
\Bea
\rho(f)(\bl z)= f(\rho^{-1}\cdot\bl z) \,\, \text{~for~} \bl z\in\mb C^n \text{~and~} \rho\in G.
\Eea
A function $f$ is called $G$-{\it invariant} if $\rho(f)=f$ for all $\rho\in G.$ There is a system of $G$-invariant  algebraically independent homogeneous polynomials $\{\theta_i\}_{i=1}^n$ associated to a pseudoreflection group $G,$ called a homogeneous system of  parameters (hsop) or basic polynomials associated to $G.$ %The degrees $d_i$ of $\theta_i$  for $i=1,\ldots,n,$  satisfy $d_1\cdots d_n=\lvert G\rvert$ and $\sum_{i=1}^n(d_i-1)=r,\, r$  being the number of pseudoreflections contained in $G$ \cite[Corollary 4.4]{S}.  
A finite pseudoreflection group $G$ is characterized by the fact that the ring of invariants of $G$ is a polynomial ring generated by some hsop $\{\theta_i\}_{i=1}^n$ associated to $G$ \cite[p.282]{ST}. 

Let $\Omega$ be a bounded domain in $\C^n$ and assume that it is a $G$-space by the (right) group action $(\rho, \bl z)\mapsto \rho^{-1}\bl z,\, \bl z\in\Omega,\,\rho\in G.$ We call such a $\Omega$ a $G$-invariant domain. It follows from \cite[Proposition 2.2]{R} and \cite[Proposition 1, p.556]{Try} that the mapping
$
\bl\theta=(\theta_1,\ldots,\theta_n):\Omega\to\bl\theta(\Omega),
$
 is a proper holomorphic mapping, $\bl\theta(\Omega)$ is a domain and the quotient topological space $\Omega/G$ is a complex analytic space biholomorphic to $\bl\theta(\Omega).$ Equivalently, $\bl\theta:\Omega\to\bl\theta(\Omega)$ is a (branched) Galois covering and $G$ is the group of deck transformations of this covering.

We proceed by fixing an appropriate functional Hilbert space consisting of holomorphic functions defined on a domain $\Omega$ in $\mb C^n.$ 
%on which a proper map acts as a multiplier. Such a Hilbert space $\m M$ will necessarily consist of holomorphic functions defined on a domain $\Omega$ in $\mb C^n$ 
Such a Hilbert space $\m H$ may be thought of as a {\it Hilbert module} over a algebra $\m A\subseteq \mathcal O(\Omega)$ if the action of $\m A$ on $\m H$ defines an algebra homomorphism $f\mapsto T_f$ of $\m A$ into $\m L(\m H),$ where $T_f$ is the bounded operator on $\m H$ defined by $T_fh=f \cdot h.$  A Hilbert module $\m H$ consisting of $\C^r$-valued holomorphic functions is said to be an \emph{analytic Hilbert module} over the polynomial ring $\C[z_1,\ldots, z_n] (= \C[\bl z])$ if it possesses a reproducing kernel, the action of $\mb C[\bl z]$ is given by pointwise multiplication in each component and $\mb C[\bl z]\otimes\C^r$ is dense in $\m H.$  In this paper, we consider an analytic Hilbert module $\m H$ consisting of scalar valued holomorphic functions  on a bounded $G$-invariant domain $\Omega$ in $\mb C^n$ possessing a $G$-invariant reproducing kernel $K,$ that is,
\Bea
K(\rho\cdot\bl z \,,\,\rho\cdot\bl w)=K(\bl z,\bl w) \,\,\text{~for~}\, \rho\in G \,\text{~and ~}\, \bl z,\,\bl w\in\Omega.
\Eea
Let $\mb C[\bl z]^G = \{p\in\C[\bl z]: \rho(p) = p \mbox{~for~all~}\rho\in G\}$ be the ring of $G$-invariant polynomials. In order to study reducing subspaces of $\mathbf M_{\bl\theta}$ on such a $\m H,$ it is convenient to view an analytic Hilbert module $\m H$ over $\mb C[\bl z]$ as a module over the subalgebra $\mb C[\bl z]^G$ by inheriting the module action.
%Note that the subalgebra $\mb C[\bl z]^G$ of $\mb C[\bl z]$  of $G$-invariant polynomials, induces the module action inherited from $\mb C[\bl z],$  that is, pointwise multiplication on $\m H$ to make it a Hilbert module over $\mb C[\bl z]^G$. 

%In this context, it is natural to ask the following questions. 

Let $\w G$ denote the equivalence classes of all irreducible representations of $G.$ It is well-known that these are finite dimensional. Let $\bl\pi_\varrho$ be a unitary representation of $G$ in the equivalence class of $\varrho\in\w G,$ that is, ${\bl\pi}_\varrho(\sigma)=
\big(\!\!\big(\bl\pi_\varrho^{ij}(\sigma)\big)\!\!\big)_{i,j=1}^m\in\mb C^{m\times m}, \sigma\in G,$ where $m$ denotes the degree of the representation $\varrho$ which we abbreviate as ${\rm deg}\,\varrho.$ Let  $\chi_\varrho (\sigma)={\rm trace}\big(\bl\pi_\varrho(\sigma)\big)$ denote the character of the representation $\varrho.$
%$\, 1_G:=1$ being the identity of the group $G.$ 
These finite dimensional representations of  $G$ induce the linear operators $\mb P_\varrho$ and $\mb P_\varrho^{ij}$ on the Hilbert space $\m H:$
\Bea && \mb P_\varrho f=\frac{{\rm deg}\,\varrho}{\lvert G\rvert}\sum_{\sigma\in G}\chi_\varrho(\sigma^{-1})(f\circ\sigma^{-1}),\,\, f\in\m H;\\
&& \mb P_\varrho^{ij} f=\frac{{\rm deg}\,\varrho}{\lvert G\rvert}\sum_{\sigma\in G}\bl\pi_\varrho^{ji}(\sigma^{-1})(f\circ\sigma^{-1}),\,\, f\in\m H.
\Eea
We note that $ \mb P_\varrho$ and $\mb P_\varrho^{ii}$ are  orthogonal projections on $\m H$ for $\varrho\in\w G,\, 1\leq i\leq {\rm deg}\,\varrho,$ as it was in the case for $G=\mathfrak S_n$, where $\mathfrak S_n$ is the permutation group on $n$ symbols \cite{BGMS}.
%These projections are discussed in detail for a finite group $G$ and specialized for  $\mathfrak S_n$ in \cite{BGMS}, $\mathfrak S_n$ being the permutation group on $n$ symbols.
Consequently, the Hilbert space $\m H$ admits the orthogonal decomposition
\Bea
\m H=\bigoplus_{\varrho\in\w G}\mb P_\varrho\m H=\bigoplus_{\varrho\in\w G}\bigoplus_{i=1}^{{\rm deg}\,\varrho}\mb P_\varrho^{ii}\m H.
\Eea
We now briefly describe the results we have proved in this paper.

\begin{enumerate}[leftmargin=*]
    \item  [1.] {\sf Existence of non-trivial reducing subspaces}. For any finite pseudoreflection group $G$ and analytic Hilbert module $\m H\subseteq \mathcal O(\Omega)$ with $G$-invariant reproducing kernel, we show that $\mb P_\varrho\m H$  and $\mb P_\varrho^{ii}\m H$ are non-trivial reducing submodules of $\m H$ over $\mb C[\bl z]^G$ for $\varrho\in\w G$ and $1\leq i \leq {\rm deg}\,\varrho.$ This generalizes the results obtained for the case $G = \mathfrak S_n$ in \cite{BGMS}.
   \vspace{.1 cm}
  
     \vspace{.1 cm}
    \item[2.] {\sf Analytic Chevalley-Shephard-Todd Theorem}. We prove a generalization of the well-known Chevalley-Shephard-Todd theorem to the algebra of holomorphic functions on $\C^n,$ see Theorem \ref{theoremb} and Theorem \ref{theoremc}. We extend this results for holomorphic functions on $G$-invariant domains. Apart from being interesting in their own rights, Theorem \ref{theoremc} plays a crucial role in realizing the multiplication operator $\mathbf M_{\bl\theta}$ as multiplication operator by the coordinate functions on a suitable functional Hilbert space. This is an important step to obtain geometric and analytic unitary invariants (see \cite{CD, CS}) to classify the family of submodules $\{\mb P_\varrho\m H\}_{\varrho\in\w G}$ completely, up to unitary equivalence. A successful implementation of this strategy for a special case is given in \cite[Section 4]{BGMS}.
   \vspace{.1 cm}
   \item[3.]{\sf Realization as analytic Hilbert modules}. 
  We show that an analytic Hilbert module $\m H$ over $\mb C[\bl z]$ on a $G$-invariant domain $\Omega$ in $\mb C^n,$  viewed as a Hilbert module over $\mb C[\bl z]^G$ can be realized as an analytic Hilbert module $\wi{\m H}$ over $\mb C[\bl z]$ consisting of $\mb C^{\lvert G\rvert}$-valued holomorphic functions on $\bl\theta(\Omega)$, see Lemma \ref{cons}. It then follows that the submodules $\mb P_\varrho\m H$ has a realization as an analytic Hilbert module of rank $({\rm deg}\,\varrho)^2$ over $\mb C[\bl z]$ consisting of holomorphic functions on $\bl\theta(\Omega)$. Indeed, this is a generalization of the Chevalley-Shephard-Todd Theorem in the context of  analytic Hilbert modules. As an application towards classification of the family of reducing submodules
   \Bea
    \{\mb P_\varrho^{ii}\m H:\varrho\in\w G, 1\leq i\leq {\rm deg}\,\varrho\}
    \Eea
    over $\mb C[\bl z]^G$ under unitary equivalence, we prove that if $\varrho,\,\varrho^\i\in\w G$ are such that ${\rm deg}\,\varrho\neq {\rm deg}\,\varrho^\i,$ then, $\mb P_\varrho^{ii}\m H$ and $\mb P_{\varrho^\i}^{jj}\m H$ are not unitarily equivalent for any $i, j$ with $1\leq i\leq {\rm deg}\,\varrho$ and $1\leq j\leq {\rm deg}\,\varrho^\i,$ see Theorem \ref{inequihilmod}. 
  \item[4.]{\sf Examples}.
  We exhibit a large class of examples of analytic Hilbert modules with $G$-invariant reproducing kernels for different choices of finite pseudoreflection groups $G$ and $G$-invariant domains $\Omega$ to which the results described so far apply. %In fact, one can find a huge body of literature on individual spaces. This paper could be thought of as an attempt to unify these studies and a program, we would like to investigate in future. This leads us to apply our results to a class of examples in the last section.
 In the last section, we prove several results using  techniques developed in this paper regarding reducing subspaces for a class of weighted Bergman space on the unit disc,  which was  obtained earlier in \cite{Z} for the Bergman space. We further obtain new results similarly in case of unit polydisc. Building on the work of Rudin in \cite{R}, we also show that the study of reducing subspaces in the case of the unit ball relates to homogeneity \cite{BM} of associated multiplication operator with respect to the automorphism groups of an inhomogeneous domains that arises naturally in this context.
 \end{enumerate}
 
Realization of the submodule $\mb P_\varrho\m H$ as an analytic Hilbert module of rank $({\rm deg}\,\varrho)^2$ on $\bl \theta(\Omega)$ is a significant improvement of the result obtained  in  \cite[Corollary 3.12]{BGMS} for $\mathfrak S_n$ showing that for any partition $\bl p$ of $n, \mb P_{\bl p}\m H$ is a locally free Hilbert module of rank $({\rm deg}\,\bl p)^2$ on an open subset of $\bl s(\Omega) \,\setminus\, \bl s(\m Z),$  $\bl s:\mb C^n\to\mb C^n$ being the symmetrization map whose $i$-th component is the $i$-th elementary symmetric polynomial in $n$ variables  and $\m Z$ being the set of critical points of $\bl s.$
%Furthermore, it is shown that $\mb P_\varrho^{ii}\m H$  is a Hilbert module of rank $\chi_\varrho(1)$ over $\mb C[\bl z]^G$ for $1\leq i\leq\chi_\varrho(1),$ see Corollary \ref{srank}. Consequently, if   $\chi_\varrho(1)\neq\chi_{\varrho^\i}(1)$   we show that the submodules $\mb P_\varrho^{ii}\m H$ and $\mb P_{\varrho^\i}^{jj}\m H$ are not unitarily equivalent for $\varrho, \varrho^\i\in\w G$ and $1 \leq i\leq\chi_\varrho(1),\, 1 \leq j \leq\chi_{\varrho^\i}(1),$ see Theorem \ref{inequihilmod}. 
We also obtain a lower bound on the number of minimal reducing submodules of $\m H$ over $\mb C[\bl z]^G,$ namely, $\sum_{\varrho\in\w G}{\rm deg}\,\varrho,$ see Theorem \ref{fd1}. Moreover, if $\mb P_\varrho^{ii}\m H$ is minimal for all $\varrho\in\w G$ and $1\leq i\leq {\rm deg}\,\varrho,$ then $\lvert\w G\rvert$ is an upper bound for the number of inequivalent reducing submodules of $\m H$ over $\mb C[\bl z]^G,$ see Corollary \ref{ub}.

Before closing this section, we would like to elaborate a little on the analytic Chevalley-Shephard-Todd Theorem. One might consider proving the theorem by first proving it for formal power series and then addressing the question of convergence. A similar approach was made in \cite[Theorem 1.2]{A} where it was proved that analytic solution to an arbitrary system of analytic equation exists provided a solution in formal power series without constant term exists. In our case, it is indeed true that the ring $\C[[\bl z]]$ of formal power series is a free module over $\C[[\bl z]]^{G}$, the ring of invariants under the action of $G$, of rank $\lvert G\rvert.$  Using $\mb C[[\bl z]]^G=\mb C[[\theta_1,\ldots,\theta_n]]$ (following \cite[Theorem 2.a, A.IV.68]{Bourbaki1}) and the fact that $\C[[\bl z]]$ (respectively $\C[[\bl z]]^{G}$) is the completion of the polynomial ring $\C[\bl z]$ (respectively $\C[\bl z]^G$) in $\mathfrak M$-adic topology, where $\mathfrak M$ is the maximal ideal generated by the indeterminates \cite[Example 1, p. 260]{ZS} and a commutative algebra technique \cite[Proposition 10.13]{AM}, the result for formal power series follows. However, it is not clear, a priori, whether in the decomposition of a holomorphic function thought of as a formal power series with respect to a basis of $\mb C[[\bl z]]$ over $\mb C[[\bl z]]^G,$ the formal power series arising as coefficients with respect to that basis are necessarily holomorphic, that is, whether they converge on $\C^n.$ Instead we approach Theorem \ref{theoremb} independently by obtaining the formulae for the coefficients with respect to a basis of $\C[\bl z]$ over $\C[\bl z]^G$.

\section{Preliminaries}
We begin by  recalling a number of useful definitions and standard results. We provide more details than usual because of the diverse nature of these.
 
\subsection{Pseudoreflection Groups}
\begin{defn}
A pseudoreflection on $\C^n$ is a linear homomorphism $\rho: \C^n \rightarrow \C^n$ such that $\rho$ has finite order in $GL(n,\mb C)$ and the rank of $I_n - \rho$ is 1, that is, $\rho$ is not the identity map and fixes a hyperplane pointwise. A group generated by pseudoreflections is called a pseudoreflection group.
\end{defn}

Recall that the group action of $G$ on $\C^n$ is defined by $(\rho, \bl z)\mapsto \rho\cdot\bl z = \rho^{-1}\bl z,\, \bl z\in\C^n,\,\rho\in G.$ For a pseudoreflection $\rho,$ define $H_{\rho} := \ker(I_n - \rho).$ By definition, the subspace $H_{\rho}$ has dimension $n-1$ and $\rho$ fixes the hyperplane $H_{\rho}$ pointwise. We call such hyperplanes \emph{reflecting.}
\begin{defn}
A hyperplane $H$ in $\C^n$ is called reflecting if there exists a pseudoreflection in $G$ acting trivially
on $H$.
\end{defn}
By definition, $H$ is the zero set of a non-zero homogeneous linear function on $\C^n$, determined up to a non-zero constant multiplier. 
\begin{ex}\label{reflex}
\begin{enumerate}
\item Let $G= \mb Z_n = \langle \sigma \rangle$ be the cyclic group of order $n.$ The action of $\mb Z_n$ on $\C$ is given by 
\Bea
\pi : G \to GL(1,\C) : \sigma \mapsto \zeta_n,
\Eea
where $\zeta_n$ denotes a primitive $n$-th root of unity. Each of $\sigma^j, j=1,\ldots,n-1,$ is a pseudoreflection. 
%In \cite{ST}, authors classify all finite irreducible pseudoreflection groups. They show that cyclic groups of order $n$ are irreducible pseudoreflection groups for any $n \in \mb N.$  
\vspace{.1 cm}
\item Let $G_1$ and $G_2$ be two finite groups generated by pseudoreflections acting on the vector spaces $V_1$ and $V_2,$ respectively. Then the product $G_1 \times G_2$ is also a pesudoreflection group acting $V_1 \oplus V_2$. Suppose $G$ is a finite abelian group. It is well-known that every non-trivial finite abelian group $G$ is isomorphic to a direct product of cyclic groups. More precisely,
\Bea
G \cong \mb Z_{m_1} \times \cdots \times \mb Z_{m_k},
\Eea
where $m_1|m_2|\ldots|m_k$ and $m_1>1$ \cite[p. 20]{Con}.
Then by the previous example, we conclude that $G$ is isomorphic to a finite pseudoreflection group.  Suppose $\mb Z_{m_j} = \langle \sigma_j \rangle.$ The action of $G$ on $\C \oplus \cdots \oplus \C$ ($k$ times) is given by 
\Bea
\pi : G \to GL(k,\C) :(1,\ldots, \sigma_j,\ldots,1) \mapsto {\rm diag}(1, \ldots,\zeta_j,\ldots, 1),
\Eea
where $\zeta_j$ denotes a $m_j$-th primitive root of unity and it is placed at $j$-th diagonal position for each $\pi(1,\ldots, \sigma_j,\ldots,1), j=1,\ldots,k.$
\vspace{.1 cm}
\item Let $G =\mathfrak S_n,$ the symmetric group on $n$ elements, acting on $\C^n$ by permuting  the coordinates. It is generated by transpositions $\{(i\,j)\}_{n \geq i>j \geq 1}$. It is easy to see that they are actually reflections, that is, pseudoreflections of order 2. Consider the following representation of $\mathfrak S_n,$ 
\Bea
\pi: \mathfrak S_n \to GL(n, \C): (i\,j) \mapsto A_{(i\,j)},
\Eea
%We correspond each $(i\,j)$ to some element of $GL(n, \C),$ say $A_{(i \,j)},$ 
where $A_{(i\,j)}$ is the permutation matrix obtained by interchanging the $i$-th and the $j$-th columns of the identity matrix.
%(z_1,\ldots,z_j,\ldots,z_i,\ldots,z_n)^{\tr} = (z_1,\ldots,z_i,\ldots,z_j,\ldots,z_n)^{\tr}.$ Clearly, each $A_{(i\,j)}$ has order 2.
Moreover, each $A_{(i\,j)}$ fixes the hyperplane $H_{ij} = \{\bl z \in \C^n : z_i = z_j\}$ pointwise and acts on the complementary line to $H_{ij}$ by $-1.$ 
\vspace{.1 cm}
\item Let $G = D_{2k} = \langle \delta, \sigma : \delta^k=\sigma^2 ={\rm Identity, \sigma \delta \sigma^{-1} = \delta^{-1}} \rangle$ be the dihedral group of order $2k.$ We define its action on $\C^2$ via the faithful representation $\pi$ defined by
\Bea
     \pi : G\to GL(2,\mb C): \delta\mapsto \begin{bmatrix}
    \zeta_k & 0\\
    0 & {\zeta_k}^{-1}
    \end{bmatrix}, \sigma\mapsto \begin{bmatrix}
    0 & 1\\
    1 & 0
    \end{bmatrix},
    \Eea
    where $\zeta_k$ denotes a primitive $k$-th root of unity. We write matrix representation of the group action with respect to the standard basis of $\C^2.$
 We have $G=\{\delta^j,\sigma\delta^j:j\in\{0,\ldots,k-1\}\},$ where $\delta^j$ is a rotation having the eigenvalues ${\zeta_k}^{\pm j}$ and $\sigma\delta^j$ is a reflection having the eigenvalues $\pm 1.$  Note that the hyperplane $H_j = \{\bl z \in \C^2 : \zeta_k^j z_1= z_2\}$ gets fixed by $\pi(\sigma\delta^j)$ for each $j \in \{0,\ldots,k-1\}.$
    
\end{enumerate}
\end{ex}
\subsection{Chevalley-Shephard-Todd Theorem}

Let $G$ be a finite pseudoreflection group that acts on the set of functions on  $\C^n$ by $\rho (f)(\bl z)=f({\rho}^{-1}\cdot \bl z).$ 
%$\C[z_1,\ldots,z_n]$ be the ring of polynomial functions on $\C^n$ and let
We recall the well-known theorems due to Chevalley-Shephard-Todd.
\begin{thm}\cite[Theorem 3, p.112]{bourbaki}\label{A}
The invariant ring $\C[\bl z]^G$ is equal to $\C[\theta_1,\ldots,\theta_n]$, where $\theta_i$ are algebraically independent
homogeneous polynomials. In particular, $\C[\bl z]^G$ is itself a polynomial algebra in $n$ variables.
\end{thm}
Let ${\bl\theta}: \C^n \rightarrow \C^n$ be the function defined by
\bea\label{theta}
{\bl\theta}(\bl z) = \big(\theta_1(\bl z),\ldots,\theta_n(\bl z)\big),\,\,\bl z\in\C^n.
\eea
Then Theorem \ref{A} can be restated as : for any $G$-invariant polynomial $p$ in $n$ variables there exists a unique polynomial $q$ in $n$ variables such that $p = q \circ \bl \theta.$
\begin{thm}\cite[Theorem 1, p.110]{bourbaki}\label{B}
The polynomial ring $\C[\bl z]$ is a free $\C[\bl z]^G$ module of rank $|G|$. Further, one can choose a  basis of $\C[\bl z]$
consisting of homogeneous polynomials.
\end{thm} 

It is known that the theorems above are equivalent  \cite[Theorem 4, p.120]{bourbaki}. The collection $\{\theta_i\}_{i=1}^n$ is called a homogeneous system of parameters (hsop) or basic polynomials associated to the pseudoreflection group $G$. Although a hsop is not unique but the degrees $d_i$'s of $\theta_i$'s are unique for $G$ up to order and they satisfy $d_1\cdots d_n=\lvert G\rvert$ and $\sum_{i=1}^n(d_i-1)=r,$ where $r$ is the number of pseudoreflections  in $G,$ see \cite[Corollary 4.3, p.487]{S}.
\begin{ex}
\begin{enumerate}
\item The ring of $\mb Z_n$-invariant complex polynomials is given by $\C[z^n].$ In other words, $\C[z]^{\mb Z_n} = \C[z^n].$ Moreover, the polynomials $1,z,\ldots, z^{n-1}$ form a basis of $\C[z]$ over $\C[z^n],$ see \cite[p.284]{ST}.
\vspace{.1 cm}
    \item Consider the ring of symmetric polynomials $\C[\bl z]^{\mathfrak S_n}$  and one
has $\C[\bl z]^{\mathfrak S_n} = \C[s_1,\ldots,s_n],$ where $s_i$ is the $i$-th elementary symmetric polynomial in $n$ variables. It is not difficult to see that the $n!$ elements
\Bea
z_1^{k_1}\cdots z_{n-1}^{k_{n-1}}, \quad 0 \leq k_i \leq i
\Eea
form a basis of $\C[\bl z]$ as a $\C[\bl z]^{\mathfrak S_n}$ module \cite[p.94]{serre}.
\vspace{.1 cm}
\item For dihedral group $D_{2n}$ of order $2n,$ $\C[z_1,z_2]^{D_{2n}} = \C[z_1z_2,z_1^n+z_2^n].$ 
\end{enumerate}

\end{ex}
 Explicit description of basic polynomials (hsop) associated to various pseudoreflection groups can be found in \cite[Section VI]{R}.

\subsection{Poincar\'e series}
\begin{defn}
Let $V$ be a (complex) vector space. A grading of $V$
is a direct sum decomposition
$$
V = \oplus_{k\geq 0} V_k
$$
such that $\dim (V_i)$ is finite for each $i$. The \Poincare series of the graded vector space $V$, (cf. \cite[p.108]{bourbaki})
denoted by $P_V$, is defined by
$$
P_V(t) = \sum_{k \geq 0} \dim(V_k) t^k.
$$
\end{defn}

If $W = \oplus_{k\geq 0} W_k$ is another graded vector space, define a grading of $V \oplus W$ and $V \otimes W$ by
\begin {align*}
 (V \oplus W)_k &= V_k \oplus W_k,\\
 (V \otimes W)_k &= \sum_{p+q = k}V_p \otimes W_q,
\end {align*}
and thus
\begin {align*}
P_{V \oplus W} &= P_V + P_W,\\
P_{V \otimes W} &= P_V\cdot P_W.
\end {align*}

The following computations derive the \Poincare series of $\C[\bl z]$ and $\C[\bl z]^G$.
%Let $A = \C[\bl z]$ be the ring of polynomial functions on $\C^n$ and let $B = A^G$ be the ring of $G$-invariant
%elements of $A$.
\begin{enumerate}
\item[(a)] Define a grading of $\C[\bl z]$ by letting $\C[\bl z]_k$ be the space of homogeneous polynomials of degree $k.$ One
has $\C[\bl z] \cong \C[z_1] \otimes \cdots \otimes \C[z_n]$ and since
\Bea
P_{\C[z]}(t) = \sum_{k\geq 0}t^k  = \frac{1}{1-t},
\Eea
a repeated application of the formula for tensor product shows that
\Bea
P_{\C[\bl z]}(t) = \frac{1}{(1-t)^n}.
\Eea

\item[(b)] Note that 
%${\C[\bl z]^G}$ is a subspace consisting of homogeneous polynomials of  ${\C[\bl z]},$ that is, 
if ${\C[\bl z]^G_k} = {\C[\bl z]^G} \cap {\C[\bl z]}_k,$ then ${\C[\bl z]^G} = \oplus {\C[\bl z]^G_k}$. Let
$d_1,\ldots, d_n$ be the degrees of $\theta_1,\ldots,\theta_n,$ respectively. Then one has
\Bea
P_{\C[\bl z]^G}(t) = \frac{1}{1-t^{d_1}}\cdots \frac{1}{1-t^{d_n}}.
\Eea
This follows from the equality $P_{\C[\theta_i]}(t) = \sum_{k\geq 0}t^{kd_i}  = (1-t^{d_i})^{-1}$ together
with the isomorphism ${\C[\bl z]^G} \cong \C[\theta_1] \otimes \cdots \otimes \C[\theta_n]$.
\end{enumerate}

\subsection{Proper holomorphic maps}
Let $\Omega_1$ and $\Omega_2$ be domains (connected open sets) in $\mb C^n$. A holomorphic map $\bl f:\Omega_1\to\Omega_2$ is said to be {\it proper} if $\bl  f^{-1}(K)$  is compact in $\Omega_1$ whenever $K\subseteq\Omega_2$ is compact. Clearly, a biholomorphic map is proper. Basic properties of proper holomorphic maps are discussed in \cite[Chapter 15]{R1}.

For two domains $D_1$ and $D_2$ in $\mb C^n$ and an analytic covering map $\bl p : D_1\to D_2,$ a {\it deck transformation} of $\bl p$  is a biholomorphism $\bl h : D_1 \to D_1$ such that $\bl p \circ \bl h =\bl  p.$ The set of all deck transformations of $\bl p$ forms a group. We call it the group of deck transformations of $\bl p$  and denote it by ${\rm Deck}(\bl p).$ It is well-known that a proper holomorphic mapping $\bl f:\Omega_1\to\Omega_2$ is a (unbranched) analytic covering 
\Bea
\bl f:\Omega_1\setminus \bl f^{-1}f( V_{\bl f})\to\Omega_2\setminus \bl f(V_{\bl f}),
\Eea
where $V_{\bl f}:=\{\bl z\in\Omega_1:J_{\bl f}(\bl z)=0\}$ is the {\it branch locus} of $\bl f$ \cite[Subsection 15.1.10]{R1}.
The group ${\rm Deck}(\bl f)$ of this (unbranched) analytic covering is called the group of deck transformations of the proper map $\bl f:\Omega_1\to\Omega_2.$
An analytic covering $\bl p$ is called {\it Galois} if ${\rm Deck}(\bl p)$ acts transitively on the fibre set  $\bl p^{-1}(\{\bl  w\})$ of $\bl  w\in D_2$ for some (and thus for all) $\bl  w \in D_2.$ Equivalently, cardinality of the set $\bl p^{-1}(\{\bl  w\})$ equals to the order of ${\rm Deck}(\bl p).$ In other words, for every $\bl w\in D_2,$ the set $\bl p^{-1}(\{\bl w\})$ is a ${\rm Deck} (\bl p)$-orbit, and vice versa. Briefly, $D_2=D_1/ {\rm Deck} (\bl p).$ Hence an analytic cover $\bl p$ is Galois if and only if $D_1/{\rm Deck}(\bl p) \cong D_2.$ 

\subsection{Hilbert modules}
We  recall several useful definitions following \cite{DP} and \cite{ChenGuo}.
\begin{defn}
A Hilbert space $\mathcal H$ is said to be a {\rm Hilbert module} over an algebra $\m A$ if the map $(f,h) \to f\cdot h,$ $f\in \m A, h\in \mathcal H,$ defines an algebra homomorphism $f \mapsto T_f$ of $\m A$ into $\m L(\m H),$ where $T_f$ is the bounded operator defined by  $T_f h = f \cdot h.$

A closed subspace $\m M$ of $\m H$ is said to be a {\rm submodule} of the Hilbert module $\m H$ if $\m M$ is closed under the action of $\m A.$ If in addition, $T_f^*\m M\subseteq\m M$ for all $f\in\m A,$ then $\m M$ is called a {\rm reducing submodule} of $\m H,$ where $T_f^*$ denotes the adjoint of $T_f.$ A reducing submodule $\m M$ is called \emph{minimal} or \emph{irreducible} if only reducing submodules of $\m M$ are  $\{0\}$ and $\m M$ itself.
\end{defn}
It is easy to see that $\m H$ admits a reducing  submodule if and only if there exists an orthogonal projection $P $ on $\m H$ such that $PT_f=T_fP$ for all $f\in\m A.$ In this case, $\m K=P\m H$ is the corresponding reducing submodule.
\begin{defn}\label{rank}
The rank of a Hilbert module $\mathcal H$ over the algebra $\mathcal A$ is $\inf  |\mathcal F|,$ where $\mathcal F \subseteq \mathcal H$ is any subset with the property that the set 
\Bea
\{r_1 f_1 + \cdots + r_k f_k: f_1, \ldots , f_k \in \mathcal F; r_1, \ldots , r_k \in \mathcal A\}
\Eea
is dense in $\mathcal H$ and $|\mathcal F|$ denotes the cardinality of $\mathcal F$ (cf. \cite[Section 2.3]{ChenGuo}). Such a set $\mathcal F$ is said to be a generating set of the Hilbert module $\mathcal H$ over $\mathcal A$.
\end{defn}
\noindent It is evident that the rank is a unitary invariant for Hilbert modules. Fix $\bl w\in\C^n$. Let $\C_{\bl w}$ denote the one dimensional module over the polynomial ring $\mathbb C[\bl{z}]$ where the module action is given by $p\cdot\lambda = p(\bl w)\lambda, p\in\C[\bl z],\lambda\in\C$ \cite[p.10]{DP}. We state \cite[Lemma 2.10]{BGMS} in the context of a pseudoreflection group $G$ and omit its proof as it is analogous to the case of $\mathfrak S_n.$

\begin{lem}\label{rankcomp}
If $\m H$ is a Hilbert module over $\mb C[\bl z]$ consisting of holomorphic functions defined on some bounded domain $\Omega \subseteq \mb C^n,$ then we have
\begin{enumerate}
\item $\m H \otimes_{\mb C[\bl z]}\mb C_{\bl w}\cong\cap_{p\in\mb C[\bl z]}\ker M^*_{p - p(\bl w)} = \cap_{i=1}^n \ker \big(M_{i}^*-\bar w_i\big);$
\item $\m H\otimes_{\mb C[\bl z]^{G}}\mb C_{\bl w}\cong\cap_{i=1}^n \ker \big(M_{\theta_i}^*-\ov{\theta_i(\bl w)}\big);$
\item for any set of generators $p_1,\ldots,p_t$ of $\m H$ over $\C[\bl z]$ (or $\mb C[\bl z]^{G}$), the vectors 
\Bea
p_1\otimes_{\mb C[\bl z]^G} 1_{\bl w},\ldots, p_t\otimes_{\mb C[\bl z]^G}1_{\bl w}
\Eea
span $\m H \otimes_{\mb C[\bl z]}\mb C_{\bl w}$ (or $\m H\otimes_{\mb C[\bl z]^G}  \mb C_{\bl w}$).
\end{enumerate}
\end{lem}

%Two Hilbert modules $\mathcal H$ and $\tilde{\mathcal H}$ are said to be (unitarily)  equivalent if there exists a unitary module map $U:\mathcal H \to \tilde{\mathcal H},$ that is, $UT_p = \tilde{T_p} U$ for all $ p \in \mathbb C[\bl z].$

%A Hilbert module $\mathcal H$ over $\mathbb C[\bl z]$ equipped with the sup norm on some compact set $X\subset\C^n,$ is said to be bounded if
%$$\|T_p\|\leq C \|p\|_{\infty, X}:= C \sup\{|p(\bl z)| : \bl z \in X\}$$ for some
%positive constant $C$ independent of $p$ and it is said to be contractive if the constant $C$ can be chosen to be $1.$

%enerate the Hilbert module $\mathcal H$ if $\{r_1 f_1 + \cdots + r_k f_k: r_1, \ldots , r_k \in \mb C[\bl z]\}$ is dense in $\mathcal H$. The \emph{rank of a Hilbert module} $\mathcal H$ over the ring $\mb C[\bl z]$  is $\inf  |\mathcal F|,$ where $\mathcal F \subseteq \mathcal H$ is any subset with the property $\{r_1 f_1 + \cdots + r_k f_k: f_1, \ldots , f_k \in \mathcal F; r_1, \ldots , r_k \in \mb C[\bl z]\}$ is dense in $\mathcal H$ and $|\mathcal F|$ denotes the cardinality of $\mathcal F$ (cf. \cite[Section 2.3]{ChenGuo}).
\begin{defn}
Let $\Omega$ be a domain in $\mb C^n$ and $\m H$ be a Hilbert space consisting of holomorphic functions on $\Omega$. A function $K:\Omega \times \Omega \to \m L(\mb C^{m})$ is called a reproducing kernel of $\m H$ if
\begin{enumerate}
\item it is holomorphic in the first variable and anti-holomorphic in the
second,
\item $K(\cdot,\bl w) \bl \xi$ is in $\mathcal H$ for $\bl w\in \Omega$ and $\bl \xi
\in \mathbb C^m$ and
\item $K$ has the reproducing property:
\Bea
\inner{\bl f}{K(\cdot, \bl w) \bl \xi} = \inner{\bl f(\bl w)}{\bl \xi} \mbox{~for~}\bl f \in \m H,\:\bl w\in \Omega,\: \bl \xi\in
\mathbb C^m.
\Eea
\end{enumerate}
A Hilbert space possessing a reproducing kernel is said to be a reproducing kernel Hilbert space.
\end{defn}

\begin{defn}
A Hilbert module $\m H$ over $\mb C[\bl z]$ of $\mb C^m$-valued holomorphic functions on $\Omega\subseteq\mb C^n,$ is said to be an {\it analytic Hilbert module}  if
\begin{enumerate}
\item [(1)] $\mb C[\bl z]\otimes\mb C^m$  is dense in $\m H$ and
\item[(2)] $\m H$ possesses a $\m L(\mb C^m)$-valued reproducing kernel on $\Omega.$
\end{enumerate}
The module action in an analytic Hilbert module is given by pointwise multiplication, that is, ${\mathfrak m}_p(\bl h)(\bl z) = p(\bl z) {\bl h}(\bl z),\, \bl z\in \Omega.$ Note that if the $i$-th component of $\bl h\in \m H$ is $h_i,$ then the $i$-th component of ${\mathfrak m}_p(\bl h)(\bl z)$ is $p(\bl z)h_i(\bl z)$ for $1\leq i\leq m.$  
\end{defn}

%If $\Omega$ is assumed to be polynomially convex, then contractivity  of the analytic Hilbert module is equivalent to saying that the compact set ${\rm clos}(\Omega)$ is a {\it spectral set} for the commuting tuple of multiplication operators $(M_{z_1}, \ldots , M_{z_n}).$

%%%%%%%%%%%%%%

\section{Analytic Chevalley-Shephard-Todd Theorem}
Theorem \ref{theorema} and Theorem \ref{theoremb} below are generalizations of Chevalley-Shephard-Todd Theorem described in Theorem \ref{A} and Theorem \ref{B}, respectively, to the algebra of holomorphic functions on $\mb C^n.$  We call them \emph{analytic Chevalley-Shephard-Todd Theorem} and abbreviate as analytic CST. This is the main result of this section.
%The first of these theorems says that if $\theta_1,\ldots,\theta_n$ are same as stated in Theorem A, then any $G$-invariant holomorphic function is a holomorphic function in $\theta_1,\ldots,\theta_n$.

\begin {thm}\label{theorema}
Let $G$ be a finite group generated by pseudoreflections on $\mb C^n$ and $\{\theta_i\}_{i=1}^n$ be a homogeneous system of parameters associated to it. Then, for a $G$-invariant holomorphic function $f$ on $\C^n,$ there exists a unique holomorphic function $g$ on $\C^n$ such that $f = g\circ\bl{\theta},$  where $\bl{\theta} = (\theta_1,\ldots,\theta_n)$. 
\end {thm}

\begin {thm}
 \label{theoremb}
 Let $G$ be a finite group generated by pseudoreflections on $\mb C^n,\, f$ be a holomorphic function on $\C^n$ and $p_1,\ldots,p_d$ be a basis of $\C[\bl z]$ as a $\C[\bl z]^G$ module, where $d=\lvert G\rvert.$ Then there exist unique $G$-invariant holomorphic functions $f_1,\ldots,f_d$ such that
 \Bea
 f = p_1f_1 + \cdots + p_df_d.
 \Eea
\end {thm}
In other words, $\mathcal O(\C^n)$ is a free module of rank $\lvert G \rvert$ over the ring $\mathcal O(\C^n)^G.$ An intriguing part of this theorem is that even in holomorphic framework, a basis of $\C[\bl z]$ as a module over $\C[\bl z]^G$ continues to work as a basis of $\mathcal O(\C^n)$ as a module over $\mathcal O(\C^n)^G.$
\begin{rem}\label{basis}
A choice of different basis of polynomial ring is related to the chosen basis by a matrix whose determinant is invertible in $\C[\bl z]^G$ and consequently is a non-zero complex number. Thus the result is independent of the basis chosen.
\end{rem}

%\begin{rem}\rm
 %An approach to prove the Theorem above is to prove similar theorem for formal power series and then one deals with the convergence issues. A similar approach was made in \cite[Theorem 1.2]{A} where it was proved that analytic solution to an arbitrary system of analytic equation exists provided a solution in formal power series without constant term exists!  In our case, it is indeed true that the ring $\C[[\bl z]]$ of formal power series is a free module over $\C[[\bl z]]^{G}$, the ring of invariants under the action of $G$, of rank $d$. Using $\C[[\bl z]]^{G} = \C[[\theta_1, \theta_2,\ldots, \theta_n]]$ (following \cite[Theorem 2.a, A.IV.68]{Bourbaki1}) and the fact that $\C[[\bl z]]$ (respectively $\C[[\bl z]]^{G}$) is the completion of the polynomial ring $A$ (respectively $\C[z_1,\ldots,z_n]^G$) in $\mathfrak M$-adic topology, where $\mathfrak M$ is the maximal ideal generated by the indeterminates \cite[Example 1, p. 260]{ZS} and a commutative algebra technique \cite[Proposition 10.13]{AM}, the result for formal power series follows. However, it is not clear, a priori, whether the components of a holomorphic functions in terms of formal power series that we obtain via the decomposition above are necessarily holomorphic, that is, whether they converge on $\C^n$! We rather approach Theorem \ref{theoremb} independently by obtaining the formulae for the components $g_1,\ldots,g_d$.
%\end{rem}

Note that ${\bl\theta}$ is $G$-invariant, by definition, and thus factors through the quotient topological space $\C^n/G$. Let ${\bl\theta}^\prime:\C^n/G\rightarrow \C^n$ be the resulting map. If $\mathfrak q:\C^n\ra \C^n/G$ is the quotient map, then one has $\bl\theta = \bl\theta^\prime\circ \mathfrak q.$ The proof of the next proposition can be found in the literature, see \cite[Theorem 2.1]{BD}, \cite{LT}. Nevertheless, we include it here for the sake of completeness.
%The following result is well known.

\begin {prop}
 \label {bijective}
 ${\bl\theta}'$ is bijective.
\end {prop}
\begin {proof}
 Equivalently, one needs to show that:
\begin{enumerate}
 \item[(i)] $\bl\theta$ is surjective;
 \item[(ii)] The fibers of $\bl\theta$ are precisely the $G$-orbits in $\C^n$.
\end{enumerate}

% \noindent The proofs are standard, but we include them for the sake of completeness.
 \noindent For (i), let $\bl{\mu} = (\mu_1, \ldots, \mu_n) \in \C^n$ and let $\mathfrak{m}$ be the maximal ideal of $\C[\theta_1,\ldots,\theta_n]$ generated by $\theta_1-\mu_1,\ldots,\theta_n-\mu_n$. Note that $\C[\theta_1,\ldots,\theta_n]/\mathfrak{m} \cong \C$ and thus $\C[\bl z]/\mathfrak{m}\C[\bl z]$ is a complex vector space of dimension $\lvert G\rvert,$ in view of Theorem \ref{B}. Thus $\mathfrak{m}\C[\bl z] \neq \C[\bl z]$ and the ideal $\mathfrak{m}\C[\bl z]$ of $\C[\bl z]$ is contained
 in a maximal ideal $\mathfrak{n}$ of $\C[\bl z].$ By the Nullstellensatz, $\mathfrak{n}$  is generated by $ z_1-\lambda_1,\ldots, z_n-\lambda_n$
 for suitable $\lambda_i \in \C.$ Let $\bl{\lambda}= (\lambda_1,\ldots,\lambda_n).$ If $p \in \mathfrak{n},$ then $p(\bl{\lambda}) = 0.$
By taking $p = \theta_i - \mu_i$, we have $\theta_i(\bl{\lambda}) = \mu_i$ for each $i,$ whence ${\bl\theta}(\bl{\lambda}) = \bl{\mu}.$

 For (ii), let $O$ and $O^\prime$ be two distinct orbits of $G$ and $p$ be a polynomial taking value
 0 in $O$ and 1 in $O^\prime$. Then the polynomial
 $$
 \hat p = \prod_{\rho\in G} \rho(p)
 $$
 has the same property and is $G$-invariant. Thus, by Theorem \ref{A}, it is a polynomial in the $\theta_i$'s. Then it follows that there
 exists a $j$ such that $\theta_j({O}) \neq \theta_j({O}^\prime)$. Thus, ${\bl\theta}$ takes $O$ and
 $O^\prime$ to different orbits, whence (ii).
\end {proof}
It follows from a Theorem of H. Cartan \cite[Theorem 7.2]{OD}. that the quotient space $\C^n/G$ can be given the structure of a complex analytic space and that the map
%$\pi$
%${\bl\theta}^\prime$ is then holomorphic %\cite[E. 49 A o.]{kaup}.\cite[Lemma (2.1), p. 257]{ACG}. Therefore,
$\bl\theta^\prime$ is holomorphic.
Further note that $\dim (\C^n/G) = n$ (cf. \cite[E. 49p.]{kaup}). Now we complete the proof of the Theorem \ref{theorema}.

%\begin {cor}
% $\pi'$ is biholomorhic.
%\end {cor}
%\subsection{Proof of theorem \ref{theorema}}\label{a}
\begin{proof}[\textbf {Proof of Theorem  \ref{theorema}}]\label{a}
It is enough to show that  ${\bl\theta}^\prime$ is biholomorhic. It is well-known (cf. \cite[Proposition 46.A.1]{kaup}) that a bijective holomorphic mapping from an equidimensional reduced complex analytic space to a complex manifold of the same dimension is biholomorphic. Since ${\bl\theta}^\prime$ is bijective and $\C^n/G$ is reduced (cf. \cite[p.246]{CH}) of pure dimension $n$, the result follows.
\end{proof}

%\section{Something}
Let $G = \{\rho_1,\ldots,\rho_d\}$ and $\{p_1,\ldots,p_d\}$ be a basis of $\C[\bl z]$ over $\C[\bl z]^G.$ We set
 \bea\label{l}
 \Lambda=\big(\!\!\big(\rho_i(p_j)\big)\!\!\big)_{i,j=1}^d.
 \eea
 We prove several lemmas to facilitate the proof of Theorem \ref{theoremb}.  
 %\begin{rem}[Notations]
  %For the sake of convenience, we denote the ring of polynomials in $n$ variables by $A$ and the ring of $G$-invariant polynomials by $B$ in further references.
 %\end{rem}
 \begin {lem}\label{nonzero}
 $\det \Lambda\not\equiv 0.$ 
\end {lem}

\begin {proof}
 Let $K$ (resp. $L$) be the fraction field of $\C[\bl z]$ (resp. $\C[\bl z]^G$). From \cite[p.110]{bourbaki}, we note that $K$ is a Galois extension of $L$ with Galois group $G$ and that $p_1,\ldots,p_d$ is a basis of $K$ over $L.$ The $(i,j)$-th entry of the matrix $\Lambda^{\tr}\Lambda$ is 
\Bea
\rho_1(p_ip_j) + \cdots + \rho_d(p_ip_j) = T_{K/L}(p_ip_j),
\Eea
where $T_{K/L}:K \rightarrow L$
 denotes the trace function  (cf. \cite[Corollary 8.13]{P}). Thus, it is the matrix
 of the bilinear form $F: K \times K \rightarrow L$, $(f,f^\prime) \mapsto T_{K/L}(ff^\prime)$. This is non-degenerate since for non-zero
 $f$, one has $F(f,1/f) = T_{K/L}(1) = [K:L] \neq 0$. Thus $(\det\Lambda)^2 = \det(\Lambda^{\tr}\Lambda)$ (it is the \emph{discriminant} of the extension $K/L$) is non-zero.
 % \cite[Lemma 12.12 and Proposition 12.13]{P}.
\end {proof}
 
\begin {lem}
 \label {degreedeterminant}
Let $m$ denote the number of pseudoreflections in $G$. Then $\det\Lambda$ is a homogeneous polynomial of degree $\frac{dm}{2}.$
\end {lem}
\begin {proof}
In view of Remark \ref{basis}, choosing a different basis of $\C[\bl z]$ (as a $\C[\bl z]^G$ module) changes $\det\Lambda$ only by a constant and thus we may assume that $p_1,\ldots, p_d$ are homogeneous (cf. Theorem \ref{B}). Let $e_j$ be the degree of $p_j.$ It is clear that $\det\Lambda$ is homogeneous and is of degree $\sum_j e_j.$ Now one has
\Bea
\C[\bl z]^G =\C[\bl \theta] {~~\rm and~~} \C[\bl z] = \oplus_{j=1}^d p_j \C[\bl z]^G,
\Eea
and by considering \Poincare series, we get
\Bea
\frac{1}{(1-t)^n} = \frac{1}{(1-t^{d_1})\cdots(1-t^{d_n})}\cdot (t^{e_1} + \cdots + t^{e_d})
\Eea
where $\deg \theta_i = d_i$ and thus
\begin {equation}
 \label {Poincare}
\sum_j t^{e_j} = \prod_i \frac{1-t^{d_i}}{1-t} = \prod_i ( 1+t + \cdots + t^{d_i-1}).
\end {equation}
We now take derivative with respect to $t$ and evaluate at $t = 1.$ This gives $\sum e_j$ on the left side. Evaluating the
derivative of $1+ t + \cdots + t^{d_i-1}$ at $t = 1$
gives $\frac{d_i(d_i-1)}{2}.$ Thus, the right side becomes
\Bea
\sum_i d_1\cdots\widehat{d_i}\cdots d_n  \frac{d_i(d_i-1)}{2},
\Eea
where $\ \widehat{}\ $ denotes omission. This is equal to
\Bea
d_1\cdots d_n\biggl(\frac{1}{2}\sum_i (d_i-1)\biggr).
\Eea
But it is well-known that $d_1\cdots d_n = d$ \cite[Chapter V, p.115]{bourbaki} and $\sum (d_i-1) = m$ \cite[Chapter V, p.116]{bourbaki}, \cite[Corollary 4.3]{S}
whence the result follows.
\end {proof}

\begin{rem}
It follows from the Equation \eqref{Poincare} that the degrees $e_j$ are independent of the chosen homogeneous basis. Moreover,  if $G$ is a reflection group, then Equation \eqref{Poincare} shows that $e_1,\ldots, e_d$ are the lengths of the elements of $G$, see \cite[Section 3.15]{humphreys}.
\end{rem}

%We call a hyperplane $H$ of $\C^n$ {\em reflecting} if there exists a pseudoreflection in $G$ acting trivially on $H$. 
Let $H$ be a reflecting hyperplane associated to the finite pseudoreflection group $G$. The set of elements of $G$ acting trivially on $H$ forms a cyclic subgroup. To see this, choose a positive hermitian form invariant
under the action of $G$. Then the line $\ell$ orthogonal to $H$ is invariant under the subgroup, thus $H$ embeds in
$\Aut(\ell) = \C\setminus \{0\},$ and hence being a subgroup of the group of $d$-th roots of unity, it is cyclic. If its order is $k$, then the subgroup is of the form $\{1,\rho,\ldots,
\rho^{k-1}\}$ where each $\rho^i$, $1 \leq i < k$ is a pseudoreflection. Let $L$ denote a nonzero homogeneous linear function on $\C^n$ vanishing on $H$ such that $H = \{\bl z\in\C^n: L(\bl z) = 0\}$. It is determined up to a non-zero constant multiplier. 

\begin {lem}\label{auxi}
  Let $N$ be a square matrix with entries in $\C[\bl z]$ and let $\bl v$ be a row of $N$. If $\bl v, \rho(\bl v),\ldots, \rho^{j}(\bl v)$ are
  distinct rows of $N$, then $\det N$ is divisible by $L^{\frac{j(j+1)}{2}}.$
 \end {lem}
 %and the cyclic subgroup of $G$ that fixes $H,$ is in the form $\{\rho^k\}_{k=0}^{j-1}$
 \begin {proof}
 We proceed by induction on $j$. The case $j = 0$ is trivial. Using row reduction, we may replace the rows by
 \Bea
 \bl v-\rho(\bl v), \rho(\bl v) - \rho^2(\bl v), \ldots, \rho^{j-1}(\bl v) - \rho^j(\bl v), \rho^j(\bl v).
 \Eea
 Note that if $p$ is a polynomial, then
 $p-\rho(p)$ is identically zero on $H$ and thus is divisible by $L$ (cf. \cite[p.52]{humphreys}). Let $\bl u = (\bl v-\rho(\bl v))/L$. The rows
 become $L\bl u, L\rho(\bl u), \ldots, L\rho^{j-1}(\bl u), \rho^j(\bl v)$. By induction, the determinant of a matrix
 with rows $\bl u, \rho(\bl u), \ldots, \rho^{j-1}(\bl u)$ is divisible by $L^{\frac{(j-1)j}{2}}$. Thus, $\det N$ is divisible
 by $L^j\cdot L^{\frac{(j-1)j}{2}} = L^{\frac{j(j+1)}{2}}.$
\end {proof}

Let $H_1,\ldots, H_r$ denote the distinct reflecting hyperplanes associated to the group $G$ and let $m_1,\ldots, m_r$ be the orders of the corresponding cyclic subgroups $K_1,\ldots, K_r,$ respectively. Note that $\sum_{i=1}^n (d_i-1) = \sum_{i=1}^r (m_i-1),$ see \cite[p.617]{St}. Thus, one has
\bea\label{power}
m = \sum_{i=1}^r (m_i-1),
\eea
%Let $L_i$ be a nonzero linear function on $\C^n$ vanishing on $H_i$.
where $m$ is the number of pseudoreflections in $G.$ Let $L_i$ be a homogeneous linear function defining $H_i$, that is, $H_i=\{\bl z \in \C^n: L_i(\bl z) = 0\}$. Recall $\Lambda$ from the Equation \eqref{l}. We establish an explicit formula (up to a constant) for $\det \Lambda$.
\begin {prop}
 \label {determinantformula}
 $$\det \Lambda = c\prod_{i=1}^r L_i^{\frac{d(m_i-1)}{2}},$$
 where $c$ is a non-zero constant.
\end {prop}

\begin {proof}
From Lemma \ref{degreedeterminant} and Equation \eqref{power}, it follows that
the degrees of both sides of the equation above is $\frac{dm}{2}.$ Thus, by Lemma \ref{nonzero}, it suffices to check that $\det\Lambda$ is divisible by  $L_i^{\frac{d(m_i-1)}{2}}.$ Let $K_i$ be the subgroup of $G$ fixing $H_i$ and let $\rho_i$ be a generator of
 $K_i.$ Partition $G$ into right cosets of $K_i.$ Fix a right coset $K_i\rho$ for some $\rho \in G.$ Let $\bl u = (p_1,\ldots,p_d)$ and let $\bl v = \rho(\bl u)$. The rows of $\Lambda$ corresponding to $K_i\rho$ are $\bl v, \rho_i(\bl v), \ldots, \rho_i^{m_i-1}(\bl v).$ Thus, by the
Lemma \ref{auxi}, $\det\Lambda$ is divisible by $L_i^{\frac{(m_i-1)m_i}{2}}$. But there are $\frac{d}{m_i}$ right cosets and repeating the argument in the Lemma \ref{auxi} again for each of these, one deduces that $\det\Lambda$ is divisible by $\frac{d}{m_i}$-th power of
$L_i^{\frac{(m_i-1)m_i}{2}}.$
\end {proof}

\begin{rem}\leavevmode\label{reflec}
\begin{enumerate}
\item[(i)] In particular, when $m_i = 2$ and $r = m$, the group $G$ is known as reflection group. In this case, the order $d$ of $G$ is even and
the product of the linear factors reduces to $(L_1\cdots L_m)^{\frac{d}{2}}.$

\item[(ii)] It is a result due to Steinberg \cite[p.118]{bourbaki}, \cite[Lemma, p.616]{St} that the jacobian
$J_{\bl\theta} = \det\big(\!\!\big(\frac{\partial \theta_i}{\partial z_j}\big)\!\!\big)_{i,j=1}^n$ is
a constant multiple of $\prod L_i^{m_i-1}.$ Thus $\det\Lambda$ is a constant multiple of $J_{\bl\theta}^{{d}/{2}}.$ Note that if $d$ is odd, then each $m_i-1$ is even and $J_{\bl\theta}$ is already a square in $\C[\bl z].$
\end{enumerate}
\end{rem}

\begin{proof}[\textbf{Proof of Theorem \ref{theoremb}}]
First we prove the uniqueness. Assume that for a holomorphic function $f$ on $\C^n$, there exist $G$-invariant holomorphic functions $f_1, \ldots, f_d$ such that $f = p_1f_1 + \cdots + p_df_d$. Let $G = \{\rho_1,\ldots,\rho_d\}$. Applying $\rho_i$ to $f$, we have
\Bea
\rho_i(f) = \rho_i(p_1)f_1 + \cdots + \rho_i(p_d)f_d, \mbox{~for~} 1\leq i\leq d.
\Eea
Let $\bl y$ be the column vector $(f_1,\ldots,f_d)^{\tr}$ and let $\bl x = \big(\rho_1(f),\ldots,\rho_d(f)\big)^{\tr}.$ Note that $\Lambda\bl y = \bl x$. Let $\bl y^\prime$ be another
vector such that $\Lambda{\bl y}^\prime = \bl x.$ Then $\Lambda(\bl y - {\bl y}^\prime) = 0$. Left multiplication by the adjoint of $\Lambda$ yields
$\det\Lambda(\bl y - {\bl y}^\prime) = 0.$ By Lemma \ref{nonzero}, it follows that $\bl y = {\bl y}^\prime$ since the ring of holomorphic functions on $\C^n$ is an integral domain. 

%\subsection{Proof of existence of theorem \ref{theoremb}}
To prove existence, we first obtain a formula for the $f_i$, when $f$ is a polynomial. Recall that $\bl x = \big(\rho_1(f),\ldots, \rho_d(f)\big)^{\tr}.$ From Theorem \ref{B}, it follows that there exists $\bl y = (g_1,\ldots,g_d)^{\tr},$ with each $g_i$ $G$-invariant, such that $\Lambda\bl y = \bl x$. Suppose $(\wi{f_1},\ldots, \wi{f_d})^{\tr} = (\adj \Lambda)\bl x$, then $(\det\Lambda) f_i = \wi{f_i},$ that is, $\det\Lambda$ divides $\wi{f_i}$ for each $i.$ Further, if $f$ is homogeneous of degree $e,$ then $f_i$ is homogeneous of degree $e - e_i$ (by comparing the homogeneous part of degree $e$ of the equation $f = \sum_{i = 1}^dp_if_i$, where $\deg p_i = e_i, 1\leq i\leq d$) and thus, $\wi{f_i}$ is homogeneous  polynomials of degree $\frac{dm}{2} + e - e_i$.
%In general, if $g = \sum_{k\geq 0}g_{(k)}$ denotes the power series expansion of $g$ holomorphic on $\C^n$ where $g_{(k)}$ the $k$-th degree homogeneous polynomial in the expansion and $\bl x = \big(\rho_1(g),\ldots, \rho_d(g)\big)^t$, then it follows that:
Suppose now that $f$ is a holomorphic function on $\C^n$ and let $f_{(k)}$ denote the homogeneous component of $f$ of degree $k$ in the power series expansion of $f.$ It follows from the above that:
\begin{enumerate}
\item the $i$-th component $\wi{f_i}$ of $(\adj\Lambda)\bl x$ is an entire function whose homogeneous components are $\wi{f_{(k), i}},\, k\geq 0,$ and

\item\label{inv} $\det \Lambda$ divides $\wi{f_{(k), i}}$ and  $f_{(k), i} = \frac{\wi{f_{(k), i}}}{\det\Lambda}$ is a $G$-invariant polynomial for each $i$ and $k.$
\end{enumerate}
Since $\wi{f_i} = \sum_{k\geq 0}\wi{f_{(k), i}},$ the proof of existence will be complete if we prove that the observations above imply that $\det \Lambda$ divides each $\wi{f_i}$ and the quotients are $G$-invariant.

We show that if $\mathfrak p$ is a product of linear polynomials that divides the homogeneous components of the power series of a holomorphic function $g$ on $\C^n,$ then $\mathfrak p$ divides $g.$ From Proposition \ref{determinantformula}, it then follows that $\det \Lambda$ divides each $\wi{f_i}$. Let $\ell$ be a linear factor of $\mathfrak p$. Using a linear change of coordinates, we may assume $\ell = z_1,$ in which case $g$ will have the power series expansion as $\sum_{i_1\geq 1, i_2,\ldots,i_n\geq 0}a_{i_1\ldots i_n}z_1^{i_1}\ldots z_n^{i_n}.$
 % as each of its homogeneous component
%$$
%f_{(k)}(\bl z) = \sum_{\substack{i_1\geq 1, i_2, \ldots, i_n\geq 0,\\ i_1 +\cdots + i_n = k}}a_{i_1\ldots i_n}z_1^{i_1}\ldots z_n^{i_n}.
%$$
 Since the absolute sum of the terms of the power series
 \Bea
 \sum_{i_1\geq 1, i_2,\ldots,i_n\geq 0}a_{i_1\ldots i_n}z_1^{i_1 - 1}\ldots z_n^{i_n}
 \Eea
is dominated by the absolute sum of the terms of those of $\frac{\partial g}{\partial z_1}$, it defines a holomorphic function, say $h$ on $\C^n$, and hence $g(\bl z) = z_1h(\bl z)$.  The homogeneous components of $g/z_1$ are divisible by $\mathfrak p/z_1$ and one
 proceeds by induction on degree of $\mathfrak p$ to see that $g$ is divisible by $\mathfrak p.$ Finally, as $f_i = \sum_{k\geq 0}{f_{(k), i}}$ where the sum converges absolutely and uniformly on compact subsets of $\C^n$, by the virtue of (\ref{inv}) above each $g_i$ is $G$-invariant.
\end {proof}
\subsection{Generalization to the case of $G$-invariant domains}\label{pseudo}

Let $\Omega$ be an arbitrary $G$-invariant domain in $\C^n$. Note that $\bl\theta(\Omega)$ is a domain in $\C^n$ \cite[Proposition 1, p.556]{Try}.

\subsubsection{\textbf{Generalization of Theorem \ref{theorema}}}\label{g3.1}
Here, we want to show that for a $G$-invariant  holomorphic functions $f$ on $\Omega$, there exists a unique holomorphic function $g$ on $\bl\theta(\Omega)$ such that $f = g\circ\bl\theta$. Defining $\bl\theta'$ as in Section \ref{a}, the same argument shows that $\bl\theta: \Omega/G\ra \bl\theta(\Omega)$ is biholomorphic, whence the result follows as before.

\subsubsection{\textbf{Generalization of Theorem \ref{theoremb}}}
Since an arbitrary holomorphic function may not be represented by a power series globally on an arbitrary $G$-invariant domain $\Omega$, we shall adopt a different approach here. We proceed by proving the following lemma.

\begin{lem}\label{gen1}
Let $f$ be a function on $G$-invariant domain $\Omega$ and $(\tilde f_1, \ldots, \tilde f_d)^{{\tr}}$ be  the column vector $(\adj\Lambda)\bl x,$ where  $\bl x = \big(\rho_1(f),\ldots, \rho_d(f)\big)^{{\tr}}.$ Suppose that there exist functions $h_j:\Omega\ra\mathbb C,\, 1\leq j\leq d,$ such that $\tilde f_j = \big(\det\Lambda\big)h_j$ on $\Omega$. Then, $h_j$ is $G$-invariant for each $j,\,1\leq j\leq d.$
\end{lem}

\begin{proof}
Note that $\tilde f_j = \det \Lambda_j$, where $\Lambda_j$ is a $d\times d$ matrix obtained by replacing the $j$-th column of the matrix $\Lambda$ by the column vector $\bl x,$ that is,
\begin{equation}\label{matrix}
\Lambda_j = \left(
        \begin{array}{ccccccc}
          \rho_1(p_1) & \ldots & \rho_1(p_{j-1}) & \rho_1(f) & \rho_1(p_{j+1}) & \ldots & \rho_1(p_d) \\
          \rho_2(p_1) & \ldots & \rho_2(p_{j-1}) & \rho_2(f) & \rho_2(p_{j+1}) & \ldots & \rho_2(p_d) \\
          \vdots & {} & \vdots & \vdots & \vdots & {} & \vdots \\
          \rho_d(p_1) & \ldots & \rho_d(p_{j-1}) & \rho_d(f) & \rho_d(p_{j+1}) & \ldots & \rho_d(p_d) \\
        \end{array}
      \right).
\end{equation}
Clearly if $f = p_j$, then $\Lambda_j = \Lambda.$ Fix a $k$ such that $1\leq k\leq d$ and let $\tilde\rho_i = \rho_k\rho_i$ for $1\leq i\leq d.$ Then $\tilde f_j(\rho_k^{-1}\cdot\bl z) = \det \Lambda_j(\rho_k^{-1}\cdot\bl z)= \det {\widetilde \Lambda}_j(\bl z)$, where
\Bea
\widetilde \Lambda_j = \left(
        \begin{array}{ccccccc}
          \tilde\rho_1(p_1) & \ldots & \tilde\rho_1(p_{j-1}) & \tilde\rho_1(f) & \tilde\rho_1(p_{j+1}) & \ldots & \tilde\rho_1(p_d) \\
          \tilde\rho_2(p_1) & \ldots & \tilde\rho_2(p_{j-1}) & \tilde\rho_2(f) & \tilde\rho_2(p_{j+1}) & \ldots & \tilde\rho_2(p_d) \\
          \vdots & {} & \vdots & \vdots & \vdots & {} & \vdots \\
          \tilde\rho_d(p_1) & \ldots & \tilde\rho_d(p_{j-1}) & \tilde\rho_d(f) & \tilde\rho_d(p_{j+1}) & \ldots & \tilde\rho_d(p_d) \\
        \end{array}
      \right).
      \Eea
 is obtained by permuting the rows of the matrix $\Lambda_j$ as $\{\tilde\rho_1,\ldots,\tilde\rho_d\}$ is a permutation of the elements $\rho_1,\ldots,\rho_d$ of the group $G.$ Therefore $\det \widetilde \Lambda_j$ and $\det \Lambda_j$ are equal except a sign, that is, $\tilde f_j(\rho_k^{-1}\cdot\bl z)$ differs from $\tilde f_j(\bl z)$ by the same sign for all $j,\, 1\leq j\leq d$ and for all functions $f$ on $\Omega$. Thus $(\det \Lambda)(\rho_k^{-1}\cdot\bl z)$ also differs from $\det \Lambda(\bl z)$ by the same sign. Hence $h_j(\rho_k^{-1}\cdot\bl z) = h_j(\bl z)$ for all $\bl z\in \Omega$ and for each $j,\,1\leq j\leq d.$ This completes the proof as $k$ was chosen arbitrarily.
\end{proof}

Let $\mathcal O(\Omega)$ denote the ring of all holomorphic functions from a domain $\Omega\subset \C^n$ to $\C$ and $\mathcal O(\Omega)^G$ denote the ring of $G$-invariant holomorphic functions. We restate the generalisation of Theorem \ref{theoremb} to $G$-invariant domains as follows.
\begin{thm}\label{theoremc}
%Let $G$ be a finite pseudoreflection group acting on $\C^n$ and $\Omega$ be a $G$-invariant pseudo-convex domains in $\C^n$. Then

If $\Omega$ is a $G$-invariant domain in $\C^n$, then $\mathcal O(\Omega)$ is a finitely generated free module over $\mathcal O(\Omega)^G$ of rank $\lvert G\rvert.$
\end{thm}
\begin{proof}
If $\{p_1, \ldots, p_d\}$ is a basis of $\C[\bl z]$ as a $\C[\bl z]^G$ module, as in Theorem \ref{theoremb}, we show that for each $f\in\mathcal O(\Omega)$, there exist unique $G$-invariant holomorphic functions $f_1,\ldots,f_d$ on $\Omega$ such that $f= p_1f_1 + \cdots + p_df_d.$ The proof remains the same as the proof of Theorem \ref{theoremb} except that now we need to prove that the $f_i$'s exist as $G$-invariant holomorphic function on $\Omega$. Lemma \ref{gen1} shows that it is enough to find a global solution to the system of equations $(\det\Lambda)\bl y = (\adj\Lambda)\bl x,$ where $\bl x = \big(\rho_1(f),\ldots, \rho_d(f)\big)^{{\tr}}.$ First we make the following elementary observation. Let $H$ be the reflecting hyperplane fixed by a pseudoreflection $\rho\in G.$ Let $g\in\mathcal O(\Omega)$ and let $\ell$ be the linear factor of $\det\Lambda$ corresponding to $H.$ Then we claim the $\ell$ divides $g - \rho(g)$ in $\mathcal O(\Omega)$. If $\bl z\notin H\cap \Omega$, then $\ell(\bl z)\neq 0$ and this is true in a neighbourhood of $\bl z.$ On the other hand, if $H\cap \Omega$ is non-empty and $\bl z\in H\cap \Omega,$ then it follows from  Weierstrass Division Theorem \cite[p.11]{GH} that $\ell$ divides $g - \rho(g)$ in a neighbourhood of $\bl z$ since $\ell$ is irreducible. Therefore the claim follows as the holomorphic functions which are obtained locally patch to give a global holomorphic function, since both $g - \rho(g)$ and $\ell$ are defined on all of $\Omega$. Using this result, one can apply the proof of Proposition \ref{determinantformula} to $\det \Lambda_i$ to deduce that $\det\Lambda$ divides $\det \Lambda_i.$ As is observed in the proof of Lemma \ref{gen1} that $(\adj\Lambda)\bl x = \big(\det\Lambda_1, \ldots, \det\Lambda_d\big)^{{\tr}},$ it follows that $(\adj\Lambda)\bl x$ is divisible by $\det\Lambda$ which completes the proof.
\end{proof}

%\begin{rem}
%\textcolor{blue}{An inequality on the rank is obtained in the setup of complex analytic or linear algebraic Lie groups in the paper \cite[Section 3.5]{MV} where one prove that the corresponding module is finitely generated. However, in our our case, we have obtained the exact analogue of Chevalley-Shephard-Todd theorem by showing that it is freely finitely generated with same rank and we will see it's use through out this paper.}
%\end{rem}

\section{Decomposition of analytic Hilbert module}\label{decanahil}
Let $G$ be a finite group, $\m H$ be a Hilbert space and  $U: G \ra \mathcal U(\mathcal H)$ be a unitary representation, where $\m U(\m H)$ denotes the group of unitary operators on $\m H.$ It is known that for a function $\chi:G \ra \C$ satisfying $\chi(\sigma^{-1}) = \ov{\chi(\sigma)}$,  the linear map $T^\chi$ on $\mathcal H$ defined by
\Bea
T^\chi := \sum_{\sigma\in G} \ov{\chi(\sigma)}U(\sigma)
\Eea
is self-adjoint \cite[Lemma 3.1]{BGMS}. 
%Suppose $\w{G}$ is the set of equivalence classes of irreducible representations where each class consisting of unitary equivalent irreducible representations.
Since, for  $\varrho \in \w{G},$ the character $\chi_{\varrho}$ of $\varrho$ satisfies ${\chi_{\varrho}(\sigma^{-1})} = \overline{{\chi_{\varrho}(\sigma)}}$ for $\sigma \in G$, the linear map $\mathbb P_{\varrho} : \m H\rightarrow \m H$ defined by the formula
 \begin{eqnarray*}
   \mb P_{\varrho}=\frac{\deg\varrho}{\lvert G \rvert}\sum_{\sigma\in G}{\chi_{\varrho}(\sigma^{-1})}U(\sigma),
   \end{eqnarray*}
    is self-adjoint.

If $\deg\varrho> 1$ for some $\varrho \in \widehat{G},$ we can further split $\mathbb{P}_{\varrho}.$ Fix $\deg\varrho = m >1.$ Suppose $\bl \pi_{\varrho}$ is a unitary representation of $G$ that belongs to the equivalence class $\varrho$ in $\widehat{G}.$ Then, for some choice of an orthonormal basis of $\C^{m},$ the matrix representation of $\bl\pi_{\varrho}$ is given by 
\Bea
\bl\pi_{\varrho}(\sigma) = \big (\!\!\big (\bl\pi_{\varrho}^{ij}(\sigma)\big )\!\!\big )_{i,j = 1}^{m} \in \C^{m\times m}.
\Eea
We define linear maps  $\mathbb{P}_{\varrho}^{ij}:\m H \to \m H \textnormal{~for~} i,j = 1,2,\ldots,m, $ by 
   \bea\label{finer}
   \mathbb{P}_{\varrho}^{ij} = \frac{\deg\varrho}{\lvert G \rvert}\sum_{\sigma\in G} \bl\pi_{\varrho}^{ji}(\sigma^{-1})U(\sigma).
   \eea
Each $\mathbb{P}_{\varrho}^{ii}$ is self-adjoint as $\bl \pi_{\varrho}$ is a unitary representation implies ${\bl \pi}^{ii}_{\varrho}(\sigma^{-1}) = \overline{{\bl \pi}^{ii}_{\varrho}(\sigma)}$. Note that the definition $\mb P_{\varrho}^{ii}$ is basis dependent while the same of $\mb P_{\varrho}$ is canonical. Since for all $\sigma \in G,$ $\chi_{\varrho}(\sigma^{-1}) = \textnormal{trace}\big({\bl \pi}_{\varrho}(\sigma^{-1})\big) = \sum_{i=1}^{m} \bl\pi_{\varrho}^{ii}(\sigma^{-1}),$ it follows that 
   \bea\label{rhoiisum}
   \mathbb{P}_{\varrho}=\frac{\deg\varrho}{\lvert G \rvert}\sum_{\sigma\in G}\Big(\sum_{i=1}^{m} \bl\pi_{\varrho}^{ii}(\sigma^{-1})\Big) U(\sigma) = \frac{\deg\varrho}{\lvert G \rvert}\sum_{i=1}^{m}\sum_{\sigma\in G}\bl\pi_{\varrho}^{ii}(\sigma^{-1})U(\sigma)=\sum_{i=1}^{\deg\varrho} \mathbb P_{\varrho}^{ii}. %\mathbb{P}_{\varrho}^{11}+\mathbb{P}_{\varrho}^{22}+\cdots+\mathbb{P}_{\varrho}^{mm}
   \eea
   %where, $\chi_{\varrho}(1) = m.$ 
   %Therefore, let us record the following result. That is
   %\bea
   %\mathbb{P}_{\varrho}=\sum_{i=1}^{\chi_{\varrho}(1)} \mathbb P_{\varrho}^{ii}.
   %\eea
 The following orthogonality relations \cite[Proposition 2.9]{KS} play a central role in this section.
\bea\label{or}
\sum_{\sigma\in G}\bl\pi_{\varrho}^{ij}(\sigma^{-1})\bl\pi_{\varrho^\i}^{\ell m}(\sigma) =
\frac{\lvert G \rvert} {{\rm deg}\,\varrho}\delta_{\varrho\varrho^\i}\delta_{im}\delta_{j\ell},
\eea
where $\delta$ is the Kronecker symbol.
The following Proposition and its Corollary can be proved similarly as in \cite[Section 3]{BGMS} by adapting the properties of projection maps given in \cite[p.162]{KS}. Thus, we omit the proofs here.
\begin{prop}\label{proj1}
For $\varrho, \varrho^\i \in \widehat{G},\, 1\leq i, j\leq {\rm deg}\,\varrho$ and $1\leq \ell,m\leq {\rm deg}\,\varrho^\i,$ 
\Bea
\mathbb{P}_{\varrho}^{ij}\mathbb{P}_{\varrho^\i}^{\ell m} = \delta_{\varrho\varrho^\i}\delta_{j\ell}\mathbb{P}_{\varrho}^{im}.
\Eea
\end{prop}

\begin{cor}\label{rhoortho}
\leavevmode
\begin{enumerate}
\item[1.] $\mb P_{\varrho} \mb P_{\varrho^\i} = \delta_{\varrho \varrho^\i} \mb P_{\varrho}$ for all $\varrho, \varrho^\i \in \widehat{G}.$
\item[2.] $\mb P_{\varrho}$ and $\mb P_\varrho^{ii}$ are orthogonal projections for $\varrho \in \widehat{G}$ and $1\leq i \leq \deg \varrho.$ 
\item[3.] $\displaystyle\sum_{\varrho \in \widehat{G}}\sum_{i = 1}^{\deg\varrho} \mathbb{P}_{\varrho}^{ii} = \sum_{\varrho\in \widehat{G}} \mathbb{P}_{\varrho} = {\mr {id}}_{\m H}.$
\end{enumerate}

\end{cor}
In view of the corollary above, we obtain  the following orthogonal direct sum decomposition of the Hilbert space $\m H:$ 
%This is the projection formula used extensively earlier in \cite[Equation (3.2)]{BS}. From  Corollary \ref{decomp}, it follows that the subspace $\mathbb{P}_{\varrho}^{ii}\big(\mb A^{(\l)}(\mb D^n)\big)$ of $\mb A^{(\l)}(\mb D^n)$ is a reproducing kernel Hilbert space for each $\varrho\vdash n$ and $1\leq i\leq \chi_{\varrho}(1).$
%and (b) $\mb A^{(\l)}(\mb D^n) = \bigoplus_{\varrho\vdash n}\bigoplus_{i = 1}^{\chi_{\varrho}(1)}\mathbb{P}_{\varrho}^{ii}\big(\mb A^{(\l)}(\mb D^n)\big)$.
%\bea
%\mb A^{(\l)}(\mb D^n) = \bigoplus_{\varrho\vdash n}\bigoplus_{i = 1}^{\chi_{\varrho}(1)}\mathbb{P}_{\varrho}^{ii}\big(\mb A^{(\l)}(\mb D^n)\big).
%\eea
%From Equation \eqref{rel} and Proposition \ref{proj1}, we  have
%\bea\label{new}\mathbb{P}_{\varrho}\big(\mb A^{(\l)}(\mb D^n)\big) = \bigoplus_{i = 1}^{\chi_{\varrho}(1)}\mathbb{P}_{\varrho}^{ii}\big(\mb A^{(\l)}(\mb D^n)\big)\eea
%and consequently, using Corollary \ref{decomp}, we obtain a finer decomposition of $\mb A^{(\l)}(\mb D^n):$
\bea\label{decomp}
\m H = \bigoplus_{\varrho\in \widehat{  G}}\mathbb{P}_{\varrho}\m H
=\bigoplus_{\varrho\in \widehat{  G}}\bigoplus_{i = 1}^{\deg \varrho}\mathbb{P}_{\varrho}^{ii}\m H. 
\eea

%We are yet to prove whether each of these projections are non-trivial. We establish this fact  for finite pseudoreflection groups in Chapter 3.

\begin{lem}\label{adjoint}
For $\varrho\in \hat G$ and $1\leq i, j\leq \deg\varrho$,
$
{\mathbb{P}_{\varrho}^{ij}}^* \mathbb{P}_{\varrho}^{ij} = \mathbb{P}_{\varrho}^{jj}.
$
\end{lem}
\begin{proof}
For $\bl z,\bl w\in\Omega,$ we have
\bea
\nonumber\langle{\mb{P}_{\varrho}^{ij}}^*\mb{P}_{\varrho}^{ij}K_{\bl z},K_{\bl w}\rangle  &=&\langle \mb{P}_{\varrho}^{ij}K_{\bl z},{\mb{P}_{\varrho}^{ij}}K_{\bl w}\rangle\\
&=&\nonumber\frac{({\rm deg}\,\varrho)^2}{\lvert G \rvert^2}\sum_{\sigma\in G} \sum_{\tau \in G} \bl\pi_{\varrho}^{ji}(\sigma^{-1}) \overline{\bl\pi_{\varrho}^{ji}(\tau^{-1})} \langle U(\sigma) K_{\bl z}, U(\tau) K_{\bl w} \rangle \\
&=&\label{equa}\frac{({\rm deg}\,\varrho)^2}{\lvert G \rvert^2}\sum_{\sigma\in G} \sum_{\tau \in G} \bl\pi_{\varrho}^{ji}(\sigma^{-1}) \bl\pi_{\varrho}^{ij}(\tau) \langle U(\tau^{-1}\sigma) K_{\bl z}, K_{\bl w} \rangle.
\eea
Let $\gamma =  \tau^{-1}{\sigma}$. Then $\tau = \sigma \gamma^{-1}$ and
$$
\bl\pi_{\varrho}^{ij}(\sigma\gamma^{-1}) = (\bl\pi_{\varrho}(\sigma\gamma^{-1}) )_{ij} = (\bl\pi_{\varrho}(\sigma)\bl\pi_{\varrho}(\gamma^{-1}) )_{ij} = \sum_{k = 1}^{\deg \rho} \bl\pi_{\varrho}^{ik}(\sigma)\bl\pi_{\varrho}^{kj}(\gamma^{-1}).
$$
Moreover,
\bea \nonumber \sum_{\sigma\in G} \bl\pi_{\varrho}^{ji}(\sigma^{-1}) \bl\pi_{\varrho}^{ij}(\tau) &=& \sum_{\sigma\in G} \bl\pi_{\varrho}^{ji}(\sigma^{-1}) \bl\pi_{\varrho}^{ij}(\sigma\gamma^{-1})\\
&=&\nonumber \sum_{\sigma\in G} \bl\pi_{\varrho}^{ji}(\sigma^{-1}) \sum_{k = 1}^{{\rm deg}\,\varrho} \bl\pi_{\varrho}^{ik}(\sigma)\bl\pi_{\varrho}^{kj}(\gamma^{-1}) \\
&=& \nonumber\sum_{k = 1}^{{\rm deg}\,\varrho} \bl\pi_{\varrho}^{kj}(\gamma^{-1}) \sum_{\sigma\in G} \bl\pi_{\varrho}^{ji}(\sigma^{-1})\bl\pi_{\varrho}^{ik}(\sigma)\\
&=&\label{equa1} \sum_{k = 1}^{\deg\varrho} \bl\pi_{\varrho}^{kj}(\gamma^{-1}) \frac{\lvert G \rvert}{\deg\varrho} \delta_{jk} = \frac{\lvert G \rvert}{\deg\varrho} \bl\pi_{\varrho}^{jj}(\gamma^{-1}),
\eea
where the penultimate equality follows from Equation \eqref{or}. Thus from Equation \eqref{equa} we have
\Bea
\langle{\mb{P}_{\varrho}^{ij}}^*\mb{P}_{\varrho}^{ij}K_{\bl z},K_{\bl w}\rangle
&=&\frac{({\rm deg}\,\varrho)^2}{\lvert G \rvert^2}\sum_{\sigma\in G} \sum_{\tau \in G} \bl\pi_{\varrho}^{ji}(\sigma^{-1}) \bl\pi_{\varrho}^{ij}(\tau) \langle U(\tau^{-1}\sigma) K_{\bl z}, K_{\bl w} \rangle\\
&=&\frac{({\rm deg}\,\varrho)^2}{\lvert G \rvert^2} \sum_{\gamma \in G}\sum_{\sigma\in G} \bl\pi_{\varrho}^{ji}(\sigma^{-1}) \bl\pi_{\varrho}^{ij}(\sigma \gamma^{-1}) \langle U(\gamma) K_{\bl z}, K_{\bl w} \rangle\\
&=&\frac{({\rm deg}\,\varrho)}{\lvert G \rvert} \sum_{\gamma \in G} \bl\pi_{\varrho}^{jj}(\gamma^{-1}) \langle U(\gamma) K_{\bl z}, K_{\bl w} \rangle\\
%&=&\frac{{\chi_{\varrho}(1)}^2}{\lvert G \rvert^2}\sum_{\sigma\in G}\bl\pi_{\varrho}^{ji}(\sigma^{-1})\big( \sum_{\gamma \in G}\bl\pi_{\varrho}^{ij}(\sigma \gamma)K_{\bl z}( \gamma \cdot w)\big)\\
%&=&\frac{{\chi_{\varrho}(1)}^2}{\lvert G \rvert^2}\sum_{\sigma\in G}\bl\pi_{\varrho}^{ji}(\sigma^{-1})\big( \sum_{\gamma \in G}\big(\sum_{k=1}^{\chi_{\varrho}(1)}\bl\pi_{\varrho}^{ik}(\sigma)\bl\pi_{\varrho}^{kj}(\gamma) \big)K_{\bl z}( \gamma \cdot w)\big)\\
%&=&\frac{{\chi_{\varrho}(1)}^2}{\lvert G \rvert^2}\sum_{\gamma\in G}\Big(\bl\pi_{\varrho}^{1j}(\gamma)\big( \sum_{\sigma \in G}\bl\pi_{\varrho}^{ji}(\sigma^{-1})\bl\pi_{\varrho}^{i1}(\sigma)\big)\\
%&&\hspace{0.5cm} +\bl\pi_{\varrho}^{2j}(\gamma)\big( \sum_{\sigma \in G}\bl\pi_{\varrho}^{ji}(\sigma^{-1})\bl\pi_{\varrho}^{i2}(\sigma)\big)+\dots
%+\bl\pi_{\varrho}^{\chi_{\varrho}(1)j}(\gamma)\big( \sum_{\sigma \in G}\bl\pi_{\varrho}^{ji}(\sigma^{-1})\bl\pi_{\varrho}^{i\chi_{\varrho}(1)}(\sigma)\big) \Big)K_{\bl z}( \gamma \cdot w)\\
%&=&\frac{\chi_{\varrho}(1)}{\lvert G \rvert}\sum_{\gamma \in G}\bl\pi_{\varrho}^{jj}(\gamma)K_{\bl z}( \gamma \cdot w)\\
&=&\langle \mb{P}_{\varrho}^{jj}K_{\bl z},K_{\bl w}\rangle,
\Eea
where the penultimate equality follows from Equation \eqref{equa1} and the last equality is immediate from the Equation \eqref{finer}. Since $\{K_{\bl z} : \bl z \in \Omega \}$ forms a total set of $\mathcal H,$ the proof is complete.
\end{proof}

%The following lemma is going to play an important role in the next chapters.
\begin{lem}\label{uni}
The restriction of $\mb{P}_{\varrho}^{ij }$ to $ \mb{P}_{\varrho}^{jj}\mathcal H$ is an isometry with range $\mb{P}_{\varrho}^{ii } \mathcal H.$ That is, $\mb{P}_{\varrho}^{ij }: \mb{P}_{\varrho}^{jj}\mathcal H\to  \mb{P}_{\varrho}^{ii } \mathcal H $ is a unitary operator.
\end{lem}
\begin{proof}

We prove the lemma by showing that ${\mb{P}_{\varrho}^{ij}} \lvert_{\mb{P}_{\varrho}^{jj}\mathcal H}$ is both isometry and surjective. Note that range $\mb{P}_{\varrho}^{ii } \mathcal H$ is a closed subspace and by the Lemma above, it is an isometry on $\mb{P}_{\varrho}^{jj} \mathcal H$. From Proposition \ref{proj1}, we have
$$
\mathbb{P}_\varrho^{ij}\mathbb{P}_\varrho^{jj} = \mathbb{P}_\varrho^{ij} = \mathbb{P}_\varrho^{ii}\mathbb{P}_\varrho^{ij},
$$ and it then follows that $\mathbb{P}_\varrho^{ij}\mathbb{P}_\varrho^{jj}\m  H \subseteq \mathbb{P}_\varrho^{ii}\m H$. Also for $f\in \m H,$
$$
\mathbb{P}_\varrho^{ii}f = \mathbb{P}_\varrho^{ij}\mathbb{P}_\varrho^{ji}f =  \mathbb{P}_\varrho^{ij}\mathbb{P}_\varrho^{jj}\mathbb{P}_\varrho^{ji}f
$$
and thus $\mathbb{P}_\varrho^{ii}\m H\subseteq \mathbb{P}_\varrho^{ij}\mathbb{P}_\varrho^{jj}\m H .$ Consequently, $\mathbb{P}_\varrho^{ij}$ is a surjective map from $\mathbb{P}_\varrho^{jj}\m H$ onto $\mathbb{P}_\varrho^{ii}\m H$.
 %${\mb{P}_{\varrho}^{ij}} \lvert_{\mb{P}_{\varrho}^{jj}\mathcal H}$ 

\end{proof}

%The first of the two equalities was obtained in  \cite[p. 6237 - 6238]{BS}, see also \cite[p.  2368]{MSZ}.
%Here, combining the decomposition of the projection $\mathbb{P}_{\varrho}$ obtained in \eqref{rel} with Corollary \ref{decomp} rise to a finer decomposition of the Hilbert module $\mb A^{(\l)}(\mb D^n).$
%It is natural to ask whether each of the projections $\mathbb{P}_{\varrho}$  is nontrivial. 
%To see $\mathbb{P}_{\varrho}$ is non-trivial, we are going to use the following well known result which is analogous to the fact that the polynomial ring $\C[\bl z]$ is a finitely generated free module over $\C[\bl z]^{{\mathfrak S}_n}$ of rank $\lvert G \rvert$.
%\begin{description}
%\end{rem}
\section{Reducing submodules of multiplication operator by proper holomorphic mappings}\label{redsub}
Let $G$ be a finite group, $\Omega \subseteq \mb{C}^n$ be a $G$-invariant domain and $\mathcal H$ be an analytic Hilbert module with $G$-invariant reproducing kernel $K,$ that is, $K$ satisfies
\Bea
K(\sigma \cdot \bl z, \sigma \cdot \bl w) = K(\bl z, \bl w) \mbox{~for ~all~}\bl z,\,\bl w\in\Omega \text{~and~} \sigma \in G.
\Eea
%we define a unitary representation $R$ of $G$ that accords with above discussion. We 
In the next lemma, we show that there is an equivalent way of formulating the $G$-invariance of the reproducing kernel in terms of the map $R:\sigma \mapsto R_\sigma$ defined by $\big(R_\sigma f \big)(\bl z) = f(\sigma^{-1} \cdot \bl z)$ for all $\sigma \in G,\, f\in \m H$.

%We denote this representation by $R_{\sigma}.$ 
\begin{lem}\label{iso}
Let $\Omega\subseteq\mb C^n$ be a bounded $G$-invariant domain and $\m H$ be an analytic Hilbert module with reproducing kernel $K.$ Then $K$ is $G$-invariant if and only if
\begin{enumerate}
    \item[1.] $R_\sigma : \mathcal H \to \mathcal H$ is well defined and,
    \item[2.] the map $R: \sigma\mapsto R_\sigma$ is a unitary map on $\mathcal H.$
\end{enumerate}
  \end{lem}

\begin{proof}
The proof that $R_\sigma$ is well defined is same as in \cite[p.765]{BGMS}. For the rest of that proof, one could just follow the proof of \cite[Lemma 3.5]{BGMS} for a group $G$ with the observation that:
\Bea
\langle K_{\sigma\cdot\bl w}, K_{\sigma\cdot\bl w'}\rangle  = K(\bl w', \bl w)
\mbox{~if~ and~ only~ if~}
\langle R_\sigma K_{\bl w}, R_\sigma K_{\bl w'}\rangle = \langle K_{\bl w}, K_{\bl w'}\rangle.
\Eea
\end{proof}

%We rewrite the formulae for projections $\mb P_{\varrho}$ and $\mb P_{\varrho}^{ii}$ for the representation $R_{\sigma}.$
We further note that the map $R$ is a unitary representation of $G$ on $\mathcal H$ usually well known as the left regular representation (see \cite[Lemma 3.5]{BGMS}). Taking $U = R$, the formulae for $\mb P_{\varrho}$ and $\mb P_\varrho^{ij}$ turn out to be
\bea\label{od}
(\mb P_{\varrho}f)(\bl z) &=&  \frac{\deg\varrho}{\lvert G \rvert}\sum_{\sigma\in G} \chi_{\varrho}({\sigma}^{-1})f(\sigma^{-1}\cdot\bl z)
\eea
and
\bea\label{ij}
(\mb P_{\varrho}^{ij}f)(\bl z) &=& \frac{\deg\varrho}{\lvert G \rvert}\sum_{\sigma\in G} \bl\pi_{\varrho}^{ji}({\sigma}^{-1})f(\sigma^{-1}\cdot\bl z).
\eea
Let $\Omega$ be a bounded $G$-invariant domain in $\mb C^n$ and $\bl \phi := (\phi_1,\ldots,\phi_n): \Omega \rightarrow \bl \phi(\Omega)$ be a proper holomorphic map. We also assume that it factors through automorphisms 
%a (branched) covering map which has non-trivial group of deck transformations, 
$G\subseteq \mr{Aut}(\Omega).$ Consequently, $\bl \phi$ is $G$-invariant, that is,  $\bl \phi(\sigma\cdot\bl z) = \bl \phi(\bl z)$ for all $\sigma \in G, \bl z\in\Omega.$ For every $i=1,\ldots,n,$ the multiplication operator is defined by $M_{\phi_i} : \m{H} \rightarrow \m{H}, $ where 
\Bea
(M_{\phi_i} f)(\bl z)= \phi_{i}(\bl z) f(\bl z)  \text{~for all~} f \in \m{H}, \bl z \in \Omega.
\Eea
Put  $\mathbf M_{\bl \phi} = (M_{\phi_1},\ldots,M_{\phi_n}).$ We also assume that $\mathbf M_{\bl \phi}$ is bounded. 

\begin{prop}\label{rhojredsub}
Suppose that $\Omega$ is a bounded and $G$-invariant domain in $\mb C^n$  and $\bl\phi:\Omega\to\bl\phi(\Omega)$ is a proper holomorphic mapping which factors through automorphisms $G\subseteq \mr{Aut}(\Omega).$ If $\m H$ is an analytic Hilbert module with $G$-invariant kernel, then $M_{\phi_k}\mathbb P_{\varrho}^{ij}=\mathbb P_{\varrho}^{ij}M_{\phi_k}$ for    $\varrho \in \widehat{G}, 1\leq i, j\leq\deg\varrho$ and $k=1,\ldots, n.  $
\end{prop}
\begin{proof}
For $f\in\m H,$ from Equation \eqref{ij}, we have

\bea\label{reducing}
\nonumber \big((\mathbb P_{\varrho}^{ij} M_{\phi_k})f\big)(\bl z)  &=& \frac{\deg\varrho}{\lvert G \rvert}\sum_{\sigma \in G}\bl\pi_{\varrho}^{ji}(\sigma^{-1}) \phi_{k}(\sigma^{-1} \cdot \bl z ) f(\sigma^{-1} \cdot \bl z )\\
&=& \frac{\deg\varrho}{\lvert G \rvert}\sum_{\sigma \in G}\bl\pi_{\varrho}^{ji}(\sigma^{-1}) \phi_{k}(\bl z ) f(\sigma^{-1} \cdot \bl z )\\
\nonumber &=& \phi_k(\bl z ) \frac{\deg\varrho}{\lvert G \rvert}\sum_{\sigma \in G}\bl\pi_{\varrho}^{ji}(\sigma^{-1})f(\sigma^{-1} \cdot \bl z )\\
\nonumber &=& \big((M_{\phi_k} \mathbb P_{\varrho}^{ij})f\big)(\bl z).
\eea
We use $G$-invariance of $\bl \phi$ in \eqref{reducing}. This completes the proof.
\end{proof}
%In a similar manner, one can show that $M_{\theta_j} \mathbb P_{\varrho}^{ii} = \mathbb P_{\varrho}^{ii} M_{\theta_j} $ for all $j = 1, \ldots, n.$
%The following corollary is immediate from Proposition \ref{rhojredsub} and 
By Equation \eqref{rhoiisum}, $\mathbb{P}_{\varrho}=\sum_{i=1}^{\deg\varrho} \mathbb P_{\varrho}^{ii}.$ Therefore, Proposition \ref{rhojredsub} gives that $M_{\phi_k}\mathbb P_{\varrho}=\mathbb P_{\varrho}M_{\phi_k}$ for each $\varrho$. In short, the orthogonal projections $\mathbb P_\varrho$ and $\mathbb P_{\varrho}^{ii}$ commute with each $M_{\phi_k}$ for $k=1,\ldots,n.$ Hence we have the following corollary.
\begin{cor}\label{reduce1}
For each $\varrho\in\w G$ and $ 1\leq i\leq \deg\varrho,$  $\mb P_\varrho^{ii}\m H$ and $\mb P_\varrho\m H$ are joint reducing subspaces  of $\mathbf M_{\bl \phi} $ on $\mathcal H.$  \end{cor}

%In a similar way, one can show that each $\mathbb P_{\varrho}^{ii}\m H$ is also a joint reducing subspace for $M_{\bl \phi.}$ 
%If we start with any $G$-invariant proper holomorphic map on a $G$-invariant domain and multiplication by this proper map is bounded on some Hilbert space with $G$-invariant kernel, then adequacy of this method is that it provides a family of joint reducing subspaces for that $n$-tuple of multiplication operators. 

%\section{Projections and reducing submodules}\label{projecreduce}
Now let $G$ be a finite group  generated by pseudoreflections on $\C^n.$ Recall an associated polynomial map ${\bl\theta}: \C^n \rightarrow \C^n$ given by
$
{\bl\theta}(\bl z) = \big(\theta_1(\bl z),\ldots,\theta_n(\bl z)\big),\bl z\in\C^n.
$
%Let $\{\theta_i\}_{i=1}^n$ be a set of basic polynomials associated to $G.$ 

\begin{prop}\label{properholo}
Let $\Omega$ be a $G$-invariant domain. Then 
\begin{enumerate}
    \item[(i)] $\bl \theta$ is precisely $G$-invariant,
    \item[(ii)] $\bl \theta(\Omega)$ is a domain, and
    \item[(iii)] $\bl \theta : \Omega \to \bl \theta(\Omega)$ is a proper map.
\end{enumerate}
\end{prop}
Proof of (i) can  be found in \cite[Proposition 2.2, p.705]{R}. The remaining part of the Proposition follows from \cite[Proposition 1, p.556]{Try}.
Note that $\bl \theta : \Omega \to \bl \theta(\Omega)$ is a (branched) covering  with $G$ as the group of deck transformations. Suppose that $\mathcal H$ is an analytic Hilbert module on $\Omega$ over  $\C[\bl z]$  with $G$-invariant reproducing kernel. We consider the $n$-tuple $\mathbf M_{\bl \theta} := (M_{\theta_1},\ldots,M_{\theta_n})$ of multiplication operators  on $\mathcal H,$ where $M_{\theta_i}$ is bounded for $i=1,\ldots, n.$ %Using Corollary \ref{projPii} and \ref{reduce1} and the proposition above, we describe a family of non-trivial joint reducing subspaces of $\mathbf M_{\bl \theta}$ on $\m H$ in the following theorem. 

% Therefore, by Corollary \ref{reduce1}, $\mathbb P_{\varrho} \m H$ and $\mathbb P_{\varrho}^{ii} \m H$ for $\varrho\in\w G$ and $1\leq i\leq\chi_\varrho(1),$ are joint reducing subspaces for the tuple $\mathbf M_{\bl \theta}.$ 
 With the aid of the following two lemmas we show that $\mb P^{ii}_\varrho\m H\neq \{0\}$ for $\varrho\in\w G$ and $1\leq i\leq\deg\varrho.$
\begin{lem}\label{notzero}
Let $\m H$ be an analytic Hilbert module with a $G$-invariant reproducing kernel. Then, for every $\varrho\in\w G, $ $\mathbb P_{\varrho}$ is a non-trivial projection on $\m H.$ 
\end{lem}
\begin{proof}
For a finite pseudoreflection group $G,$ we recall that $\mathbb{P}_{\varrho}\C[\bl z]$ is a free module over $\C[\bl z]^G$ of rank $(\deg\varrho)^2,$ see \cite[Theorem II.2.4., p.28]{Pan} and $\mathbb{P}_{\varrho}\C[\bl z]$ is contained in $\mathbb{P}_{\varrho} \m H$ for every $\varrho\in \w G.$ Hence $\mathbb P_{\varrho}$ is a non-trivial projection on $\m H.$ 
\end{proof}
We record the non-triviality of the projections $\mathbb{P}_\varrho^{ii}$ in the following corollary. The main ingredient of the proof is borrowed from \cite[p.166]{KS}.
\begin{cor}\label{projPii}
Let $\m H$ be an analytic Hilbert module with a $G$-invariant reproducing kernel. Then, for every $\varrho\in\w G $ and  $1\leq i\leq \deg\varrho,$ $\mathbb{P}_\varrho^{ii}$ is a non-trivial projection on $\m H.$
 \end{cor}
\begin{proof}
Since each $\mathbb{P}_\varrho$ is non-trivial by Lemma \ref{notzero}, it follows from Equation \eqref{rhoiisum} and Lemma \ref{uni} that each $\mathbb{P}_\varrho^{ii}$ is non-trivial.
\end{proof}

%From Proposition \ref{proj1}, we have
%$$
%\mathbb{P}_\varrho^{ij}\mathbb{P}_\varrho^{jj} = \mathbb{P}_\varrho^{ij} = \mathbb{P}_\varrho^{ii}\mathbb{P}_\varrho^{ij},
%$$ and it then follows that $\mathbb{P}_\varrho^{ij}\mathbb{P}_\varrho^{jj}\m  H \subseteq \mathbb{P}_\varrho^{ii}\m H$. Also for $f\in \m H,$
%$$
%\mathbb{P}_\varrho^{ii}f = \mathbb{P}_\varrho^{ij}\mathbb{P}_\varrho^{ji}f =  \mathbb{P}_\varrho^{ij}\mathbb{P}_\varrho^{jj}\mathbb{P}_\varrho^{ji}f
%$$
%and thus $\mathbb{P}_\varrho^{ii}\m H\subseteq \mathbb{P}_\varrho^{ij}\mathbb{P}_\varrho^{jj}\m H .$ Consequently, $\mathbb{P}_\varrho^{ij}$ is a surjective map from $\mathbb{P}_\varrho^{jj}\m H$ onto $\mathbb{P}_\varrho^{ii}\m H$.
%$\mathbb{P}_\varrho^{ij}$ is a surjective map from $\mathbb{P}_\varrho^{jj}\m H$ onto $\mathbb{P}_\varrho^{ii}\m H$.
%Now $\mathbb{P}_\varrho^{ij}\mathbb{P}_\varrho^{jj}f = 0$ implies that $\mathbb{P}_\varrho^{ji}\mathbb{P}_\varrho^{ij}\mathbb{P}_\varrho^{jj}f = 0$ and hence $\mathbb{P}_\varrho^{jj}f = (\mathbb{P}_\varrho^{jj})^2f = 0$. This shows that $\mathbb{P}_\varrho^{ij}$ is injective on $\mathbb{P}_\varrho^{jj}\m H.$ The operator  $\mathbb{P}_\varrho^{ij}$, being a finite linear combination of unitaries,  is bounded and hence an invertible map (by the open mapping theorem) from $\mathbb{P}_\varrho^{jj}\m H$ onto $\mathbb{P}_\varrho^{ii}\m H$. Since each $\mathbb{P}_\varrho$ is non-trivial by Lemma \ref{notzero}, it follows from Equation \eqref{rhoiisum} that each $\mathbb{P}_\varrho^{ii}$ is non-trivial.

Note that the polynomial map $\bl \theta$ associated to a pseudoreflection group $G$ is a $G$-invariant proper holomorphic map. Hence Corollary \ref{reduce1} along with Corollary \ref{projPii}  describes a family of non-trivial joint reducing subspaces of $\mathbf M_{\bl \theta}$ on $\m H$ in the following theorem. 

%\begin{thm}\label{fd1}
%Let $\m H$ be an analytic Hilbert module with  $G$-invariant kernel. Then $\mb P_\varrho^{ii}\m H$ is a joint reducing subspace for $M_{\theta_k},\,k=1,\ldots,n,$ for every $\varrho\in\w G$ and for each $i$ with $ 1\leq i\leq {\rm deg}\,\varrho.$  Consequently, the number of joint  reducing subspaces for the operator tuple $\mathbf M_{\bl\theta}=(M_{\theta_1},\ldots, M_{\theta_n})$ is at least $\sum_{\varrho\in\w G} {\rm deg}\, \varrho.$
%\end{thm} 
%However, the above result may be translated  in the Hilbert module language. We record it as a corollary.   
\begin{thm}\label{fd1}%\label{hilmod}
Let $\m H$ be an analytic Hilbert module with a $G$-invariant kernel. Then each $\mb P_\varrho^{ii}\m H$ is a reducing submodule of the Hilbert module $\m H$ over the ring $\C[\bl z]^G$ of $G$-invariant polynomials (that is, joint reducing subspaces for the operator tuple $\mathbf M_{\bl\theta}=(M_{\theta_1},\ldots, M_{\theta_n})$, for every $\varrho\in\w G$ and for each $i$ with $ 1\leq i\leq  {\rm deg}\,\varrho.$  Consequently, the number of  reducing submodules over $\C[\bl z]^G$ is at least $\sum_{\varrho\in\w G} {\rm deg}\, \varrho.$
\end{thm}

%The following corollary provides an upper bound on the number of mutually inequivalent  reducing submodules of $\m H$ over $\mb C[\bl z]^G.$
%\begin{cor}
%If $\mathbb P_{\varrho}^{ii}\m H$ is a minimal reducing submodule for all $\varrho\in\w G$ and $1\leq i\leq \chi_\varrho(1). $ Then the number of inequivalent  reducing submodules of $\mathbf M_{\bl\theta}$ cannot exceed $\lvert \w G\rvert.$
%\end{cor}

%Proposition \ref{rhojredsub} and Corollary \ref{hilmod}, we have the following theorem.
%\begin{thm}\label{equiredsubsp}
%For $\varrho\in\w G$ and $1\leq i,j\leq \chi_\varrho(1),$ the  reducing submodules $\mathbb P_{\varrho}^{ii}\m H$ and $\mathbb P_{\varrho}^{jj}\m H$ of $\m H$ are unitarily equivalent as Hilbert  modules over the ring of $G$-invariant polynomials. 
%\end{thm}

%\begin{lem}
%$\mb P_{\varrho}\mathcal O (\Omega)$ forms a free module of rank $\chi_{\varrho}(1)^2$ over the ring of $G$-invariant holomorphic functions.
%\end{lem}
%--------------------------------------------------------------------
 \subsection{Examples of analytic Hilbert modules with $G$-invariant kernels}
We give several examples where the results established in the Section \ref{decanahil} and Section \ref{redsub} apply. In particular, the Theorem \ref{fd1} gives an account of joint reducing subspaces of $\mathbf M_{\bl \theta}$ on $\mathcal H$, where $\bl\theta$ and $\m H$ vary over the examples below.
\begin{ex}\label{example}\rm
Let $\mathcal U_d$ denote the group of unitary operators on $\C^d.$ Recall that any finite group $G$ generated by pseudoreflections on $\C^d$ is a subgroup of $\mathcal U_d.$ The Bergman space $\mb A^2(\Omega)$ of a domain $\Omega\subseteq \mb C^d$ consists of all analytic functions on $\Omega$ which are square integrable with respect to the Lebesgue measure. It is a Hilbert space with a reproducing kernel $B_\Omega,$ called the Bergman kernel. The transformation rule for the Bergman kernel under biholomorphic mappings \cite[p. 419]{JP} shows that the Bergman kernel $B_\Omega$ of an $\m U_d$-invariant domain $\Omega$ in $\mb C^d$ is $\m U_d$-invariant. In particular, the Bergman kernel  of a $G$-invariant domain in $\mb C^d$ is $G$-invariant.

% Let us recall  Berezin-Wallach set for a domain $\Omega\subseteq\mb C^d$ is given by
% $$\wi{\mathscr W}(\Omega) := \{\l\in \mb C : B^{\l}_\Omega \textnormal{ is positive definite kernel on } \Omega \},$$
%  where $B^{\l}_\Omega(\bl z, \bl w) := (B_\Omega(\bl z,\bl w))^{\l} $ for all $\bl z, \bl w \in \Omega.$ Clearly, $\wi{\mathscr W}(\Omega)\neq \emptyset,$ since it contains all positive integers.   Each $B^{\l}_\Omega$ induces a unique Hilbert space $\mathscr{H}^{(\l)}$ consists of analytic functions on $\Omega$ \cite[p. 172]{BM}.  One of the significances of this example is that just from $G$-invariance  of $\Omega\subseteq \mb C^d$ we are able to conclude that the Bergman kernel  $B_\Omega$ of $\Omega$ is $G$-invariant and $B^\l_\Omega$ is also $G$-invariant for $\l\in\wi{\mathscr W}(\Omega),$ as well. In addition, we provide  families of examples of $G$-invariant domains  where the formulae for the Bergman kernels are known explicitly. We point out that multiplication operators by the coordinate functions on a Hilbert space $\m H$ of holomorphic functions on a $G$-invariant domain $\Omega\subseteq \mb C^d,$ having a $G$-invariant positive definite kernel $K$ need  not be bounded. Consequently, $\m H$ need not define a Hilbert module over $\mb C[z_1,\ldots, z_d].$  
  
\vspace{0.1in}
\begin{enumerate}[leftmargin=*]
   \item[(i)]For $n\geq m\geq 1,$ let $\mb C^{m \times n}$  denote the linear space of $m \times n$ complex matrices. Let $G \subseteq \mathcal U_m$ be a finite pseudoreflection group and $G$ act on $\mb C^{m \times n}$ by
$(g , \bl z)\mapsto g^{-1} \bl z .$ If $\mb B_{mn}$ denotes the open unit ball in $\mb C^{m\times n}$ with respect to the operator norm on $\mb C^{m \times n},$ then $\mb B_{mn}$ is invariant under $G$. By the discussion above the Bergman space  $\mb A^2(\mb B_{mn})$ is an analytic Hilbert module over the polynomial ring with $mn$-variables  with  $G$-invariant Bergman kernel $B_{\mb B_{mn}}$ with the suitable choice of $G$. In fact, one can directly verify the $G$-invariance of $B_{\mb B_{mn}}$  from the formula
\bea \label{bk}
B_{\mb B_{mn}}(\bl z, \bl w)=\big(\det(I_m -\bl z \bl w^{*})\big)^{-(m+n)}, 
\eea
where  $I_m$ is  the identity matrix of order $m$ and $X^*$ denotes the conjugate transpose of the matrix $X,$ \cite[p.172]{BM}.

%In order to characterize the subset $E\subseteq\wi{\mathscr W}(\mb B_{mn})$ such that $B_{\mb B_{mn}}^\l$ (pointwise power) for $\l\in E$, defines a Hilbert module over the polynomial ring with $mn$-variables with reproducing kernel $B_{\mb B_{mn}}^\l,$ we appeal to \cite[Theorem 1.1]{BM}. In order to do that it is convenient to have the identical parametrization of 
We recall that the Berezin-Wallach set $\mathscr W$ for $\mathbb B_{mn}$ \cite[p.172]{BM}:
%$\mathscr W_{mn}:=\mathscr W(\mb B_{mn})$ from \cite[p. 172]{BM}.  
%Accordingly, we set
\Bea
\mathscr W:= \{\l : B^{(\l)}_{\mb B_{mn}} \textnormal{ is positive definite kernel on } \mb B_{mn} \},
\Eea
 where $B^{(\l)}_{\mb B_{mn}}(\bl z, \bl w) :=\big (B_{\mb B_{mn}}(\bl z,\bl w)\big)^{\frac{\l}{m+n}} $ for all $\bl z, \bl w \in \mb B_{mn}.$ 
  Each $B^{(\l)}_{\mb B_{mn}}$ induces a unique Hilbert space $\mathscr{H}^{(\l)}$ consisting of analytic functions on $\mb B_{mn}.$ % The set $\mathscr W$ has been determined in \cite{FK}. It is
  %\bea\label{Wcd}\mathscr W=\mathscr W_d\sqcup\mathscr W_c,\eea
  %where $\mathscr W_d,$ the discrete part of the Berezin-Wallach set, and $\mathscr W_c,$ its continuous part, are given by
  
  %\bea\label{Wallach}
  %\mathscr W_d=\{0, 1,\ldots, m-1\},\hspace{.2 cm} \mathscr W_c=\{\l:\l>m-1\}.
 % \eea
For $\l\in\mathscr W,$ the $mn$-tuple $\mathbf M^{(\l)}=\big(M_{ij}^{(\l)}\big)$ of multiplication operators on $\mathscr H^{(\l)}$ defined by
\Bea 
\hspace{0.45 in}(M_{ij}^{(\l)}f)(\bl z)=z_{ij}f(\bl z), \hspace{.2 cm}\bl z=(z_{ij})\in\mb B_{mn}, ~ f\in\mathscr H^{(\l)},~ 1\leq i\leq m, ~ 1\leq j\leq n,
\Eea
is bounded if and only if $\l$ is in the continuous part of $\mathscr W$, that is, $\l>m-1$ \cite[Theorem 1.1]{BM}. In fact, it follows from \cite[Theorem 1.1]{BM} that $\mathscr H^{(\l)}$ is an analytic Hilbert module over the polynomial ring with $mn$ variables with $G$-invariant reproducing kernel $B_{\mb B_{mn}}^{(\l)}$ for $\l>m-1$. 

The open unit ball $\mb B_n$  with respect to the $\ell^2$-norm on $\mb C^n$ is a particular case of this family of examples, namely, $\mb B_n=\mb B_{1n}$ and the continuous part of $\mathscr W$ in this case is $\l>0$.  
\begin{rem}
One can take the following group actions as well:
\begin{enumerate}
    \item Let $G \subseteq \mathcal U_n$ be a finite pseudoreflection group and $G$ act on $\mb C^{m \times n}$ by by $(g , \bl z)\mapsto \bl{z} g .$
    
    \item 
 Let $G_1\subseteq\m U_m$ and  $G_2\subseteq\m U_n$ be finite pseudoreflection groups and let the finite pseudoreflection group $G:=G_1\times G_2\subseteq \m U_m\times \m U_n$ act on $\mb C^{m\times n}$ by 
$$(g_1, g_2)\cdot \bl z=g_1^{-1}\bl zg_2.\text{~ for~} (g_1, g_2)\in G_1\times G_2 \text{ ~and ~} \bl z\in\mb C^{m\times n}.$$
\end{enumerate}
Analogous results as in the Section \ref{redsub}
will hold if we modify the group action there with the above. 
\end{rem}

 %Any unitary operator $U\in\m U_n$  has the following property 
 %\bea \label{unitary}
 %\langle U\bl z,U\bl w\rangle=\langle \bl z,\bl w\rangle \text{~for~} \bl z, \bl w \in \C^n.
% \eea
% Consequently, $\mb B_n$ is a  $\m U_n$-invariant domain.
%By the discussion above the Bergman kernel of $\mb B_n$ is invariant under any finite group $G\subseteq \m U_n$ generated by pseudoreflections. Indeed, one can directly verify that the Bergman kernel \eqref{bk}
%$$ $$B_{\mb B_n}(\bl z, \bl w)=\big(1-\langle \bl z, \bl w \rangle\big)^{-(n+1)},~~ \bl z,\bl w\in\mb B_n,$$ is invariant under $G.$ From \eqref{Wcd} and \eqref{Wallach},  $\mathscr W_c(\mb B_n)=\{\l\in\mb C:\l>0\}.$  Hence by Theorem 1.1 in \cite{BM}, for each $\l>0,$  $ \mb A^{(\l)}(\mb B_n)$ is an analytic Hilbert module over $\mb C[z_1,\ldots, z_n]$ with $G$-invariant reproducing kernel $B_{\mb B_n}^\l.$ In addition to the Bergman kernel $B_{\mb B_n}$ (for $\l=n+1$) this family includes  the reproducing kernels of the Hardy space $H^2(\partial \mb B_n)$ and the Drury-Arveson space $H^2_n$ for $\l=n$ and $\l=1,$ respectively.
 \vspace{.1 cm}
    \item [(ii)]Let $\mathfrak S_n$ denote the permutation group on $n$ symbols. The action of $\mathfrak S_n$ on $\mb C^n$ is given by
   \Bea
   \sigma\cdot (z_1,\cdots,z_n):=(z_{\sigma^{-1}(1)},\cdots, z_{\sigma^{-1}(n)}) \text{~for~} \sigma\in\mathfrak S_n, (z_1,\cdots, z_n)\in\mb C^n.
   \Eea
   \begin{enumerate}
       \item [(a)] Let $\mathbb D^n$ be the unit polydisc in $\C^n$. The Hilbert space $\mb A^{(\l)}(\mb D^n),\l>0$  with reproducing kernel 
       \Bea
       \prod_{j=1}^n(1-z_j\ov w_j)^{-\l}, ~~\bl z,\bl w\in\mb D^n
      \Eea
       is an analytic Hilbert module over $\mb C[\bl z]$ with $\mathfrak S_n$-invariant kernel. 
       \vspace{.1 cm}
       \item[(b)]Let $\a\in\mb R$ be fixed. A holomorphic function $f:\mb D^n\to \mb C$ belongs to the {\it Dirichlet-type space} $\mathfrak D_\a$ if the coefficients in its Taylor series expansion
       \Bea
       f(z_1,\ldots,z_n)=\sum_{k_1=0}^\infty\ldots\sum_{k_n=0}^\infty a_{k_1}\ldots a_{k_n}z^{k_1}_1\ldots z^{k_n}_n
       \Eea
       satisfies
       \bea\label{norm}
       \Vert f\Vert_\a^2=\sum_{k_1=0}^\infty\ldots\sum_{k_n=0}^\infty(k_1+1)^\a\ldots (k_n+1)^\a\vert a_{k_1}\ldots a_{k_n}\vert^2 <\infty.
       \eea
       The space $\mathfrak D_\a$ is a reproducing kernel Hilbert space for $\a\in \mb R$ (see \cite{BKKLSS, LB, K}). In fact, this is an analytic Hilbert module over $\mb C[\bl z]$ \cite[p.198]{LB}. The formula for norm $\Vert\cdot\Vert_\a$  on $\mathfrak D_\a$ in \eqref{norm}  shows that $R_\sigma$ defined in Lemma \ref{iso} is unitary on $\mathfrak D_\a$ for $\sigma\in\mathfrak S_n$ and hence by Lemma \ref{iso}, the reproducing kernel of $\mathfrak D_\a$ is $\mathfrak S_n$-invariant. 
       \vspace{.1 cm} 
       \item[(c)] The minimal ball $\mb B_*$ is a domain in $\mb C^n, n\geq 2,$ defined by
    \Bea
    \mb B_*:=\{\bl z\in\mb C^n:\frac{\Vert \bl z\Vert^2+\vert z\bullet z\vert}{2}<1\},
    \Eea
    where $\Vert \bl z\Vert^2=\sum_{j=1}^n\vert z_j\vert^2$ and $ z\bullet z=\sum_{j=1}^nz_j^2,$ see \cite{OPY, OY}. From $\mathfrak S_n$-invariance of $\mb B_*$  and  the discussion above, the Bergman kernel $B_{\mb B_*}$ of $\mb B_*$ is  $\mathfrak S_n$-invariant. Clearly, the Bergman space $\mb A^2(\mb B_*)$ forms an analytic Hilbert module over $\mb C[\bl z]$ with $\mathfrak S_n$-invariant kernel.
   \end{enumerate} 
  % More examples of $\mathfrak S_n$-invariant domains and $\mathfrak S_n$-invariant kernels were discussed in \cite[p. 764 ]{BGMS}.
    \item[(iii)]%For $k\geq 1,$  let $G=\langle \delta,\sigma\rangle \cong D_{2k}$ be the dihedral group of order $2k,$ let $\zeta_k:=\exp\frac{2\pi i}{k}$ be a primitive $k$-th root of unity and let 
   % \Bea
   % D_V: G\to GL(2,\mb C): \delta\mapsto \begin{bmatrix}
   % \zeta_k & 0\\
   % 0 & {\zeta_k}^{-1}
  %  \end{bmatrix}, \sigma\mapsto \begin{bmatrix}
   % 0 & 1\\
   % 1 & 0
    %\end{bmatrix}
   % \Eea
   % be a faithful representation of $D_{2k}.$
   % We have $G=\{\delta^j, \sigma\delta^j:j\in\{0,\ldots,k-1\}\},$ where $\delta^j \in G$ is a rotation having the eigenvalues ${\zeta_k}^{\pm j}\in\mb C$ and $\sigma\delta^j\in G$ is a reflection having the eigenvalues $\pm 1\in\mb  C.$
   % Let $G$ act on $\mb C^2$ by matrix multiplication via the faithful representation $D_V,$ viewing $\mb C^2$ as the vector space of column matrices of length $2.$
  For $k\geq 1,$ we recall the action of the dihedral group $D_{2k}$ of order $2k$ on $\mb C^2$ from Example \ref{reflex}. Clearly, $\mb D^2$ is invariant under the action of the dihedral group $D_{2k}$. The  weighted Bergman space $\mb A^{(\l)}(\mb D^2)$ for $\l>0$, the Dirichlet-type space $\mathfrak D_\a$ on $\mb D^2$ for $\a\in\R$, the Bergman space $\mb A^2(\mb B_*)$ on minimal ball $\mathbb B_*$ in $\C^2$ are analytic Hilbert modules over $\mb C[z_1,z_2]$ with $D_{2k}$-invariant kernel. 
\end{enumerate}
\end{ex}

\begin{ex}\rm
Suppose that $K:\mb B_n\times\mb B_n\to\C$ is given by 
 \bea \label{Diri} K(\bl z,\bl w)=\sum_{k=0}^\infty a_k\langle \bl z,\bl w\rangle^k \text{~with~} a_k\geq 0 \text{~for~} k\geq 0.
 \eea
 Then $K$ is clearly a positive definite kernel on $\mb B_n$ and  $K$ is $G$-invariant for $G\subseteq\m U_n.$
 %Moreover, any unitary operator $U$ on $\C^n$ has the following property 
 %$$\langle U\bl z,U\bl w\rangle=\langle \bl z,\bl w\rangle \text{~for~} \bl z, \bl w \in \C^n. By the Equation \eqref{unitary} and the discussion above  we conclude that the kernel $K$ remains invariant under the action of any finite group $G\subseteq \m U_n$ generated by pseudoreflections on $\mb C^n.$ 
\vspace{0.1in}
\begin{enumerate}[leftmargin=*]
\item[(i)] In addition to the positive powers of the Bergman kernel $B_{\mb B_n}$ (mentioned  in the earlier example too), the family \eqref{Diri}  includes the reproducing kernel of the Dirichlet space $\mathscr D_n$ on $\mb B_n$ given  by \cite[p.218]{GHX}
 \Bea
 -\frac{\ln (1-\langle \bl z,\bl w\rangle)}{\langle \bl z,\bl w\rangle}=\sum_{m=0}^\infty\frac{\langle \bl z,\bl w\rangle^m}{m+1}, ~~\bl z,\bl w\in\mb B_n.
 \Eea
 %The kernels $K_{n+1},$ $K_n$ and $K_1$ are reproducing kernels of the Bergman space $L^2_{a}(\mb B_n),$ the Hardy space $H^2(\partial \mb B_n)$ and the Drury-Arveson space $ H^2_n$, respectively.
 The Dirichlet space is an analytic Hilbert module over $\mb C[\bl z]$ with a $G$-invariant kernel.

%\begin{ex}
%Let  $L^2_{a}(\mb B_{mn})$ denote  the Bergman space consisting of all analytic functions on $\mb\ B_{mn}$ which are square integrable with respect to Lebesgue measure. It is well known that  $L^2_{a}(\mb B_{mn})$ is a reproducing kernel Hilbert space with the positive definite kernel $$B(\bl z,\bl w)=\big(\det(I_m -\bl z \bl w^{*})\big)^{-(m+n)}, $$ 
%where $\bl z, \bl w \in \mb B_{mn},$  $I_m$ is  the identity matrix of order $m$ and $X^*$ denotes the conjugate transpose of the matrix $X,$ see \cite{JA}.
%\end{ex}
 
%\begin{ex}
% Let us recall Wallach set for the open matrix unit ball  $\mb B_{mn}$ $$\mathscr W_{mn} = \{\l : B^{(\l)} \textnormal{ is positive definite kernel on } \mb B_{mn} \},$$
% where $B^{(\l)}(\bl z, \bl w) := (B(\bl z,\bl w))^{\frac{\l}{(m+n)}} $ for all $\bl z, \bl w \in \mb B_{mn}.$ 
 % Each $B^{(\l)}$ induces a unique Hilbert space $\mathscr{H}^{(\l)}$ consists of analytic functions on $\mb B_{mn}.$ See \cite{M}.

 % These are called twisted Bergman spaces. However, for $\l = m+n,$ we get back ordinary Bergman space.

 %Note that $B^{(\l)}$ is $G$-invariant whenever $B$ is so. This shows that previous discussions are applicable here as well. 
  
%\end{ex}

\vspace{.1 cm}
\item[(ii)] The holomorphic Sobolev space $H^s(\mathbb B_n)$ is a reproducing kernel Hilbert space consisting of holomorphic functions on $\mathbb B_n$ with square-integrable derivatives of order $\leq s$ and equipped with the following inner product 

\bea\label{inner}
\inner{f}{g}  = \sum_{|\bl\alpha|\leq s} \int_{\mb B_n} \frac{\del^{\bl\alpha}f}{\del{z}^{\bl\alpha}}  \frac{\del^{\bl\alpha}\overline g}{\del{\overline z}^{\bl\alpha}} dm,  
\eea
where $dm$ is the Lebesgue measure on $\mb B_n$ and $\bl\a\in \mathbb Z^n_+$ \cite[p.275]{B}.
%$= \sum_{\alpha} \frac{z^{\alpha}\overline w^{\alpha}}{\norm{z^{\alpha}}_s^2}.$ 
Since $\vert\det U\vert=1$ for $U\in\m U_n\subseteq {\rm Aut}(\mb B_n),$ the Lebesgue measure on $\mb B_n$ is $\m U_n$-invariant and hence $R_\sigma$ defined in Lemma \ref{iso} is unitary on $H^s(\mb B_n)$ for $\sigma\in G\subseteq\m U_n.$ Therefore, the reproducing kernel of $H^s(\mb B_n)$ is  $\m U_n$-invariant by Lemma \ref{iso} and hence, of the form \eqref{Diri} \cite[p.228]{GHX}. The holomorphic Sobolev space $H^s(\mb B_n)$ is an analytic Hilbert module over $\C[\bl z]$ with $G$-invariant kernel.
\end{enumerate}
\end{ex}

\section[Realization as analytic Hilbert modules] {Realization of reducing submodules as analytic Hilbert modules}

%. 
%Next lemma is required to prove part (c).
%We need another lemma to prove the next theorem.
%First we state one of the main results of this section. 

%rank of Hilbert module $\mb P_{\varrho}^{jj} \mathcal H$ over $\C[z_1,\ldots,z_n]^G.$ Suppose the rank of $\mathbb{P}_\varrho^{ii}\m H$ over the ring $\C[z_1,\ldots,z_n]^G$  

%Clearly, $\mathbb{P}_\varrho^{ii}\m H$ and $\mathbb{P}_\varrho^{jj}\m H$ are two Hilbert modules of same rank and the rank each module is strictly less than $\chi_{\varrho}(1)^2$ if $\chi_\varrho(1)>1.$
%The succeeding corollary provides an upper bound on the number of mutually inequivalent  reducing submodules of $\m H$ over $\mb C[\bl z]^G.$ 
%\begin{cor}\label{ub}
%If $\mathbb P_{\varrho}^{ii}\m H$ is a minimal  reducing submodule for all $\varrho\in\w G$ and $1\leq i\leq \chi_\varrho(1), $ then the number of inequivalent  reducing submodules of $\m H$ over $\mb C[\bl z]^G$ cannot exceed $\lvert \w G\rvert.$
%\end{cor}
%\begin{proof}
%Since each $\mb P_\varrho^{ii} \m H$ is minimal and is equivalent to all $\mb P_\varrho^{jj} \m H$ for $j=1,\ldots,\chi_\varrho(1),$ it is enough to choose one  For every $\varrho \in \widehat{G},$ each element of the set $\{\mb P_\varrho^{ii} \m H\}_{i=1}^{\chi_\varrho(1)}$ is equivalent to each other. So we take a representative for 
%\end{proof}
%\subsection{Realization as an analytic Hilbert module on $\bl \theta(\Omega)$}

Throughout this section, we assume that $\mathcal H \subseteq\mathcal O(\Omega)$ is an analytic Hilbert module over $\C[\bl z]$ with a $G$-invariant reproducing kernel, where $G$ is a finite pseudoreflection group. Let $\{\theta_i\}_{i=1}^n$ be a hsop associated to $G$ and $\bl \theta : \Omega \to \bl \theta(\Omega)$ be the map given by ${\bl\theta} = (\theta_1,\ldots,\theta_n)$. We  show that $\mathcal H$ considered as a Hilbert module over $\C[\bl z]^G$ can be realized, up to an isomorphism, as an analytic Hilbert module $\wi{\m H}\subseteq\mathcal O(\bl \theta(\Omega))\otimes \C^{|G|}$ over $\mb C[\bl z]$. In particular, this means that the multiplication operator $\mathbf{M}_{\bl \theta}$ on $\m H$ can be realized as the multiplication by the coordinate functions on $\tilde{\m H}$. Indeed, this is a generalization of the Chevalley-Shephard-Todd Theorem in the context of  analytic Hilbert modules.

Let $p_1,\ldots,p_d$ be a basis of $\C[\bl z]$ as a free module over $\C[\bl z]^G$, where $d=\lvert G\rvert. $ Consequently, by the analytic version of CST, it follows that each $f\in\m H$ can be written as 
\Bea
f=\sum_{k=1}^dp_kf_k,
\Eea
where each $f_k$ is  a  $G$-invariant holomorphic function on $\Omega.$  It follows from subsection \ref{g3.1} that there exists $\tilde f_k$ holomorphic on $\bl \theta(\Omega)$  
%(where $\bl \theta =(\theta_1,\ldots,\theta_n):\Omega\to \bl \theta(\Omega), \theta_k$'s are such that $\C[z_1,\ldots,z_n]^G=\mb C[\theta_1, \ldots, \theta_n]$) 
such that  $f_k=\tilde f_k\circ \bl \theta$ for $k=1,\ldots,d.$ On that account, we define the map $\Gamma:\m H\to \m O(\bl \theta\big(\Omega)\big)\otimes\mb C^d$ by 
\bea\label{gamma}
\Gamma f=\sum_{k=1}^d\tilde f_k\otimes \varepsilon_k,
\eea
where $\{\varepsilon_1,\ldots, \varepsilon_d\}$ is the standard basis of $\mb C^d.$ Clearly,  $\Gamma$ is injective by construction. Let $\wi{ \m H}:=\Gamma  \m H.$ Since $\Gamma$ is linear by construction, $\wi{\m H}$ can be made into a Hilbert space by borrowing the inner product from $\m H,$ that is,
\Bea
 \langle \Gamma f, \Gamma g\rangle_{\wi{\m H}}:=\langle f, g\rangle_{\m H} \,\, \text{for all} \,\, f, g\in\m H.
 \Eea
This makes the map $\Gamma : \m H \to \wi{\m H},$ a unitary. The next lemma  is one of the key results of this section. 

\begin{lem}\label{cons}
$\wi{\m H}$ is a reproducing kernel Hilbert space consisting of $\mb C^{\lvert G\rvert}$-valued holomorphic functions on $\bl\theta(\Omega).$
\end{lem}

\begin{proof}
Our proof is constructive. We break it into two parts. First we construct a function $\mb K$ on $\bl \theta(\Omega) \times \bl \theta(\Omega)$ from the reproducing kernel $K$ of $\mathcal H.$ Subsequently, we prove that $\mb K$ is the reproducing kernel for $\wi{\mathcal H}.$

\subsection*{Construction of $\mb K$} Since $K_{\bl w}\in\m H$ for every $\bl w\in\Omega,$ by analytic CST, there exist unique $K^{(1)}_i(\cdot, \bl w), 1\leq i\leq d$ such that $K_{i,\bl w}^{(1)}:=K^{(1)}_i(\cdot, \bl w)$ is holomorphic, satisfying 
\Bea
K(\bl z, \bl w)=\sum_{i=1}^d p_i(\bl z)K^{(1)}_i(\bl z,\bl w).
\Eea
By analogous argument as in Lemma \ref{gen1}, we have
\Bea
K^{(1)}_i(\bl z, \bl w)=\frac{\det\big( \Lambda^{(1)}_i(K)(\bl z,\bl w)\big)}{\det \Lambda(\bl z)},
\Eea
 where $\Lambda(\bl z)=\big(\!\!\big(\big(\rho_i\big(p_j(\bl z)\big)\big)\!\!\big)_{i,j=1}^d$ and  $\Lambda^{(1)}_i(K)(\bl z,\bl w)$ is the matrix $\Lambda(\bl z)$ with its $i$-th column replaced by the column $\Big(\big(\!\!\big(\rho_i(K_{\bl w}(\bl z))\big)\!\!\big)_{i=1}^d\Big)^{\rm tr}.$

By Kolmogorov decomposition of a reproducing kernel, it is known that there exists a function $F:\Omega\ra \mathcal L(\mathcal H, \C)$ such that $K(\bl z, \bl w) = F(\bl z) F(\bl w)^*$ for $\bl z, \bl w$ in $\Omega$ \cite[Theorem 2.62]{AMc}. For our purpose it is enough to take $F(\bl z)=\textnormal{ev}_{\bl z}$, where $\textnormal{ev}_{\bl z}$ denotes the evaluation map at $\bl z$. Thus, $F$ is a  holomorphic function from $\Omega$ into $ \mathcal L(\mathcal H, \C)$ such that $F(\bl z)h = h(\bl z)$ for $\bl z \in \Omega, h \in \mathcal H.$ 

Expanding $\det\big( \Lambda^{(1)}_i(K)(\bl z,\bl w)\big)$ with respect to the $i$-th column, separating $F(\bl w)^*$ and then summing up again, we obtain $\det\big( \Lambda^{(1)}_i(K)(\bl z,\bl w)\big)= \det \big( \Lambda^{(1)}_i F(\bl z) \big) F(\bl w)^*. $ Consequently, one has 
\Bea
K^{(1)}_i(\bl z, \bl w)=\frac{\det \big( \Lambda^{(1)}_i F(\bl z) \big)}{\det \Lambda(\bl z)} F(\bl w)^*,
\Eea
where $ \Lambda^{(1)}_iF(\bl z)$ is the matrix $\Lambda(\bl z)$ with its $i$-th column replaced by the column $\Big(\big(\!\!\big(\rho_i(F(\bl z))\big)\!\!\big)_{i=1}^d\Big)^{\rm tr}.$ 
%This expression implies $K^{(1)}_i(\bl z, \bl w)$ is anti-holomorphic in $\bl w$ as $F(\bl w)^*$ is so. 
Set $G_i(\bl z) = \frac{\det \big( \Lambda^{(1)}_i F(\bl z) \big)}{\det \Lambda({\bl z})}$ for $1\leq i\leq d$ and thus, 
\bea\label{breaking}
\overline{K^{(1)}_i(\bl z, \bl w)} = F(\bl w) G_i(\bl z)^*.
\eea  Since $F(\bl z)$ is in $\mathcal L(\mathcal H, \C),$ so is $G_i(\bl z)$ and hence $G_i(\bl z)^* \in \mathcal L(\C,\mathcal H)$ for every $\bl z \in \Omega.$
%Note that $F$ is nothing but the evaluation map that is, $F(\bl z)(f)=f(\bl z)$ for all $f \in \mathcal H.$ 
%Next let us fix some $\bl z \in \Omega$ and look at $\overline{K^{(1)}_i(\bl z, \bl w)}.$ 
Thus, for a fixed $\bl z \in \Omega$ there exists $h \in \mathcal H$ such that $G_i(\bl z)^*(1)=h$. 
%Next denote $G(\bl z)^*(1)=h$ and observe $h \in \mathcal H.$ 
Hence $\overline{K^{(1)}_i(\bl z, \bl w)} = F(\bl w)h = h(\bl w).$ So for each $\bl z,$ $\overline{K^{(1)}_i(\bl z, \bl w)} \in \mathcal H$ and holomorphic in $\bl w.$ Therefore, we have that
\begin{enumerate}
\item[1.]$K^{(1)}_i(\bl z,\bl w)$' s are unique and $G$-invariant in $\bl z,$ that is, $K^{(1)}_i(\rho \cdot \bl z,\bl w) = K^{(1)}_i(\bl z,\bl w)$ for every $\rho \in G$ and $ 1\leq i\leq d$,
\item[2.]$K^{(1)}_i(\bl z,\bl w)$ is holomorphic in $\bl z$ and anti-holomorphic in $\bl w\in\Omega$,
\item[3.] $\ov{K_i^{(1)}(\bl z,\bl w)}\in\m H$, when viewed as a function of $\bl w$ fixing $\bl z.$
\end{enumerate}
Now for a fixed $\bl z,$ consider $\ov{K^{(1)}_i(\bl z,\bl w)}$ as a function of $\bl w,$ by analytic CST, we have 
$$\ov{K^{(1)}_i(\bl z,\bl w)}=\sum_{j=1}^d p_j(\bl w)\w K^{(12)}_{ij}(\bl z,\bl w).$$
Defining $K^{(12)}_{ij}(\bl z,\bl w):=\ov{\w K^{(12)}_{ij}(\bl z,\bl w)},$ we get
$$K^{(1)}_i(\bl z,\bl w)=\sum_{j=1}^dK_{ij}^{(12)}(\bl z, \bl w)\ov{p_j(\bl w)},$$
where
\begin{equation}\label{kd12}
K^{(12)}_{ij}(\bl z, \bl w)=\frac{{\det\big( \Lambda^{(2)}_j(K_i^{(1)})(\bl z,\bl w)\big)}}{\ov{\det \Lambda(\bl w)}} = G_i(z)G_j(w)^*.
\end{equation}
Here the matrix $\Lambda^{(2)}_j(K_i^{(1)})(\bl z,\bl w)$ is obtained by replacing the $j$-th column of the matrix $\Lambda(\bl w)$ by the column $\Big(\big(\!\!\big(\rho_j(K_{i}^{(1)}(\bl z,\bl w))\big)\!\!\big)_{j=1}^d\Big)^{\rm tr}.$ Note that the difference in the superscripts of $\Lambda_j$'s emphasize whether we apply the operation on the first variable $\bl z$ or on the second variable $\bl w.$ For example, $\Lambda^{(2)}_j(h)(\bl z, \bl w)$ denotes a matrix in $\bl w$ keeping $\bl z$ fixed while $ \Lambda^{(1)}_j(h)(\bl z, \bl w)$ denotes a matrix in $\bl z$ but $\bl w$ is fixed, whenever $h \in \mathcal H$. Thus, we have
\begin{equation}\label{in z}
K(\bl z,\bl w)=\sum_{i,j=1}^dp_i(\bl z)K^{(12)}_{ij}(\bl z, \bl w)\ov{p_j(\bl w)}.
\end{equation}
 As before, using \eqref{kd12} instead of \eqref{breaking}, we infer that $K^{(12)}_{ij}(\bl z, \bl w)$'s are unique, $G$-invariant and anti-holomorphic in $\bl w$, holomorphic in $\bl z$ and belongs to $\mathcal H$ as a function of $\bl z$ for each $\bl w,\,1\leq j\leq d$. At this point, we prove that  this decomposition of  $K(\bl z,\bl w)$ is independent of  the order of choice of variables $\bl z$ or $\bl w$. A similar decomposition for $\overline{K(\bl z, \bl w)}$ in the second variable $\bl w$ first, keeping $\bl z$ fixed, yields
\begin{equation}\label{first in w}
   K(\bl z, \bl w) = \sum_{j=1}^d \overline{p_j(\bl w)} K^{(2)}_j(\bl z, \bl w),
\end{equation}
where $K^{(2)}_j(\bl z,\bl w)$'s are unique, $G$-invariant and anti-holomorphic in $\bl w$, holomorphic in $\bl z$ and belongs to $\mathcal H$ as a function of $\bl z$ for each $\bl w, 1\leq j\leq d$. Therefore, applying analytic CST on $K^{(2)}_j(\bl z,\bl w)$ as a function of $\bl z,$ we get 
\begin{equation}\label{for w}
     K^{(2)}_j(\bl z, \bl w) = \sum_{i=1}^d p_i(\bl z) K^{(21)}_{ji}(\bl z, \bl w),
\end{equation}
and then
%\begin{equation}\label{in w}
 $K(\bl z,\bl w) = \sum_{i,j=1}^dp_i(\bl z)K^{(21)}_{ji}(\bl z, \bl w)\ov{p_j(\bl w)}.$
%\end{equation}
From Equation \eqref{in z}, one then concludes that 
\bea\label{equality}
K^{(12)}_{ij}(\bl z, \bl w)=K^{(21)}_{ji}(\bl z, \bl w).
\eea
%This observation will also be useful in the later part of this proof. 
In fact, in the same way, one can show that if 
\bea\label{unique0}
K(\bl z,\bl w)=\sum_{i,j}p_i(\bl z)H_{ij}(\bl z,\bl w)\ov{p_j(\bl w)},
\eea
then $H_{ij}(\bl z,\bl w)=K^{(12)}_{ij}(\bl z,\bl w)$ and hence, is uniquely determined.
Now, by construction, $K_{ij}^{(12)}$'s are separately $G$-invariant, that is,
$$K^{(12)}_{ij}(\rho\cdot \bl z, \rho^\i\cdot \bl w)=K^{(12)}_{ij}(\bl z, \bl w)~~\text{for all}~~\rho,\rho^\i\in G.$$
By repeated applications of analytic CST in each variable, there exists a function $\wi K^{(12)}_{ij}$ on $\bl \theta (\Omega)\times\Omega$ such that 
$K^{(12)}_{ij}(\bl z,\bl w)=\wi K^{(12)}_{ij}\big(\bl \theta(\bl z),\bl w\big)$,
which is $G$-invariant  and anti-holomorphic in $\bl w$ and hence, there exists a function $\wi{\wi K}^{(12)}_{ij}$ on $\bl \theta (\Omega)\times\bl \theta (\Omega)$ satisfying $K^{(12)}_{ij}(\bl z,\bl w)=\wi K^{(12)}_{ij}\big(\bl \theta(\bl z),\bl w\big)=\wi{\wi K}^{(12)}_{ij}\big(\bl \theta(\bl z),\bl \theta(\bl w)\big).$ Define 
	\Bea
	\mb K(\bl u,\bl v)=\big(\!\!\big(\wi{\wi K}^{(12)}_{ij}(\bl u,\bl v)\big)\!\!\big)_{i,j=1}^d, ~~\bl u,\bl v\in\bl \theta(\Omega).
	\Eea
Clearly, by definition, $\mb K$ is holomorphic in $\bl u$ and anti-holomorphic in $\bl v$.

\subsection*{Reproducing property of $\mb K$} We are now required to prove the following:
\begin{enumerate}
    \item[(a)] $\mb K(\cdot, \bl v)\bl x\in\wi{\m H}$ for $\bl x\in\mb C^d$ and,
    \item[(b)] $\langle \tilde f,\mb K(\cdot ,\bl v)\bl x\rangle_{\wi{\m H}} = \langle \tilde{f}(\bl v),\bl x\rangle_{\C^d}$ for $\bl x \in \C^d$ and $\tilde{f} \in \m H.$
\end{enumerate}
Let $\{\varepsilon_1, \ldots,\varepsilon_d\}$ be the standard basis of $\mb C^d,$
$$\mb K(\bl u,\bl v)\varepsilon_j=\Big(\big(\!\!\big(\wi{\wi K}^{(12)}_{ij}(\bl u,\bl v)\big)\!\!\big)_{i=1}^d\Big)^{\tr},$$
the $j$-th column of $\mb K(\bl u,\bl v).$ It is enough to show that $\mb K(\bl u,\bl v)\varepsilon_j\in\wi{\m H}$ for $1\leq j\leq d.$ We note that by construction, for $\bl u=\bl \theta(\bl z), \bl v=\bl \theta(\bl w), \,\bl z, \bl w \in \Omega$,
$$\big(\!\!\big(\wi{\wi K}^{(12)}_{ij}(\bl u,\bl v)\big)\!\!\big)_{i=1}^d=\big(\!\!\big( K^{(12)}_{ij}(\bl z,\bl w)\big)\!\!\big)_{i=1}^d, 1\leq j\leq d.$$
So it suffices to show that $\sum_{i=1}^dp_i(\bl z)K^{(12)}_{ij}(\bl z,\bl w)\in\m H$
 for each fixed $\bl w\in\Omega$ as a function of $\bl z.$ This is true since  each $K^{(12)}_{ij}(\bl z,\bl w)\in\m H$. In fact, Equation \eqref{for w} and Equation \eqref{equality} give us $K^{(2)}_j(\bl z, \bl w) = \sum_{i=1}^d p_i(\bl z) K^{(12)}_{ij}(\bl z, \bl w)$ for fixed $\bl w \in \Omega$ and consequently, 
$$\Gamma\big(K^{(2)}_j(\bl z, \bl w)\big)=\mb K(\bl u,\bl v)\varepsilon_j~~ \text{for}~~ 1\leq j\leq d.$$
%$K(\bl z, \bl w)=\sum_{j=1}^d\ov{p_j(\bl w)}K^{(2)}_j(\bl z,\bl w),$ by previous discussion we have $K^{(2)}_{j}(\bl z,\bl w)\in\m H,$
 Therefore, for any $\bl x = (x_1, \ldots, x_n)\in\C^n$,
\Bea
\mb K(\cdot,\bl v)\bl x =\sum_{j=1}^d x_j\Gamma\big(K_j^{(2)}(\cdot,\bl w)\big)=\Gamma\big(\sum_{j=1}^d x_jK_j^{(2)}(\cdot,\bl w)\big),
\Eea
and  thus
\bea\label{repprop}
\langle \Gamma f,\mb K(\cdot,\bl v)\bl x\rangle_{\wi{\m H}} = \langle f,\big(\sum_{j=1}^d x_jK_j^{(2)}(\cdot,\bl w)\big)\rangle_{\m H} = \sum_{j=1}^d \bar{x_j} \langle f,K_j^{(2)}(\cdot,\bl w)\rangle_{\m H}.
\eea
%Taking complex conjugate of $K(\bl w,\bl z)=\sum_{j=1}^dp_j(\bl w)K^{(1)}_j(\bl w,\bl z),$ we obtain $K(\bl z,\bl w)=\sum_{j=1}^d\ov{p_j(\bl w)}\ov{K^{(1)}_j(\bl w,\bl z)}=\sum_{j=1}^d\ov{p_j(\bl w)}\ov{K^{(2)}_j(\bl z,\bl w)}.$ Hence $K_j^{(2)}(\bl z,\bl w)=\ov{K_j^{(1)}(\bl z,\bl w)}\in\m H$ for $1\leq j\leq d.$
%Recall that $$K_j^{(1)}(\bl z,\bl w)=\frac{\det\big(M^{(1)}_j(K)(\bl z, \bl w)\big)}{\det M(\bl z)},$$ interchanging $\bl z$ and $\bl w,$  $$K_j^{(1)}(\bl w,\bl z)=\frac{\det\big(M^{(1)}_j(K)(\bl w, \bl z)\big)}{\det M(\bl w)}.$$ 

From Lemma \ref{gen1} and Equation \eqref{first in w}, we have 
\bea\label{k2}
K_j^{(2)}(\cdot,\bl w)=\frac{{\det\big(\overline{\Lambda}^{(2)}_j(K)(\cdot, \bl w)\big)}}{\ov{\det \Lambda(\bl w)}}~~\text{for}~~ 1\leq j\leq d,
\eea
where $\overline{\Lambda}^{(2)}_j(K)(\cdot, \bl w)$ is the matrix $\big(\!\!\big(\overline{\rho_i\big({p_j(\bl w)}\big)}\big)\!\!\big)_{i,j=1}^d$ with its $j$-th column replaced by the column
$\Big(\big(\!\!\big(\rho_j(K(\cdot,\bl w)\big)\big)\!\!\big)_{j=1}^d\Big)^{\tr}.$  
Note that, for $1\leq j\leq d$,
\Bea
\langle f, \rho_j\big(K(\cdot,\bl w)\big) \rangle = \langle f, K(\cdot, \rho_j^{-1}\cdot \bl w) \rangle = f(\rho_j^{-1}\cdot \bl w) = (\rho_j\big(f(\bl w)\big).
\Eea
Expanding  $\det\big(\overline{\Lambda}^{(2)}_j(K)(\cdot,\bl w)\big)$ with respect to the $j$-th column, we get 
\Bea
\langle f, \det\big(\overline{\Lambda}^{(2)}_j(K)(\cdot,\bl w)\big) \rangle=\langle f, \sum_{l=1}^d \rho_l\big(K(\cdot,\bl w)\big) \overline{A_l(\bl w)}  \rangle  = \sum_{l=1}^d  A_l(\bl w)\rho_l\big(f(\bl w)\big),\Eea
where $\overline{A_l(\bl w)}$ is the cofactor associated to $\rho_l\big(K(\cdot,\bl w)\big).$
In short, 
\Bea
 \langle f, \det\big(\overline{\Lambda}^{(2)}_j(K)(\cdot,\bl w)\big) \rangle =\det \Lambda_{j}f(\bl w)\,\,\text{for all}\,\, 1\leq j\leq d,
\Eea
where $\Lambda_{j}f(\bl w)$ is the matrix $\big(\!\!\big(\rho_i\big(p_j(\bl w)\big)\big)\!\!\big)_{i,j=1}^d$ with its $j$-th column replaced by the column $\Big(\big(\!\!\big(\rho_j(f(\bl w)\big)\big)\!\!\big)_{j=1}^d\Big)^{\tr}$, as we find in the proof of Lemma \ref{gen1}. Hence by Equation (\ref{k2})
\begin{equation*}
\det \Lambda(\bl w)\langle f, K_j^{(2)}(\cdot ,\bl w)\rangle= \langle f, \ov{\det \Lambda(\bl w)}K_j^{(2)}(\cdot ,\bl w)\rangle = {\det \Lambda_{j}f(\bl w)},
\end{equation*}
and thus
\bea\label{rp}
\langle f, K_j^{(2)}(\cdot ,\bl w)\rangle=\frac{\det \Lambda_{j}f(\bl w)}{\det \Lambda(\bl w)}=f_j(\bl w)=\tilde f_j(\bl v) = \langle \tilde{f}(\bl v), \varepsilon_j\rangle_{\C^d}.
\eea
The equality before the penultimate one follows from the Lemma \ref{gen1}. Note that Equation \eqref{rp} holds when $\det \Lambda(\bl w)\neq 0$ and since $f_j$ exists on all of $\Omega,$ they are equal everywhere.
 Thus from Equation \eqref{repprop} and Equation \eqref{rp}, the reproducing property of $\mb K$ is proved. 
\end{proof}
\begin{rem}
\begin{enumerate}

\item The lemma above is equivalent to the fact that point evaluations are bounded on $\tilde{\m H}$, that is, there exists  $K>0$ such that  $\sum_{i = 1}^d|\tilde f_i(\bl u)|^2\leq K^2\|f\|_{\m H}^2$ for every $f\in\m H$ and $\bl u\in\bl\theta(\Omega)$. However, we were unable to obtain this seemingly simple inequality independently even in the particular case when $\Omega = \mb D^n, \,\m H=\mb A^2(\mb D^n)$ and $\bl \theta$ being the symmetrization map $\bl s$ whose $i$-th component $s_i$ is the $i$-th elementary symmetric polynomial in $n$ variables. The inequality in this case becomes:
\Bea
\sum_{i = 1}^d|\tilde f_i(\bl u)|^2\leq K^2\int_{\mb D^n}|f|^2dV
\Eea
for some $K>0$, $\bl u\in\bl s(\mb D^n)$ and $dV$ being the Lebesgue  measure. 

\item Another approach to obtain a reproducing kernel is via orthonormal basis. One can write $K(\bl z, \bl w) = \sum_{n\geq 0}e_n(\bl z)\ov{e_n(\bl w)}$ for some orthonormal basis $\{e_n\}_{n\geq 0}$ of $\m H$. Since each $e_n$ is a holomorphic function on $\Omega$, by analytic CST, we have $e_n(\bl z) = \sum_{i=1}^dp_i(\bl z)e_i^{(n)}(\bl z)$ and consequently,
\Bea
 K(\bl z, \bl w) = \sum_{n\geq 0} \begin{pmatrix} p_1(\bl z)  &  p_2(\bl z) &\ldots & p_d(\bl z)\end{pmatrix} \big(\!\!\big(e_i^{(n)}(\bl z)\ov{e_j^{(n)}(\bl w)}\big)\!\!\big)_{i,j=1}^d\begin{pmatrix}
\ov{p_1(\bl w)}\\
\ov{p_2(\bl w)}\\
\vdots\\
\ov{p_d(\bl w)}
\end{pmatrix}.
\Eea
However, it is not apparently clear  whether the sum $\sum_{n\geq 0} e_i^{(n)}(\bl z)\ov{e_j^{(n)}(\bl w)}$ converges pointwise for each $i, j.$ To avert this convergence problem, in the proof above, we applied analytic CST directly to $K(\bl z, \bl w).$ As a byproduct of the proof of Lemma \ref{cons}, we see  that $\Big(\!\!\Big(\sum_{n\geq 0} \big(e_i^{(n)}(\bl z)\ov{e_j^{(n)}(\bl w)}\big)\Big)\!\!\Big)_{i,j=1}^d$ converges pointwise to the reproducing kernel of $\wi{\m H}$ from Equation \eqref{unique0}.
\item From the proof above, a Kolmogorov decomposition of $\mb K$ is apparent, that is, $\mb K(\bl u, \bl v) = \mb F(\bl u)\mb F(\bl v)^*,$ where $\mb F(\bl u) =\big (\widetilde G_1(\bl u),\ldots, \widetilde G_d(\bl u)\big)^{\mathrm tr}$ such that $G_i(\bl z) = \widetilde G_i\big(\bl\theta(\bl z)\big)$ for $1\leq i\leq d$. Such $\widetilde G_i$ exists because of the definition of $G_i$ and analytic CST.
\end{enumerate}
\end{rem}

Recall that $\m H\subseteq\mathcal O(\Omega)$ is an analytic Hilbert module over $\mb C[\bl z]$ and the module action of pointwise multiplication is denoted by ${\mathfrak m}_p$ for $p \in \C[\bl z].$ As noted earlier, $\mathcal H$ also forms a Hilbert module over the ring $\C[\bl z]^G$ with ${\mathfrak m}_p$ as the module action, that is, module action inherited from $\C[\bl z].$
The following lemma asserts that the module action induced by $\Gamma$ makes $\wi{\m H}$ an analytic Hilbert module.

\begin{lem}
The reproducing kernel Hilbert space $\wi{\m H}\subseteq\mathcal O(\bl \theta(\Omega))\otimes \C^{|G|}$ is an analytic Hilbert module over $\C[\bl z]$.
%for $\varrho \in \widehat{G}.$ 
\end{lem}
\begin{proof}
By construction, $\wi{\m H}$ consists of $\mb C^{d}$-valued holomorphic functions on $\bl \theta(\Omega),$ where $d = \lvert G\rvert$. Since $\C[\bl z]$ is dense in $\mathcal H$ and $\Gamma$ is indeed a unitary onto $\wi{\m H}$, so  $\Gamma(\C[\bl z]) = \C[\bl z] \otimes \C^d$ is also dense in $\wi{\mathcal H}.$  Let $\tilde f\in {\wi{\m H}}$, then there exists $f \in \m H$ such that $\Gamma(f) =\tilde{f}$. Then for $p\in \C[\bl z],\, \Gamma\big(( p\circ \bl \theta)\cdot f\big) = \sum_{i=1}^d (p\tilde{f_i}) \otimes \varepsilon_i \in \wi{\m H}$. Define $\wi{T}_p : \wi{\m H} \to \wi{\m H}$ by
\bea\label{modmap}
\wi T_p\tilde f=\wi{T}_p(\sum_{i=1}^d \tilde{f_i} \otimes \varepsilon_i) = \sum_{i=1}^d (p\tilde{f_i}) \otimes \varepsilon_i 
\,\,\text{for}\,\, p\in\C[\bl z].
\eea
By the closed graph theorem, $\wi T_p$ is a bounded linear map.  In fact, as $\Gamma$ is a unitary onto $\wi{\m H}$,  the equality 
\Bea
\Vert{\wi{T}_p\tilde{f}}\Vert= \Vert\Gamma \big(( p\circ \bl \theta)\cdot f\big) \Vert= \Vert {( p\circ \bl \theta)\cdot f}\Vert=\Vert{\mathfrak{m}_{p\circ \bl \theta}f}\Vert
\Eea
implies that $\|\wi{T}_p\|= \|\mathfrak{m}_{p\circ \bl \theta}\|$. Therefore, $\wi{T}_p$ indeed gives a module action on $\widetilde{\m H}.$ 
This completes the proof.
\end{proof}

%Since the ring of polynomials in $n$ variables is dense in $\mathcal H$ and $\Gamma$ is indeed an isometry, so $\C[\theta_1,\ldots,\theta_n] \otimes \C^d$ is also dense in $\wi{\mathcal H}.$ By Lemma \ref{cons}, $\wi{\m H}$ is a reproducing kernel Hilbert space consisting of holomorphic functions on $\bl \theta(\Omega)$ whose reproducing kernel is doubly $G$-invariant. Hence $\wi{\mathcal H}$ is an analytic Hilbert module on $\bl \theta (\Omega)$ over $\C[z_1,\ldots,z_n]^G.$
%$ Moreover, $\theta(\Omega)$ is a $G$-invariant domain. 
%Theorem \ref{A} ensures that for any $G$-invariant polynomial $h$ there exists a unique polynomial $\tilde{h}$ such that $h = \tilde{h} \circ \bl \theta$.
%The module action in $\wi{\m H}$ over $\C[z_1,\ldots,z_n]^G$ is given by $\wi{T}_h$ where
%and the choice of $\tilde{h}$ is unique. Then 
%\bea\label{modmap}
%\wi{T}_h \begin{pmatrix}
%\tilde{f_1}\\
%\tilde{f_2}\\
%\vdots\\
%\tilde{f_d}
%\end{pmatrix}
%= 
%\begin{pmatrix}
%&\tilde{h} \tilde{f_1}\\
%&\tilde{h} \tilde{f_2}\\
%&\vdots\\
%&\tilde{h} \tilde{f_d}
%\end{pmatrix}.
%\eea
%[$\tilde{T}_h$ is well defined.] 
%for 
%$ \tilde{f} = \begin{pmatrix}
%\tilde{f_1}\\
%\vdots\\
%\tilde{f_d}
%\end{pmatrix} \in \wi{\m H}.$ Equivalently, $\wi{T}_h\tilde{f} = \Gamma(hf).$ Recall that $\Gamma$ is an isometry, so $\pnorm{\wi{T}_h\tilde{f}}= \pnorm{\Gamma(hf) }= \norm {hf}=\norm{\mathfrak{m}_hf}.$ Now $\m H$ is a Hilbert module over $\C[z_1,\ldots,z_n]^G,$ so $\mathfrak{m}_h$ is bounded and so is $\wi{T}_h$. Therefore, $\wi{T}_h$ indeed gives a module action in $\widetilde{\m H}.$ 

With the aid of the lemma above, we realize  the Hilbert module $\m H$ over $\mb C[\bl z]^G$ as an analytic Hilbert module over $\mb C[\bl z].$

\begin{cor}\label{modtheta}
 The Hilbert module $\mathcal H$ over $\C[\bl z]^G$ is unitarily  equivalent to an analytic Hilbert module $\wi{\mathcal H}$ over $\C[\bl z]$. 
 \end{cor}
 \begin{proof}
 From the Equation \eqref{modmap}, 
 \bea\label{thetaiequi}
 \wi{T}_p\Gamma f=\wi{T}_p\tilde{f} = \Gamma\big(( p \circ \bl \theta)\cdot f\big)= \Gamma(\mathfrak{m}_{p \circ \bl \theta}f)\,\, \text{for}\,\, p \in \C[\bl z].
 \eea
 Since $\Gamma : \m H \to \wi{\m H}$ is a unitary operator,  the proof is complete.
 %Suppose for some $f \in \m H,$ $\Gamma f =\tilde{f}.$ By Equation \eqref{modmap}, 
 %\bea\label{thetaiequi}
 %\wi{T}_h\Gamma f=\wi{T}_h\tilde{f} = \Gamma(hf)= \Gamma\mathfrak{m}_hf,
 %\eea
 %for $h \in \C[z_1,\ldots,z_n]^G$ and $\Gamma : \m H \to \wi{\m H}$ is unitary. Hence the proof is complete.
 \end{proof}
 \begin{rem}\label{main2}
Denote the $i$-th coordinate projection on a domain $D$ in $\mb C^n$ by $Z_i:D\to\mb C,$ where $Z_i(\bl z)=z_i$  for $\bl z\in D.$ We define the multiplication operator $M_i : \wi{\m H} \to \wi{\m H}$ by $$M_i \tilde{f} = 
\sum_{j=1}^d (Z_i\tilde{f_j}) \otimes \varepsilon_j, \, \,i=1,\ldots,n.$$ Evidently, $ M_i \tilde{f}= \wi{T}_{Z_i} \tilde{f}.$ 
Since $\theta_i = Z_i \circ \bl \theta,$ taking $p =Z_i$ in Equation \eqref{thetaiequi}, we get
\Bea
\Gamma M_{\theta_i}f =  M_{i} \Gamma f ~~ {\rm for~~ all }~~ f \in \m H,
\Eea
where $M_{\theta_i}$ is the multiplication operator on $\m H$.
Let  $\mathbf{M} = (M_1,\ldots,M_n)$ and $\mathbf M_{\bl \theta} = (M_{\theta_1},\ldots,M_{\theta_n}).$ The corollary above says that the operator tuple $\mathbf M_{\bl \theta}$ on $\mathcal H$ is unitarily equivalent to $\mathbf M$ on $\wi{\mathcal H}$. 
\end{rem}

We now describe analogous results for the reducing submodules. Recall from \eqref{decomp} that $\m H$ admits an orthogonal decomposition $\m H = \bigoplus_{\varrho \in \widehat{G}} \mb P_\varrho \m H.$ Clearly, $\C[\bl z] = \bigoplus_{\varrho \in \widehat{G}} \mathbb{P}_{\varrho}\C[\bl z]$ where the direct sum is orthogonal direct sum borrowed from the Hilbert space structure mentioned earlier. From \cite[Proposition II.5.3., p.28]{Pan}, it follows that for each $\varrho\in \widehat G$, $\mathbb{P}_{\varrho}\C[\bl z]$ is a free module over $\C[\bl z]^G$ of rank $(\deg \varrho)^2$.
 Fix $\varrho \in \widehat{G}$, define $\Gamma_{\varrho} : \mathcal H \to \m O\big(\bl \theta(\Omega)\big) \otimes \C^{\lvert G\rvert}$ by $\Gamma_{\varrho}(f) :=  \Gamma \mb P_{\varrho} f$ and denote the image of $\Gamma_{\varrho} \m H $ by $ \wi{\m H}_{\varrho}.$ Further, we observe that the restriction operator $\Gamma_\varrho : \mb P_\varrho\m H \to \wi{\m H}_{\varrho}$ is a unitary, by construction. Therefore, $\wi{\m H}_{\varrho}$ is a closed subspace of $\wi{\m H}$ and the following corollary is immediate from Lemma \ref{cons}.
%Since $\mb P_\varrho\m H$ is a reproducing kernel Hilbert space, we have the following Corollary.

\begin{cor}
For $\varrho\in\w G,~ \wi{\m H}_{\varrho}$ is a reproducing kernel Hilbert space. 
\end{cor}  
\begin{cor}\label{equi}
For each $\varrho \in \widehat{G},$ the Hilbert module  $\mathbb{P}_{\varrho}\mathcal H$ over $\C[\bl z]^G$ is equivalent to the Hilbert module $\wi{\mathcal H}_{\varrho}$ over $\C[\bl z].$ 
%Consequently, the Hilbert module $\wi{\mathcal H}_{\varrho}$ is also of rank $\chi_{\varrho}(1)^2$ over the ring $\C[\bl z].$
\end{cor}
\begin{proof}
We observe that $\wi{\m H}_{\varrho}$ forms a Hilbert module over $\C[\bl z],$ where the module map is same as defined in Equation \eqref{modmap}. By construction, $\Gamma_{\varrho} : \mathbb{P}_{\varrho}\mathcal H \to \wi{\m H}_{\varrho}$ is a unitary operator. Moreover, 
\Bea
\wi{T}_p \Gamma_{\varrho}(f) = \wi T_p(\tilde{f})=\Gamma_\varrho\big((p \circ \bl \theta)\cdot f\big)=\Gamma_\varrho (\mathfrak{m}_{p\circ \bl \theta}f).
\Eea
This shows that $\Gamma_{\varrho}$ intertwines the module actions in $\mathbb{P}_{\varrho}\mathcal H$ and $\wi{\mathcal H}_{\varrho}$.
%Since the rank of a Hilbert module is an unitary invariant, the second statement follows from Lemma \ref{modrank}.
\end{proof}

Since $\sum_{\varrho \in \widehat G}(\deg \varrho)^2 =\lvert G\rvert= d,$ we can choose a basis of $\mathbb{P}_{\varrho}\C[\bl z]$ over $\C[\bl z]^G$ for each $\varrho \in \widehat G$ such that together they form a basis $\{p_1,\ldots,p_d\}$ of $\C[\bl z]$ over $\C[\bl z]^G$. We will work with such a choice for the rest of the section.  We enumerate $\widehat{G}$ as $\{\varrho_1,\ldots,\varrho_\ell\}.$  Since for $f \in \mb P_{\varrho^\i} \m H$ and $\varrho \neq \varrho^\i,$ $\Gamma_\varrho f = 0$, an element of $\wi{\m H}_{\varrho_k}$ is in the form 
$$\begin{pmatrix}
\bl 0\\
\tilde{f_1}\\
\vdots\\
\tilde f_{( {\rm deg}\,\varrho_k)^2}\\
\bl 0
\end{pmatrix}$$ 
where $\tilde{f_1}$ is at the $\sum_{i=1}^{k-1}( {\rm deg}\,\varrho_i)^2 +1$-th coordinate of $\C^d.$ The corresponding reproducing kernel $\mb K_{\varrho_k}$ can be written as a block matrix of whose $ij$-th block is a zero matrix of size ${({\deg}\,\varrho_i)^2} \times ({\deg}\,\varrho_j)^2$ except the $kk$-th block, say ${\mathfrak K}_{\varrho_k}$, which is the only nonzero block (follows from Corollary \ref{projPii}) of size ${({\deg}\,\varrho_k)^2} \times ({\deg}\,\varrho_k)^2$.  Since $\Gamma = \sum_{i = 1}^\ell \Gamma_{\varrho_i}$, we have $\wi{\m H} = \oplus_{i = 1}^\ell{\wi{\m H}}_{\varrho_i}$ and hence, 
\bea\label{reprobreak}
\mb K(\bl u, \bl v) 
= \sum_{i = 1}^\ell \mb K_{\varrho_i}(\bl u, \bl v) 
=
\begin{pmatrix}
&{\mathfrak K}_{\varrho_1} &\bl 0 &\cdots &\bl 0\\
&\bl 0 & {\mathfrak K}_{\varrho_2} &\cdots &\bl 0\\
&\vdots &\vdots &\ddots  &\vdots\\
&\bl 0 &\bl 0 &\cdots & {\mathfrak K}_{\varrho_{\ell}}
\end{pmatrix},
\eea
that is, the reproducing kernel of $\widetilde{\mathcal H}$ breaks into diagonal blocks where the $j$-th block corresponds to the reproducing kernel of $\wi{\m H}_{\varrho_j}$ for $j=1,\ldots,\ell,$ by identifying
\Bea
\begin{pmatrix}
\bl 0\\
\tilde{f_1}\\
\vdots\\
\tilde f_{(\deg\varrho)^2}\\
\bl 0
\end{pmatrix} \mbox{~with~} \begin{pmatrix}
\tilde{f_1}\\
\vdots\\
\tilde f_{(\deg\varrho)^2}
\end{pmatrix}.
\Eea
We  summarize  these discussions as follows:
\begin{cor}\label{t2}
For $\varrho\in\widehat G,$ the Hilbert module $\mb P_\varrho \m H$ over $\C[\bl z]^G$ is equivalent to an analytic Hilbert module consisting of $\mb C^{(\deg\varrho)^2}$-valued holomorphic functions on $\bl \theta(\Omega)$ over $\C[\bl z]$ with reproducing kernel ${\mathfrak K}_{\varrho}$.
\end{cor}

We close this section by relating ${\mathfrak K}_{\varrho}$ to the reproducing kernel $K$ of $\m H$ for a specific $\varrho\in \widehat G$. Note that if $K_\varrho$ denotes the reproducing kernel of the submodule $\mb P_\varrho \m H$, $\varrho\in \widehat G$, then it is given by the formula:
\bea\label{rksub}
K_\varrho(\bl z, \bl w) = \langle \mb P_\varrho K_{\bl w}, K_{\bl z}\rangle = \frac{\deg \varrho}{|G|} \sum_{\sigma \in G}\ov{\chi_\varrho(\sigma)}  K(\sigma^{-1} \cdot \bl z, \bl w ).
\eea
In particular, if $\varrho$ is a one dimensional representation, then  $\mb P_\varrho \C[\bl z]$ is freely finitely generated of rank $1$ over $\C[\bl z]^G$ and if $\{p\}$ is basis, then by the proof of Lemma \ref{cons} and Corollary \ref{t2} and the discussion before it, we have 
\bea\label{onegen}
K_\varrho(\bl z, \bl w) = p(\bl z) {\mathfrak K}_\varrho\big(\bl\theta(\bl z), \bl\theta(\bl w)\big)\ov{p(\bl w)}.
\eea 
Further, from \cite[p.2-3]{G} it follows that the relative invariant space (see \cite{RISt}) in $\m H$ is the corresponding reducing submodule, that is,
\Bea
R^G_{\varrho}(\m H) := \{f \in \m H : f \circ \sigma^{-1} = \chi_\varrho(\sigma) f \mathrm {~for ~ all~} \sigma \in G \} = \mb P_\varrho \m H.
\Eea
Consider the canonical one dimensional representation $\mu$ of any finite pseudoreflection group $G$ given by $\mu (\sigma) = \big(\det(\sigma)\big)^{-1}$, $\sigma \in G$ \cite[p.139, Remark (1)]{RISt}. From \cite[Lemma 4.4]{G}, \cite[p.616]{St} and the equation above, one concludes that the jacobian $J_{\bl \theta}$ of the map $\bl \theta$ is a generator of the submodule $\mb P_\mu \m H$ over $\C[\bl z]^G$. Thus, in the following proposition, we have the generalization of \cite[Proposition 4.9]{G} from Bergman space to the case of arbitrary $G$-invariant reproducing kernel Hilbert space.

\begin{prop}
For $\bl z, \bl w \in \Omega$, 
$$ 
{\mathfrak K}_\mu\big(\bl\theta(\bl z), \bl\theta(\bl w)\big) = \frac{1}{|G|}  \big(J_{\bl \theta}(\bl z)\big)^{-1}\big(\sum_{\sigma \in G}\ov{\chi_\mu(\sigma)}  K(\sigma^{-1} \cdot \bl z, \bl w )\big)\big(\ov{ J_{\bl \theta}(\bl w)}\big)^{-1}.
$$
\end{prop}

\subsection[Unitary equivalence]{On unitary equivalence of reducing submodules}

We obtain the following results towards classification of the family of reducing submodules
   \Bea
    \{\mb P_\varrho^{ii}\m H:\varrho\in\w G, 1\leq i\leq {\rm deg}\,\varrho\}
    \Eea
    over $\mb C[\bl z]^G$ under unitary equivalence.
%the family of reducing submodules $\{\mb P_{\varrho}^{ii}\m H: i=1,\ldots, \chi_\varrho(1)\}_{\varrho \in \widehat{G}}$ under unitary equivalence, partially.
\begin{thm}\label{inequihilmod}
Let $\m H$ be an analytic Hilbert module with a  $G$-invariant kernel. If $\varrho$ and $\varrho^\i\in\w G$ are such that $ {\rm deg}\,{\varrho}\neq  {\rm deg}\,{\varrho^\i},$ then
\begin{enumerate}
\item[$(1)$] $\mb P_{\varrho}\mathcal H$ and $\mb P_{\varrho^\i}\mathcal H$ are not unitarily equivalent submodules over the ring of $G$-invariant polynomials, and
\item[$(2)$] the submodules  $\mb P_{\varrho}^{ii}\m H$ and $\mb P_{\varrho^\i}^{jj}\m H$ are not unitarily equivalent for any $i,j$ with $1\leq i \leq  {\rm deg}\,{\varrho}$ and $1\leq j \leq  {\rm deg}\,{\varrho^\i}.$
\item[$(3)$] However, the submodules $\mb P_{\varrho}^{ii}\m H$ and $\mb P_{\varrho}^{jj}\m H$ are unitarily equivalent for any $\varrho\in\w G$ and $i,j$ with $1\leq i,j \leq  {\rm deg}\,{\varrho}.$
\end{enumerate}
\end{thm}
%On the contrary, we have equivalence for the submodules $\mb{P}_{\varrho}^{jj } \mathcal H$ and $\mb{P}_{\varrho}^{ii }\mathcal H$.
The following corollary is immediate.
 \begin{cor}\label{ub}
 Suppose that $\m H$ is an analytic Hilbert module over $\mb C[\bl z]$ and  each member of the family $\{\mb P_\varrho^{ii}\m H:\varrho\in\w G,\,\, 1\leq i\leq  {\rm deg}\,\varrho\}$ of  reducing submodules of $\m H$ over $\mb C[\bl z]^G$ is minimal. Then the number of inequivalent reducing submodules from this family cannot exceed $\lvert\w G\rvert.$
  \end{cor}

The theorem above is a generalization of the \cite[Theorem 3.14]{BGMS} from $\mathfrak S_n$ to an arbitrary finite pseudoreflection group. In \cite[Theorem 3.14]{BGMS}, it is shown that the dimension of joint kernel away from the set of critical points of the symmetrization map $\bl s$ was used to determine the equivalence in the case of $G = \mathfrak S_n$. We show here that the realization as an analytic Hilbert module enables us to determine the dimension of the joint kernel on all of $\Omega$ leading to the proof of the Theorem \ref{inequihilmod}. In fact, we are going to show that the dimension of the joint kernel is constant over $\Omega$ and is equal to the rank of the Hilbert module $\m H$ over $\C[\bl z]^G.$ Note that these observations improve the earlier results as we are able to compute the dimension and rank without deleting the  set of critical points of the map $\bl\theta$. 
%The following Lemma is another immediate consequence of  Proposition \ref{equi}.

To prove Theorem \ref{inequihilmod}, we note that Lemma \ref{uni} ensures the existence of a unitary operator $\mb P_
 \varrho^{ij}$ between any $\mb P_\varrho^{jj}\m H$ and $\mb P_\varrho^{ii}\m H$ and by Proposition \ref{rhojredsub}, $\mb P_\varrho^{ij}$  intertwines the module actions in submodules $\mb P_\varrho^{jj}\m H$ and $\mb P_\varrho^{ii}\m H$ over $\mb C[\bl z]^G$.  Thus, $\mb P_\varrho^{ii}\m H$ and $\mb P_\varrho^{jj}\m H$ are unitarily equivalent submodules over $\mb C[\bl z]^G$. 
%We record it as following theorem.
%Combining Lemma \ref{uni}, Proposition \ref{fd1} and Proposition \ref{reduce1}, part (3) of the Theorem follows. 
However, we need further tools to prove part (1) and (2). As in \cite[Corollary 3.7]{BGMS}, the following theorem estimates the rank (see \ref{rank}) of the Hilbert module $\m H$ and $\mb P_{\varrho}\mathcal H$ over $\mb C[\bl z]^G.$

\begin{thm}\label{rankupbd}
Let $G$ be a finite pseudoreflection group and $\Omega \subseteq \C^n$ be a $G$-invariant domain and let $\m H$ be an analytic Hilbert module on $\Omega$ over $\C[\bl z]$ with $G$-invariant reproducing kernel. The rank of the Hilbert module $\mathcal H$  over $\C[\bl z]^G$ is at most $\lvert G \rvert$ and the rank of the Hilbert module $ \mb P_{\varrho}\mathcal H$ over $\C[\bl z]^G$ is at most $( {\rm deg}\,\varrho)^2.$
\end{thm}
\begin{proof}
Suppose $|G|=d.$ By CST Theorem, we know that there exists a subset $\mathcal F = \{f_i\}_{i=1}^d \subset \C[\bl z]$ such that $\C[\bl z]=\{\sum_{i=1}^d r_if_i : f_i \in \mathcal F, r_i \in \C[\bl z]^G\}.$ Since $\C[\bl z]$ is dense in $\m H,$ ${\rm rank}_{\C[\bl z]^G}~ \m H \leq |\mathcal F| =d.$ On the other hand, from the facts that $\mathbb{P}_{\varrho}\C[\bl z]$ is a free module over $\C[\bl z]^G$ of rank $( {\rm deg}\,\varrho)^2$  \cite[Proposition II.5.3., p.28]{Pan} and $\mathbb{P}_{\varrho}\C[\bl z]$ is dense in $\mb P_\varrho \m H,$ we similarly have ${\rm rank}_{\C[\bl z]^G}~\mb P_{\varrho}\mathcal H\leq ( {\rm deg}\,\varrho)^2.$ 
\end{proof}
We denote the compression operator $M_{\theta_i}^* \big|_{\mb P_{\varrho} \mathcal H}$ by $ {M_{\theta_i}^{(\varrho)}}^*.$
 \begin{lem}\label{piecedim}
 %Let $G$ be a finite pseudoreflection group and $\Omega$ be a $G$-invariant domain in $\C^n.$ The  If $\m H$ is a Hilbert module over $\mb C[\bl z]$ consisting of holomorphic functions defined on some bounded domain $\Omega \subseteq \mb C^n,$ and $\mathbf{M}_{\bl theta}$ is the $n$-tuple multiplication operators defined on $\m H$ then we have
 For every $\bl w \in \Omega$ and for every $\bl v\in\bl\theta(\Omega)$, we have the following:
 \begin{enumerate}
 \item[(i)] $ \mathrm{rank}_{\C[\bl z]}~\wi{\m H} = \dim \bigcap_{i=1}^n  \ker\big( M^*_i - \ov{\bl v_i} \big)= \lvert G\rvert$;
     \item[(ii)] ${\rm rank}_{\C[\bl z]^G}~ \m H = \dim\bigcap_{i=1}^n \ker\big(M^*_{\theta_i} - \ov{\theta_i(\bl w)}\big) = \lvert G\rvert ,$ and
     \item[(iii)] for every $\varrho \in \widehat{G}$, 
     $${\rm rank}_{\C[\bl z]}~{\wi{\m H}}_{\varrho} = {\rm rank}_{\C[\bl z]^G}~\mb P_{\varrho}\mathcal H = \dim\bigcap_{i=1}^n \ker\big({M_{\theta_i}^{(\varrho)}}^* - \ov{\theta_i(\bl w)}\big) = ( {\rm deg}\,\varrho)^2.$$ 
 \end{enumerate}
  \end{lem}
 \begin{proof}
 Suppose that $\lvert G\rvert=d$ and $\{\varepsilon_1,\ldots,\varepsilon_d\}$ is the standard basis of $\C^d.$  Considering $\varepsilon_i$ as the constant function on $\bl \theta(\Omega)$ for $i=1,\ldots, d,$ and using the reproducing property of $\mb K$, one can show that $\sum_{i=1}^d \alpha_i \mb K_{\bl v} \varepsilon_i = 0 $ implies $\overline{\alpha_j} =0 $ for all $j=1,\ldots,d.$ Equivalently, the set $\{\mb K_{\bl v} \varepsilon_i\}_{i=1}^d$ is linearly independent.  Hence 
$ \dim \bigcap_{i=1}^n  \ker\big( M^*_i - \ov{\bl v_i} \big)\geq d.$ Since $\C[\bl z]\otimes \C^d$ is dense in $\wi{\m H}$, from Lemma \ref{rankcomp}, we have
$$
d\leq \dim \bigcap_{i=1}^n  \ker\big( M^*_i - \ov{\bl v_i} \big) \leq \mathrm{rank}_{\C[\bl z]}~\wi{\m H}\leq d.
$$
This completes the proof of $(i)$. Since $\mathbf{M}$ and $\mathbf M_{\bl \theta}$ are unitarily equivalent, as noted in Remark \ref{main2}, part $(ii)$ follows from part $(i)$.  Similarly, using Corollary \ref{equi} and \ref{t2}, part $(iii)$ follows from part $(i)$. 
\end{proof}
We now prove part (1) and (2) of Theorem \ref{inequihilmod}. 

\begin{proof}[\textbf{Proof of Theorem \ref{inequihilmod}}] Part (1) of the Theorem follows from the lemma above. To prove Part (2),
set the compression $M_{\theta_i} \big|_{\mb P_{\varrho}^{jj} \mathcal H} = {M_{\theta_i}^{(\varrho,j)}}$ as from Proposition \ref{fd1}, it follows that $\mb P_{\varrho}^{jj} \mathcal H$ is a joint reducing subspace for the $n$-tuple of operators $(M_{\theta_1},\ldots,M_{\theta_n}),$ where $j=1,\ldots, {\rm deg}\,\varrho.$ 
We know from Lemma \ref{uni} that for a fixed $\varrho \in \widehat{G}$ and $1 \leq \ell ,j \leq  {\rm deg}\,\varrho,$ there exists a unitary $\mb P_\varrho^{j \ell} : \mb P^{\ell\ell}_\varrho \m H \to \mb P_{\varrho}^{jj} \m H$ that intertwines ${M_{\theta_i}^{(\varrho,\ell)}}$ and ${M_{\theta_i}^{(\varrho,j)}}$ for every $i=1,\ldots,n.$ 
Then, for a fixed $\bl w \in \Omega,$ the unitary $({\mb P_\varrho^{j \ell}})^*$ maps $\cap_{i=1}^n \ker\big({M_{\theta_i}^{(\varrho,j)}}^* - \ov{\theta_i(\bl w)}\big)$ onto $\cap_{i=1}^n \ker\big({M_{\theta_i}^{(\varrho,\ell)}}^* - \ov{\theta_i(\bl w)}\big).$ Hence,
\Bea
\dim\cap_{i=1}^n \ker\big({M_{\theta_i}^{(\varrho,j)}}^* - \ov{\theta_i(\bl w)}\big)=\dim\cap_{i=1}^n \ker\big({M_{\theta_i}^{(\varrho,\ell)}}^* - \ov{\theta_i(\bl w)}\big),
\Eea
and we conclude that the dimension of $\bigcap_{i=1}^n \ker \big({M_{\theta_i}^{(\varrho,j)}}^* - \ov{\theta_i(\bl w)}\big)$ is equal for all $1 \leq j \leq  {\rm deg}\,\varrho,$  say $m.$
 
 From Equation \eqref{rhoiisum} and Corollary \ref{rhoortho}, we get $\mb P_{\varrho} \m H= \oplus_{i=1}^{ {\rm deg}\,\varrho} \mb P_{\varrho}^{ii} \m H.$
Subsequently,
\Bea
\bigcap_{i=1}^n \ker\big({M_{\theta_i}^{(\varrho)}}^* - \ov{\theta_i(\bl w)}\big) = \bigoplus_{j=1}^{ {\rm deg}\,\varrho} \bigcap_{i=1}^n \ker\big({M_{\theta_i}^{(\varrho,j)}}^* - \ov{\theta_i(\bl w)}\big) \text{~for all~} \varrho.
\Eea
This implies that $( {\rm deg}\,\varrho)^2 = m \cdot  {\rm deg}\,\varrho,$ that is, $m= {\rm deg}\,\varrho.$ Therefore, by an analogous argument as above, there does not exist any unitary operator from $\mb P_\varrho^{jj}\m H$ to $\mb P_{\varrho^\i}^{jj}\m H$ that intertwines ${M_{\theta_i}^{(\varrho,j)}}^*$ and ${M_{\theta_i}^{(\varrho^\i,j)}}^*,$ whenever $ {\rm deg}\,\varrho \neq  {\rm deg}\,{\varrho^\i}.$ Hence part (2) follows.
\end{proof}
\noindent Combining these results with Corollary \ref{equi}, we have the following corollary.
\begin{cor}
If $ {\rm deg}\,\varrho \neq  {\rm deg}\,\varrho^\i,$ then 
$\wi{\mathcal H}_{\varrho}$ and $\wi{\mathcal H}_{\varrho^\i}$ are not equivalent as Hilbert  modules over $\C[\bl z].$ 
\end{cor}

\section{Application to multiplication operators by proper maps}\label{last}

%In this section, we apply techniques developed in this paper to characterize the reducing subspaces of multiplication operators  by proper holomorphic mappings on reproducing kernel Hilbert spaces consisting of holomorphic functions.

%In additoin to obtaining a lower bound on the number of reducing subspaces we realize them as analytic Hilbert modules over a suitable polynomial algebra. We find joint reducing submodules by finding projections in the joint commutant of the relevant tuple of multiplication operators.  Analogous problems were studied extensively in \cite{DPW,DSZ, GHu, GHu1, HZ, JT, T}. 

%We recall two definitions from \cite[p. 702]{R}.Suppose that $p:D_1\to D_2$ is an analytic cover and $G$ is a subgroup of $\mr{Aut}(D_1).$ 
%\begin{enumerate}[leftmargin=*]
%\item [1.] We call $p$ to be \emph{$G$-invariant} if $p \circ U = p$ for all $U \in G.$ Every cover $p$ is clearly ${\rm Deck}(p)$-invariant. Recall that ${\rm Deck}(p)=\{h\in {\rm Aut}(D_1):p\circ h=p\}.$
%\item[2.] The cover $p$ is  \emph{precisely $G$-invariant} means that $p(z) = p(w)$ for some $z, w\in D_1$ if and only if  $Uz =w$ for some $U \in G.$ Equivalently, for every $w\in D_2,$ the set $p^{-1}(\{w\})$ is a $G$-orbit, and vice versa. Briefly, $D_2=D_1/G.$ That is, $p$ is precisely $G$-invariant if and only if $G=\mr{Deck}(p).$ In other words, a cover $p$ is Galois if and only if it is precisely ${\rm Deck}(p)$-invariant.
%\end{enumerate}
In this section, we describe the joint reducing subspaces  of a tuple of (bounded) multiplication operators by a  representative of a proper holomorphic map on a bounded domain $\Omega$  in $\mb C^n$ which is factored by automorphisms $G\subseteq {\rm Aut}(\Omega)$ on a Hilbert space consisting of holomorphic functions with $G$-invariant reproducing kernel. We recall the notions of {\it representative} and {\it factored by automorphisms} from \cite{BD}, see also \cite{DS}.
\begin{enumerate}[leftmargin=*]
\item [1.] Consider a proper holomorphic mapping $\bl f:D_1\to D_2,$ where $D_1, D_2$  are domains in $\mb C^n.$  By a {\it representative} of $\bl f,$ we mean a complex manifold $\w D$ and a proper mapping $\hat{\bl  f}:D_1\to\w D$ with the properties:
\begin{enumerate}
\item [(i)]$\hat{\bl f}(D_1)=\w D,$
\item[(ii)] $\hat {\bl f}(\bl z_1)=\hat {\bl f}(\bl z_2)$ if and only if $\bl f(\bl z_1)=\bl f(\bl z_2).$
\end{enumerate}
Given any $\hat {\bl f}$ with these properties, there exists a biholomorphism $\bl h: \hat D \ra D_2$ such that $\bl f = \bl h\circ\hat 
{\bl f}.$ Thus any representative $(\hat {\bl f}, \hat D)$ determines $(\bl f, D_2)$ up to biholomorphism.
\vspace{.1 cm}
\item[2.] If $\bl f:D_1\to D_2$ is {\it factored by automorphisms}, that is, if there exists a finite subgroup $G\subseteq {\rm  Aut}(D_1)$ such that
\Bea
\bl f^{-1}\bl f(\bl z)=\bigcup_{\rho\in G} \rho(\bl z) \text{~for all~}\bl  z\in D_1,
\Eea 
then $\bl f$ is $G$-invariant ($\bl f\circ \rho=\bl f$ for $\rho\in G$). In \cite[p.702]{R}, the notion of $\bl f$ factoring by automorphisms is referred as $\bl f$ is precisely $G$-invariant. In this case, $\bl f = \tilde {\bl f}\circ\bl\eta,$ where $\bl\eta:D_1\to D_1/G$ is a quotient map and $\tilde {\bl f}:D_1/G\to D_2$ is a biholomorphism.  Therefore, $G\cong\mr{Deck} (\bl f)$ and the quotient map 
$\bl \eta:D_1\to D_1/G$
gives a representative $(\hat {\bl f}=\bl\eta, \hat D=D_1/G)$ of $\bl f.$
\end{enumerate}
Before proceeding to the class of concrete examples, we make the following remark regarding the existence of factorization.
%\begin{rem}\rm
%Being a subgroup of the unitary group of $\mb C^n$ a finite pseudoreflection group $G$  acts biholomorphically on $\mb C^n$ and has the property that  $\lvert \det \rho \rvert =1$ for all $\rho\in G.$ Therefore, if $\Omega$ is a bounded $G$-invariant domain in $\mb C^n,$ it follows from the transformation rule for the Bergman kernel that the Bergman kernel of $\Omega$ is also $G$-invariant. Moreover, since $\Omega$ is bounded any bounded function on $\Omega$ induces a bounded multiplication operator on the Bergman space $\mb A^2(\Omega).$ Hence all the examples to follow of reproducing kernel Hilbert spaces includes the Bergman space and it is also an example of analytic Hilbert module on $\Omega,$ known as the Bergman module on $\Omega.$
%\end{rem}
\begin{rem}
A factorization does not always exist (see \cite[p.223]{DS}, \cite[p.711]{R}, \cite[p.168]{BD}). If $D_1, D_2 \subseteq \mb C^n, n\geq 2,$ and $D_1$ is a strongly pseudoconvex domain, then $\bl f:D_1\to D_2$ can be factored by automorphisms, in the following cases:
\begin{enumerate}[leftmargin=*]
 \item [1.] $D_1=\mb B_n,$ the open euclidean ball in $\mb C^n$ and in this case $G$ is conjugate to a pseudoreflection group \cite[Theorem 1.6, p.704]{R}, \cite[p.506]{BB}.
    \vspace{.1 cm}
    \item[2.] $D_1$ is bounded and simply connected with $C^2$-boundary and $\bl f\in C^\infty(\bar D_1)$ \cite{Bd} or with $C^\infty$-boundary \cite{BB}. In \cite[Theorem 2]{BB}, authors prove that if $M$ is a complex manifold of dimension $n,\, \Omega$ is a bounded, simply connected, strictly pseudoconvex domain in $\mb C^n$ with $C^\infty$ boundary and $\bl F:\Omega\to M$ is a proper holomorphic mapping, then there exists a finite group $G$ of automorphisms of $\Omega$ such that
    \Bea
    \bl F^{-1}\big(\bl F(\bl z)\big)=\{g(\bl z):g\in G\} \,\,\text{for}\,\, z\in\Omega.
    \Eea
    
    \item[3.] $D_1$ is a complexified sphere and in this case $G$  is a real reflection group \cite[Theorem 3.2]{BD}.
    \vspace{.1 cm}
    \item[4.] If $R$ is a Reinhardt domain in $\mb C^n,\, n>1,$ and $D$ is a bounded simply connected strictly pseudoconvex domain with $C^\infty$ boundary, then any  proper holomorphic mapping $\bl F:R\to D$   is factored by a finite pseudoreflection group $G\subseteq {\rm Aut}(R),$ see \cite[p.61]{DS1}.
    \vspace{.1 cm}
    \item[5.] A generalized pseudoellipsoid in $\mb C^n,\, n\geq 2,$ is a bounded Reinhardt domain 
  $\Sigma(\bl\a)$ defined by 
  \Bea
  \Sigma(\bl\a)=\Big\{\bl z\in\mb C^n:\sum_{j=1}^n\lvert z_j\rvert^{2\a_j}<1\Big\}, \text{~where~} \bl \a=(\a_1,\ldots,\a_n)\in\mb R_+^n,
\Eea
$\mb R_+$ being the set of positive real numbers \cite{DS}.
    In \cite[Theorem 2.5,  Theorem 2.6]{DS}, the authors find necessary and sufficient conditions for a proper holomorphic map  between two distinct generalized  pseudoellipsoids (that is,  a proper holomorphic map $\bl f:\Sigma(\bl\a)\to\Sigma(\bl\b)$) to factor by automorphisms $G\subseteq {\rm Aut}\big(\Sigma(\bl\a)\big)$. It is also shown in \cite[Lemma 2.2]{DS} that $G$ is conjugate in ${\rm Aut}\big(\Sigma(\bl\a)\big)$ to a finite pseudoreflection group.
    \item[6.] In \cite[p.18]{M}, Meschiari generalized \cite[Theorem 1.6, p.704]{R} to the case of bounded symmetric domains of classical type. He proved  that if $D_1$ is a bounded symmetric domain of classical type then $G$ is conjugate to a finite pseudoreflection group. 

\end{enumerate}
\end{rem}
\subsection{Example (Unit disc)}
A {\it finite Blaschke product} is a rational function of the form
\Bea
{\rm B}(z)=e^{i\theta}\displaystyle\prod_{j=1}^n\bigg(\frac{z-a_j}{1-\bar a_j z}\bigg)^{k_j},
\Eea
where $a_1,\ldots, a_n \in\mb D$ are the distinct zeros of $\rm B$ with multiplicities $k_1,\ldots,k_n,$ respectively and $\theta\in\mb R.$ The number of zeros (counting multiplicity) of ${\rm B}$ in $\mb D$ is called the {\it order} or {\it degree} of ${\rm B}.$
Finite Blaschke products are examples of proper holomorphic self-maps of the unit disc $\mb D.$ In fact,  they constitute the set of all possible proper holomorphic self-maps of the unit disc $\mb D,$ see  \cite[Remarks 3, p.142]{NN}.

The group M\"ob$(\mb D)$  of biholomorphic automorphisms of $\mb D$ is given by 
\Bea
\text{ M\"ob}(\mb D)=\{\v_{t,a}(z)=t\frac{z-a}{1-\ov a z}: t\in\mb T, a\in\mb D\},
\Eea
where $\mb T$ denotes the multiplicative group of complex numbers of unit modulus. We write it simply M\"ob. 
%M\"obius transformations $\phi_{t,a}(z) =t \frac{z-a}{1-\bar a z}$ where, $t \in \mb T$ and $a \in \mb D.$ We use the notation $\phi_a(z)$ for $\frac{z-a}{1-\bar a z},$ for simplicity. A finite Blaschke product is a finite product of M\"obius transformations. 
Let ${\rm B} : \mb D \to \mb D$ be a finite Blaschke product of order $m.$ It is well-known that ${\rm B}$ is an analytic (branched) covering map of degree $m.$ We denote the set of critical values of ${\rm B}$ in $\mb{D}$ by $ E = \{{\rm B}(c): {\rm B}'(c)= 0, c \in \mb D \}$ and  $F = {\rm B}^{-1}(E)$. The restriction  ${\rm B} : \mb D \smallsetminus F \to \mb D \smallsetminus E$ is an unbranched covering of degree $m.$ We know that 
%the group of deck transformations ${\rm Deck(B)}=\{\v\in\text{~M\"ob~}(\mb D):{\rm B}\circ\v={\rm B}\}$   is non-trivial if and only ${\rm B}$ is a non-trivial Galois covering (that is, ${\rm B}$ is not a automorphism), that is, if and only if ${\rm B}$ is factored by automorphisms. Equivalently, 
${{\rm B}}$ is a Galois covering if and only if it is factored by automorphism if and only if order of ${\rm Deck(B)}$ is equal to degree of ${{\rm B}}$. In general, ${\rm B}$ is not a Galois covering, see \cite[p.711]{R}. We exhibit examples in each of which ${\rm B}$ forms a Galois covering and we describe a family of reducing subspaces of the multiplication operator $M_{{\rm B}}$ on a family of reproducing  kernel Hilbert spaces which includes well-known function spaces  on the unit disc $\mb D.$
%Note that $M_{{\rm B}}$ is a bounded operator on $\mb A^2(\mb D).$
\vspace{0.1in}
\begin{enumerate}[leftmargin=*]
    \item Fix ${\rm B}(z) = z^m \text{ for } z \in \mb D, \, m\geq 2.$ Then  
    \Bea
    {\rm Deck(B)} = \{ \xi_k : \mb D \to \mb D : k=0,1\ldots, m-1\},
    \Eea
    where $\xi_k(z) = e^{\frac{2k \pi}{m}i} z.$ Hence ${\rm Deck(B)}$ is isomorphic to $\mb Z_m.$
    
Consider a sequence $\{a_n\}_{n=0}^\infty$ of non-negative real numbers such that 
    \bea\label{diagker}
    K(z,w)=\sum_{n=0}^\infty a_n(z\ov w)^n
    \eea
    is a positive definite kernel on $\mb D\times\mb D$ and $\m H_K$ is the Hilbert space with the reproducing kernel $K.$ Assume that $M_z$ (multiplication by the coordinate function) is a bounded operator on $\m H_K.$  In particular, if $\l>0$ and $a_n=\frac{(\l)_n}{n!},$ then 
\Bea
K(z,w)=(1-z\ov w)^{-\l}
\Eea
is a positive definite kernel and $M_z$ is bounded on $\m H_K,$ where $(\l)_n=\l(\l+1)\ldots(\l+n-1).$ The choices $\l=1,2$ correspond to the Hardy space $H^2(\mb D)$ and the Bergman space $\mb A^2(\mb D),$ respectively.  Also, if $a_n=\frac{1}{n+1},\, M_z$ is bounded on $\m H_K,$ the choice $\{\frac{1}{n+1}\}_{n=0}^\infty$ corresponds to the Dirichlet space $\m D$ on $\mb D.$
    Clearly,
    \bea\label{dinv}
    K\big(\xi_k(z),\xi_k(w)\big)=\sum_{n=0}^\infty a_n \big(\xi_k(z)\overline{\xi_k(w)}\big)^n =\sum_{n=0}^\infty a_n(z\ov w)^n = K(z,w)
    \eea
    for $k=0,1,\ldots, m-1.$ That is, the  kernel $K$ is ${\rm Deck(B)}$-invariant. Since the group ${\rm Deck}{\rm (B)}$ is cyclic $\widehat{{\rm Deck}{\rm (B)}}$ consists of exactly $m$ one dimensional representations of ${\rm Deck}{\rm (B)}.$% which is same as the characters of the cyclic group $\mb Z_m$ as  ${\rm Deck}{\rm (B)}\cong\mb Z_m.$  
 Therefore, 
\Bea    
\w{{\rm Deck}{\rm (B)}}=\{\chi_j: j=0,1\ldots, m-1\},
\Eea
where $\chi_j(\xi_k)=e^{\frac{2kj\pi}{m} i}$ for $j,k=0,1,\ldots, m-1.$
By Corollary \ref{rhoortho} and Equation \eqref{od}, the operator $\mb P_j:\m H_K\to\m H_K$ defined by
\Bea
(\mb P_jf)(z)=\frac{1}{m}\sum_{k=0}^{m-1}\ov{\chi_j(\xi_k)}f(\xi_k^{-1} \cdot z)
\Eea
 is an orthogonal projection for $j=0,1,\ldots, m-1.$ Since $\widehat{{\rm Deck}{\rm (B)}}$ consists of one dimensional representations only, Theorem \ref{fd1} and Lemma \ref{piecedim} describe the family of minimal reducing subspaces of $M_{z^m}$ on $\m H_K.$ Note that in this case $\lvert\w{{\rm Deck}{\rm (B)}}\rvert=\sum_{j=0}^{m-1}1=m.$
\begin{thm}\label{one}
The reducing  submodules of $\m H_K$ over   $\C[z^m]$ are   $\mb P_j \m H_K,\, j=0,1,\ldots, m-1$. Each of these submodules is of rank $1$ and consequently, they are minimal reducing subspaces of $M_{z^m}$ on $\m H_K.$
\end{thm}
\begin{rem}\label{unique}\rm
   Clearly, $\m M_j:=\ov{{\rm span}}\{z^{mk+j}\}_{k=0}^\infty$ for $j=0,1,\ldots, m-1$ is the family of reducing subspaces of $M_{z^m}$ on $\m H_K.$ The aforementioned spanning set  forms an orthogonal basis of $\m M_j$ and $\mb P_j\m H_K=\m M_{j}$ for $j=0,1,\ldots, m-1.$ We compute the reproducing kernel ${\mathfrak K}_{m - 1}$ of $\m M_{m-1}$ following Corollary \ref{t2} and Equations \eqref{rksub} and \eqref{onegen}. Since $\{1,\ldots, z^{m - 1}\}$ is a basis of $\C[z]$ over $\C[z^m]$, from Equation \eqref{diagker} we have
   \Bea
   m^2z^{m - 1}{\mathfrak K}_{m - 1}(z^m, w^m)\ov{w}^{m - 1} &=& \frac{1}{m}\sum_{k=0}^{m-1}\ov{\chi_{m - 1}(\xi_k)}K(\xi_k^{-1}\cdot z, w) \\
   &=& \frac{1}{m}\sum_{k=0}^{m-1}e^{\frac{2k\pi}{m}i}\sum_{n\geq 0}a_n e^{\frac{2kn\pi}{m}i}(z\ov{w})^n\\
   &=& \frac{1}{m}\sum_{n\geq 0}a_n(z\ov{w})^n\sum_{\ell=0}^{m-1}(\omega^\ell)^{n + 1},\, \mathrm{~where~} \omega = e^{\frac{2\pi }{m}i} \\
   &=& (z\ov{w})^{m - 1}\sum_{j\geq 0}a_{(j + 1)m - 1}(z\ov{w})^{jm}.
   \Eea
   We further observe that 
   \begin{enumerate}
       \item [(i)] for the Hardy space, $a_{(j + 1)m - 1} = 1$ implying that ${\mathfrak K}_{m - 1} = \frac{1}{m^2}K_{H^2(\mb D)},$
       \item[(ii)] for the Bergman space, $a_{(j + 1)m - 1} = (j + 1)m$ implying that ${\mathfrak K}_{m - 1} = \frac{1}{m}K_{\mb A^2(\mb D)},$
      \item [(iii)] for the Dirichlet space, $a_{(j + 1)m - 1} = \frac{1}{{(j + 1)m}}$ implying that ${\mathfrak K}_{m - 1} = \frac{1}{m^3}K_{\m D}.$
   \end{enumerate}   Hence, the restriction of $M_{z^m}$ to $\m M_{m-1}$, when  $\m H_K = H^2(\mb D), \mb A^2(\mb D)$ and $\m D$, is unitarily equivalent to the unilateral shift, Bergman  shift (the weighted shift with weight sequence $\Big\{\sqrt{\frac{n+1}{n+2}}\Big\}_{n=0}^\infty$) and  Dirichlet shift (the weighted shift with weight sequence $\Big\{\sqrt{\frac{n+2}{n+1}}\Big\}_{n=0}^\infty$), respectively. Analogous phenomenon continues to happen for a proper holomorphic multiplier on the Bergman space of an arbitrary bounded domain in $\mb C^n,$ see \cite{G}.  
   \end{rem}
   
    \item For $m\geq 2,$ consider ${\rm B} = \v_a^m$ for $a\in\mb D,$ where $\v_a(z)=\frac{z-a}{1-\ov{a} z}.$ It is well-known  that the  kernel $K_\l(z, w)=(1-z\ov w)^{-\l},\,\l>0$ on $\mb D\times\mb D$ is positive definite and $M_z$ is homogeneous on the associated Hilbert space $\mb A^{(\l)}(\mb D)$ \cite{GMpams}. Equivalently, $K_\l$ is quasi-invariant (see \cite[p.2422]{KGjfa}) which in this case means 
    \Bea
    K_\l(z,w)=\v^\i(z)^{\l/2}K_\l\big(\v(z),\v(w)\big){\ov{\v^\i(w)}}^{\l/2}\,\, \text{~for~all~}\v\in \text{M\"ob},\,\, z, w\in\mb D.
    \Eea
   Therefore, the operator $U^{(\l)}:\mbox{M\"ob}\to\m U\big(\mb A^{(\l)}(\mb D)\big)$  defined by 
    \bea\label{quasi}
    U^{(\l)}(\v^{-1})f = (\v^\i\big)^{\l/2} \cdot (f \circ \v) 
    \eea 
     is a (projective) unitary representation of  M\"ob, $\m U(\m H)$ being the group of unitary operators on the Hilbert space $\m H.$ This defines a (genuine) unitary representation of M\"ob if $\l$ is a positive even integer \cite[p.12]{JA}. For $\v_a\in$ {~M\"ob} and $U^{(\l)}(\v_a^{-1})$  intertwines the multiplication operators $M_{z}$ and $M_{\v_a}$ on $\mb A^{(\l)}(\mb D),$ and consequently, 
    \Bea
    U^{(\l)}(\v_a^{-1})M_{z^m} = M_{\rm B} U^{(\l)}(\v_a^{-1}).
    \Eea
Therefore, if $\m M$ is a reducing subspace of $M_{z^m},$ then $U^{(\l)}(\v_a^{-1})(\m M)$ is a reducing subspace of $M_{\rm B}.$ Hence  we conclude the following result from Theorem \ref{one}.
    \begin{thm}
     For $\l>0$ and ${\rm B}=\v_a^m,\, M_{\rm B}$ acting on $\mb A^{(\l)}(\mb D)$ has exactly $m$ number of minimal reducing subspaces on $\mb A^{(\l)}(\mb D).$ 
    \end{thm}
    %Note that $\l=1$ and $\l=2$ correspond to the Hardy space $H^2(\mb D)$ and Bergman space $\mb A^2(\mb D),$ respectively.
    \vspace{.1 cm}
\item Choose ${\rm B}(z)= z\v_a(z)  \text{ for } a \in \mb D.$ Since $\v_a\big(\v_{-1,a}(z)\big)=-z$ for $a, z\in\mb D$ and for each fixed $a\in\mb D,\, \v_{-1, a}$ is the unique involution ($\v_{-1,a}^2$= id), ${\rm Deck(B)} = \{{\rm id},\v_{-1,a}\}$. Then $\widehat{{\rm Deck}{\rm (B)}}=\{1, -1\}.$ For a positive even integer $\l,$ $U^{(\l)}:\mbox{M\"ob}\to\m U\big(\mb A^{(\l)}(\mb D)\big)$ as in Equation \eqref{quasi} is a unitary representation and hence its restriction to ${\rm Deck(B)}$ is so. Thus, by Corollary \ref{rhoortho}, we have the orthogonal projections 
\Bea
\mb P_1=\frac{1}{2}\big(I+U^{(\l)}(\v_{-1,a})\big) \,\,\text{~and~} {\mb P_{-1}}=\frac{1}{2}\big(I-U^{(\l)}(\v_{-1,a})\big),
\Eea
satisfying $\mb P_1+{\mb P_{-1}} = I,$ where $I$ denotes the identity operator on $\mb A^{(\l)}(\mb D).$ Being unitary $U^{(\l)}(\v_{-1, a})\not\equiv 0$ and it is clear that $U^{(\l)}(\v_{-1, a})\neq \pm I$ for all $a\in\mb D.$ Therefore, both of the projections are non-trivial. 
Note even though the group is a pseudoreflection group, we are not working with the left regular representation and thereby, the results for arbitrary unitary representation in Section \ref{decanahil} become relevant. 
    
%It is well-known  that the reproducing kernel $K_\l(z, w)=(1-z\bar w)^{-\l},\,\l>0$ on $\mb D$ is quasi-invariant, that is
% \Bea
% K_\l(z,w)=\v^\i(z)^{\l/2}K\big(\v(z),\v(w)\big)\ov{\v^\i(w)^{\l/2}}\,\, \text{~for~all~}\v\in \text{~M\"ob~}, z, w\in\mb D.
% \Eea
%  Therefore, the operator $U(\v):\mb A^{(\l)}(\mb D)\to\mb A^{(\l)}(\mb D)$ defined by 
\begin{rem}
  Since $\v_{-1,a}^2={\rm id},\, U^{(\l)}(\v_{-1, a})$ is a self-adjoint unitary, that is, 
   \Bea
   \big(U^{(\l)}(\v_{-1, a})\big)^*=U^{(\l)}\big((\v_{-1, a})^{-1}\big)=U^{(\l)}(\v_{-1, a}).
   \Eea
    Thus, its spectrum contains $\pm 1$ as eigenvalues. Let $\m M_1$ and ${\m M_{-1}}$ be the eigenspaces of $U^{(\l)}(\v_{-1, a})$ corresponding to the eigenvalues $1$ and $-1$ respectively. So $\mb P_1\m M_1=\m M_1$ and $\mb P_{-1}\m M_{-1}=\m M_{-1}.$
   \end{rem}
From definition it is easy to verify that
\Bea
U^{(\l)}(\v_{-1,a})M_{\rm B}=M_{\rm B}U^{(\l)}(\v_{-1,a}) \,\,\text{~for~}\,\, {\rm B}(z)=z\v_a(z).
\Eea
 and therefore, from the remark above, we have the following description of reducing subspaces of $M_{\rm B}$ for ${\rm  B}(z)=z\v_a(z).$
   \begin{thm}\label{two}
  For ${\rm B}(z)=z\v_a(z),\,\, M_{\rm B}\mb P_1=\mb P_1M_{\rm B}.$ Therefore, $\m M_1$ and $\m M_{-1}$ are reducing subspaces of $M_{{\rm B}}$ on $\mb A^{(\l)}(\mb D)$ for every positive even integer $\l.$
   \end{thm}

    \vspace{.1 cm}
    \item Choose the finite Blaschke product to be ${\rm B}= \v_a\v_b ,\,\, a, b\in\mb D$ being the two zeros of ${\rm B}.$ We have 
    \Bea
    {\rm B}=\psi\circ\v_{-1,b}\,\,\text{for}\,\, \psi(z)=z\v_{-1,a}\circ\v_{-1,b}(z), \,\, z\in \mb D.
     \Eea
     It is easy to see that there is a unimodular constant $c$ such that 
    \Bea
    \v_{-1,a}\circ\v_{-1,b}=c\v_\mu,
    \Eea
    where $\mu=\v_{-1,a}(b),$ see \cite[Lemma 5.3]{MSR}. Since  
    \Bea
    U^{(\l)}(\v_{-1,b})M_\psi=M_{\rm B} U^{(\l)}(\v_{-1,b}), \,\,\text{where}\,\, \psi(z)=cz\v_\mu(z).
    \Eea
   Thus, we conclude the following result from Theorem \ref{two}.
    \begin{thm}
    For ${\rm B}=\v_a\v_b$ and every positive even integer $\l,$ the minimal reducing subspace of $M_{\rm B}$ acting on $\mb A^{(\l)}(\mb D)$ are 
    \Bea
    U^{(\l)}(\v_{-1,b})(\m M_\mu)\,\,\,\,\,\, \mbox{and}\,\,\,\,\,\, U^{(\l)}(\v_{-1,b})(\m M_\mu^\perp).
    \Eea
    \end{thm}
  
    %\item In \cite[Proposition 4.2, p. 203]{CC1}, Cassier and Chalendar identify a class of Blaschke products of order $n$ whose group of deck transformations is cyclic group of order $n.$ 
    %\begin{prop}
    %Let $n$ be a natural number and $u $ be in the form $\phi_a$ for $a \in \mb D$ such that 
    %\begin{enumerate}
    %    \item $u^n(0) = 0$ and,
     %   \item $\{0,u(0),\ldots, u^{n-1}(0)\}$ is a set of $n$ distinct points in $\mb D.$
    %\end{enumerate}
    %Then for the Blaschke product ${\rm B}(z) = z \prod_{i=1}^{n-1} \phi_{u^i(0)}(z)$, ${\rm Deck(B)}$ is generated by $u.$
    %\end{prop}
    %Thereafter, each Blaschke product $\rm B$ of such form is a Galois covering. We conclude with alike arguments as given in \ref{one} that for every $\varrho \in \widehat{\textnormal{Deck}}\textnormal{(B)}$, $\mb P_\varrho \big(\mb A^2(\mb D)\big)$ is minimal and irreducible reducing subspace of $M_{{\rm B}}$ on $\mb A^2(\mb D).$
\end{enumerate}
Similar computations are done in \cite{Z} for $\l=2,$ nevertheless, we include it here to emphasize applications of our techniques, not only to the case $\l=2,$ but to a broader class of Hilbert modules.

\subsection{Example (Unit polydisc)}
Here we generalize the results obtained in the case of unit disc to polydisc and with a little more computations, we would see that analogous results hold. It is well-known that  any proper holomorphic self-map $\mathbf B$ of the  polydisc $\mb D^n$ is given by 
\bea\label{Bl}
{\mathbf B}(\bl z) = \big({\rm B}_1 (z_{\sigma(1)}),\ldots, {\rm B}_n (z_{\sigma(n)})\big),
\eea
for some $\sigma\in\mathfrak S_n$, where each ${\rm B}_i$ is a finite  Blaschke product, see \cite[p.76]{RN}. Recall that biholomorphic automorphism group of $\mb D^n$ is given by \cite[p.68]{RN}
\Bea
{\rm Aut}(\mb D^n)={{\rm Aut}(\mb D)}^n\rtimes\mathfrak S_n.
\Eea
%\vspace{0.1in}
\begin{enumerate}[leftmargin=*]
\item If we choose ${\rm B}_i(\bl z)=z_i^{m_i},$ where $m_i$ is a positive integer for $i=1,\ldots, n.$ Then
%\Bea
%{\mathbf B}(\bl z)=(z_1^{m_1},\ldots, z_n^{m_n}) \text{~for~}\bl z=(z_1,\ldots, z_n)\in\mb D^n \mathrm{~and~}
%\Eea
\bea\label{deck}
    {\rm Deck}(\mathbf B) = \{ \bl\xi : \mb D^n \to \mb D^n:  \bl\xi=\big(\xi_1,\ldots,\xi_n\big), \mathrm{~for~} \xi_j\in{\rm Deck(B_j)}, 1\leq j\leq n\}.
    \eea
Hence ${\rm Deck}(\mathbf B)$ is isomorphic to $\mb Z_{m_1}\times\ldots\times\mb Z_{m_n}.$ Thus, $\widehat{{\rm Deck}{(\mathbf B)}}$ consists of exactly $m_1\ldots m_n$ one dimensional representations of ${\rm Deck}{(\mathbf B)},$ and
\Bea    
\w{{\rm Deck}{(\mathbf B)}}=\{\chi_{\bl j}: \bl j\in\mb Z_{m_1}\times\ldots\times\mb Z_{m_n}\}, 
\Eea
where each $\chi_{\bl j}$ is a character of $\mb Z_{m_1}\times\ldots\times\mb Z_{m_n}.$

Let $\m H_K$ be an analytic Hilbert module with a diagonal reproducing kernel $K$ 
such that
    \bea\label{repk}
    K(\bl z,\bl w)=\sum_{I\in\mb Z_+^n} a_I {\bl z}^I{\ov {\bl w}}^I
    \eea
    for some $a_I\geq 0$, where $\mb Z_+$ is the set of non-negative integers, $I=(i_1,\ldots,i_n)\in\mb Z_+^n$ is a multi-index and ${\bl z}^I=z_1^{i_1}\ldots z_n^{i_n}.$ 
Set $\mathbf M:=(M_1,\ldots, M_n).$ 
 %This is same as assuming that $\sup_{n\geq 0}\lvert\frac {a_n}{a_{n+1}}\rvert<\infty.$ 
For $\l_i>0$ for $i=1,\ldots, n$ and $a_I=\frac{(\l_1)_{i_1}\ldots (\l_n)_{i_n}}{i_1!\ldots i_n!},$ we get the Hilbert space $\mb A^{(\bl\l)}(\mb D^n)$ of holomorphic functions with reproducing kernel $K^{(\bl\l)}$ given by
 \Bea
K^{(\bl\l)}(\bl z,\bl w) =\prod_{j=1}^n(1-z_j\bar w_j)^{-\l_j}
 \Eea
 where $(\l)_n=\l(\l+1)\ldots(\l+n-1).$ This example includes the Bergman space and the Hardy space on the polydisc. Also, if $a_I=\prod_{j=1}^n\frac{1}{i_j+1},\, M_i$ is bounded on $\m H_K$ for $i=1,\ldots, n,$ the choice $\{\prod_{j=1}^n\frac{1}{i_j+1}\}_{I\in\mb Z_+^n}$ corresponds to the Dirichlet space on $\mb D^n,$ see \cite{BKKLSS, LB}.
From calculations analogous to that in  Equation \eqref{dinv}, it follows that the  kernel $K$ is ${\rm Deck(\mathbf B)}$-invariant.  Following Equation \eqref{od}, we define $\mb P_{\bl j}:\m H_K\to\m H_K$ by
\bea\label{projmono}
(\mb P_{\bl j}f)(\bl z)=\frac{1}{m_1\ldots m_n}\sum_{\bl\xi \in {\rm Deck}(\mathbf B)}\ov{\chi_{\bl j}(\bl\xi)}f(\bl\xi^{-1}\cdot\bl z)
\eea
for $\bl j\in\mb Z_{m_1}\times\ldots\times\mb Z_{m_n},$ where $\bl\xi^{-1}=(\xi_1^{-1},\ldots,\xi_n^{-1}).$ Corollary \ref{rhoortho} ensures that each $\mb P_{\bl j}$ is an orthogonal projection. We denote the tuple of multiplications operators $(M_{\rm B_1},\ldots, M_{\rm B_n})$ on $\m H_K$ by $\mathbf M_{\mathbf B}.$
Hence Theorem \ref{fd1} and Lemma \ref{piecedim} describe the family of minimal reducing subspaces of $\mathbf M_{\mathbf B}$ on $\m H_K.$ We have $\lvert\w{{\rm Deck}{(\mathbf B)}}\rvert=\sum_{\bl k\in \w{{\rm Deck}{(\mathbf B)}}}1=m_1\ldots m_n$ here, as well.
\begin{thm}\label{one2}
The reducing  submodules of $\m H_K$ over   $\C[z_1^{m_1},\ldots,z_n^{m_n}]$ are  $\mb P_{\bl j}\m H_K, \bl j\in\w{{\rm Deck}{(\mathbf B)}}.$  Each of these submodules is of rank $1$ and consequently, they are minimal reducing subspaces of $\mathbf M_{\mathbf B}$ on $\m H_K.$

%For each $\bl j\in\w{{\rm Deck}{(\mathbf B)}}$, the submodules $\mb P_{\bl j}\m H_K$ of the analytic Hilbert module $\m H_K$ are reducing submodules of rank $1$ over $\C[z_1^{m_1},\ldots,z_n^{m_n}]$ and consequently, they are minimal reducing subspaces of $\mathbf M_{\mathbf B}$ on $\m H_K.$
\end{thm}

\begin{rem}\rm
Similar to the Remark \ref{unique}, we consider 
$$\m M_{\bl j} := \ov{\mbox{span}}\{z_1^{m_1k_1 + j_1}\cdots z_n^{m_nk_n + j_n}: k_i\in\mb Z_+\}$$ for $\bl j = (j_1, \ldots,j_n)\in \mb Z_{m_1}\times\ldots\times\mb Z_{m_n}$. Then 
$\mb P_{\bl j} (\m H_K) = \m M_{\bl j}$. Since $H^2(\mb D^n) = \otimes_1^n H^2(\mb D), {\mb A}^2(\mb D^n) = \otimes_1^n {\mb A}^2(\mb D), \m D(\mb D^n) = \otimes_1^n \m D(\mb D)$, similar computation as in the the Remark \ref{unique} will also show that the restriction of $(M_{z_1^{m_1}}, \ldots, M_{z_n^{m_n}})$ on $\m H_K$ to $\m M_{\bl m - \bl 1}$ where $\bl m - \bl 1 = (m_1 - 1, \ldots, m_n - 1)$, is unitarily equivalent to the multiplication operator by coordinate functions on $\m H_K$ for $\m H_K = H^2(\mb D^n), {\mb A}^2(\mb D^n)$ and $\m D(\mb D^n)$. 
\end{rem}
\vspace{.1 cm}

\item For $a_1,\ldots,a_n\in\mb D,$ take
\Bea
\mathbf B(\bl z)=\big(\v_{a_1}^{m_1}(z_1),\ldots, \v_{a_n}^{m_n}(z_n)\big) \text{~for~}\bl z=(z_1,\ldots, z_n)\in\mb D^n.
\Eea
Let ${\rm B}_i(\bl z) = \v_{a_i}^{m_i}(z_i)$ for $\bl z\in\mathbb D^n$, $1\leq i\leq n$. Analogous to the case of the unit disc, the  kernel $K^{(\bl\l)}(\bl z,\bl w)$ is quasi-invariant as
\Bea
K^{(\bl\l)}(\bl z,\bl w)=\prod_{i=1}^n{\big(\v_i^\i(z_i)\big)}^{\l_i/2}K^{(\bl\l)}\big(\Phi(\bl z),\Phi(\bl w)\big)\prod_{i=1}^n{\big(\v_i^\i(w_i)\big)}^{\l_i/2},
\Eea
where $\Phi=(\v_1,\ldots,\v_n)\in\mbox{M\"ob}^n.$ Therefore, the operator $U^{(\bl \l)}:\mbox{M\"ob}^n\to\m U\big(\mb A^{(\bl\l)}(\mb D^n)\big)$  defined by
\Bea
U^{(\bl\l)}(\Phi^{-1})=\prod_{i=1}^n{(\v_i^\i)}^{\l_i/2}(f\circ \Phi),
\Eea
is a (projective) unitary representation of $\mbox{M\"ob}^n$ and as mentioned earlier, is a (genuine) unitary representation of $\mbox{M\"ob}^n,$ when $\bl\l=(\l_1,\ldots,\l_n)$ is tuple of positive even integers. For $\bl a = (a_1,\ldots,a_n)$, let $\Phi_{\bl a} = (\v_{a_1},\ldots,\v_{a_n})$. Since $U^{(\bl\l)}(\Phi_{\bl a}^{-1})$  intertwines the multiplication operators $M_{z_i^{m_i}}$ and $M_{\rm B_i}$ on $\mb A^{(\bl\l)}(\mb D^n)$ for $i=1,\ldots, n,$ we conclude the following result from Theorem \ref{one2}.
    \begin{thm}
     The multiplication operator tuple  $\mathbf M_{\mathbf B}$ has exactly $\prod_{i=1}^nm_i$ number of minimal joint reducing subspaces on $\mb A^{(\bl\l)}(\mb D^n),$  for $\bl\l = (\l_1,\ldots,\l_n),\l_i>0.$
    \end{thm}
    %Note that $\l=1$ and $\l=2$ correspond to the Hardy space $H^2(\mb D)$ and Bergman space $\mb A^2(\mb D),$ respectively.
    \vspace{.1 cm}

\item  For $a_1,\ldots,a_n\in\mb D,$ take
\Bea
\mathbf B(\bl z)=\big(z_1\v_{a_1}(z_1),\ldots, z_n\v_{a_n}(z_n)\big) \text{~for~}\bl z=(z_1,\ldots, z_n)\in\mb D^n.
\Eea
Let ${\rm B}_i(\bl z) = z_i\v_{a_i}(z_i)$ for $\bl z\in\mathbb D^n$, $1\leq i\leq n$. By analogous arguments as in the case of ${\rm B}(z)=z\v_a(z) $ for $a,z\in\mb D,$ it follows that 
\Bea
{\rm Deck}(\mathbf B)=\prod_{i=1}^n\{{\rm id},\v_{-1, a_i}\}.
\Eea
That is, ${\rm Deck}(\mathbf B)\cong\mb Z_2^n.$ So $\w{{\rm Deck}(\mathbf B)}=\w{\mb Z_2^n}.$ For $\bl k\in \mb Z_2^n,$ let 
\Bea
\chi_{\bl k}(\bl j)=e^{2\pi i(\bl j\cdot\bl k)/2}, \,\,\text{where}\,\, \bl j\cdot\bl k=\sum_{r=1}^nj_rk_r\,\,\text{ for}\,\, \bl j,\bl k\in\mb Z_2^n.
\Eea
It is easy to verify that 
\bea\label{commute}
U^{(\bl\l)}(\Phi)M_{{\mathrm B}_i}=M_{{\mathrm B}_i}U^{(\bl\l)}(\Phi) \,\, \text{for}\,\, i=1,\ldots, n,
\eea
for all $\Phi\in {\rm Deck}(\mathbf B).$
As $\w{{\rm Deck}(\mathbf B)}\cong{\rm Deck}(\mathbf B)$ \cite[Theorem 3.9]{Con}, for $\bl k\in\mb Z_2^n,$ define the linear operator $\mb P_{\bl k}:\mb A^{(\bl\l)}(\mb D^n)\to\mb A^{(\bl\l)}(\mb D^n)$ by
\Bea
\mb P_{\bl k}=\frac{1}{2^n}\sum_{\Psi_{\bl j}\in{\rm Deck}(\mathbf B)}\ov{\chi_{\bl k}(\bl j)}U^{(\bl\l)}(\Psi_{\bl j}^{-1}),
\Eea
where we use the isomorphism $\Psi_{\bl j}\mapsto\bl j$ from ${\rm Deck}(\mathbf B)$ onto $\mb Z_2^n.$
Since for every $n$-tuple of positive even integers $\bl\l$, $U^{(\bl\l)}$ is a unitary representation of $\mbox{M\"ob}^n,$ and hence of ${\rm Deck}(\mathbf B)$, from Corollary \ref{rhoortho} it follows that
 the linear operators $\{\mb P_{\bl k}\}_{\bl k\in\w{{\rm Deck}(\mathbf B)}}$ on $\mb A^{(\bl\l)}(\mb D^n)$ are orthogonal projections satisfying 
\Bea
\mb P_{\bl k}\mb P_{{\bl k}^\i}=\mb P_{\bl k}\delta_{\bl k{\bl k}^\i} ,\,\,\,\,\sum_{\bl k\in\w{{\rm Deck}(\mathbf B)}}\mb P_{\bl k}=I
\Eea 
and 
\Bea
\mb P_{\bl k}M_{{\rm B}_i}=M_{{\rm B}_i}\mb P_{\bl k} \,\,\text{for}\,\,\bl k\in \w{{\rm Deck}(\mathbf B)},\,\ i=1,\ldots, n.
\Eea
Thus, we have the following theorem.
\begin{thm}\label{mtwo}
For every $n$-tuple of positive even integers $\bl\l,$ and each $\bl k\in\w{{\rm Deck}(\mathbf B)},$ $\mb P_{\bl k}\big(\mb A^{(\bl\l)}(\mb D^n)\big)$ is a reducing subspace of the multiplication operator  $\mathbf M_{\mathbf B}$ acting on $\mb A^{(\bl\l)}(\mb D^n).$  
\end{thm}
\vspace{.1 cm}
\item  For $a_1,\ldots,a_n, b_1,\ldots,b_n\in\mb D,$ take
\Bea
\mathbf B(\bl z)=\big(\v_{a_1}(z_1)\v_{b_1}(z_1),\ldots, \v_{a_n}(z_n)\v_{b_n}(z_n)\big) \text{~for~}\bl z=(z_1,\ldots, z_n)\in\mb D^n.
\Eea
Let ${\rm B}_i(\bl z) = \v_{a_i}(z_i)\v_{b_i}(z_i)$ for $\bl z\in\mathbb D^n$, $1\leq i\leq n$. If $\Psi:\mb D^n\to\mb D^n$ is given by 
\Bea
\Psi(\bl z)=\big(z_1\v_{-1, a_1}\circ\v_{-1,b_1}(z_1),\ldots, z_n\v_{-1, a_n}\circ\v_{-1,b_n}(z_n)\big),\,\,
\Eea
then ${\mathbf B}=\Psi\circ \Phi_{\bl b},$ where $\Phi_{\bl b}:\mb D^n\to\mb D^n$ is given by $\Phi_{\bl b}=(\v_{-1,b_1},\ldots,\v_{-1,b_n})$ and $\bl z=(z_1,\ldots,z_n).$ Since $\Phi_{\bl b}\in\mbox{M\"ob}^n$ is an involution, we have
\bea\label{intert}
U^{(\bl\l)}(\Phi_{\bl b}) M_{\Psi_i}=M_{\mathrm B_i}U^{(\bl\l)}(\Phi_{\bl b})\,\,\text{for}\,\, i=1,\ldots, n,
\eea
where $\Psi_i(\bl z)=z_i\v_{-1, a_i}\circ\v_{-1,b_i}(z_i).$
With notations of Theorem \ref{mtwo}, we have the following conclusion from Equation \eqref{intert} and Theorem \ref{mtwo}.
\begin{thm}
For every $n$-tuple of positive even integers $\bl\l,$ and each $\bl k\in\w{{\rm Deck}(\mathbf B)},$ $\big(U^{(\bl\l)}(\Phi_{\bl b})\mb P_{\bl k}\big)\big(\mb A^{(\bl\l)}(\mb D^n)\big)$ is a reducing subspace of the multiplication operator $\mathbf M_{\mathbf B}$ acting on $\mb A^{(\bl\l)}(\mb D^n).$  
\end{thm}

\end{enumerate}
\subsection{Example (Reinhardt domains)} We apply the following result from \cite[p. 61 - 63]{DS1}.
\begin{thm}\label{Rein}
Let $R$ be a Reinhardt domain in $\mb C^n,\, n>1$ and $D$ be a bounded simply connected strictly pseudoconvex domain with $C^\infty$ boundary. If $\bl F:R\to D$ is proper holomorphic mapping then
\begin{enumerate}
    \item [(i)] there exists a finite pseudoreflection group $G \subseteq{\rm Aut}(R)$ such that $\bl F$ is precisely $G$-invariant (factored by automorphisms);
    \item[(ii)] $\bl F = \Phi\circ \bl{\rm P}$, where $\bl {\rm P}$, a basic polynomial associated to $G$, is given by 
    \bea\label{mono}
    \bl {\rm P}(\bl z)=(z_1^{m_1},\ldots, z_n^{m_n}),
    \eea
    for some $(m_1,\ldots,m_n)\in\mb N^n$,
    and $\Phi: \bl{\rm P}(R)\ra D$ is a biholomorphism. 
    
\end{enumerate}
\end{thm}

We consider the polynomial map $\bl {\rm P}:\mb C^n\to\mb C^n$ given by the equation above. Since $R$ is a Reinhardt domain, $\mb T^n\subseteq {\rm Aut}(R).$ Therefore, ${\rm Deck}(\bl {\rm P})$ is the same as ${\rm Deck}(\mathbf B)$ described in Equation \eqref{deck}.
 Let $\m H_K$ denote analytic Hilbert module on $R$ with reproducing kernel $K$ which is ${\rm Deck}(\bl {\rm P})$-invariant. In particular, the kernel defined in Equation \eqref{repk} is ${\rm Deck}(\bl {\rm P})$-invariant.  Then by Equation \eqref{od} and Corollary \ref{rhoortho}, it follows that $\mb P_{\bl k}:\m H_K\to\m H_k$ defined by 
\Bea
(\mb P_{\bl k}f)(\bl z)=\frac{1}{m_1\ldots m_n}\sum_{\bl\xi\in {\rm Deck}(P)}\ov{\chi_{\bl k}(\bl\xi)}f(\bl\xi^{-1} \cdot \bl z),\,\, \bl z\in R.
\Eea
is an orthogonal projection for each $\bl k\in\w{{\rm Deck}(\bl {\rm P})}\cong\mb Z_{m_1}\times\ldots\times\mb Z_{m_k}.$ Therefore, Theorem \ref{fd1} and Lemma \ref{piecedim} describe the family of minimal reducing subspaces of the multiplication tuple $\bl {\rm M}_{\bl {\rm P}}$ whose $i$-th component is the multiplication operator by $z_i^{m_i}$ on $\m H_K.$ Recall that $\bl {\rm P}$ is a representative of the proper holomorphic mapping $F:R \to D$ in Theorem \ref{Rein}. Here also we have $\lvert\w{{\rm Deck}(\bl {\rm P})}\rvert=\sum_{\bl k\in { {\rm Deck}(\bl P)}}1=m_1\ldots m_n.$
\begin{thm}
For each $\bl k\in\w{{\rm Deck}{(\bl {\rm P})}}$, the submodules $\mb P_{\bl k} (\m H_K)$ of the analytic Hilbert module $\m H_K$ are reducing submodules of rank $1$ over $\C[z_1^{m_1},\ldots,z_n^{m_n}]$ and consequently, they are minimal reducing subspaces of of $\bl {\rm M}_{\bl {\rm P}}$ on $\m H_K.$
\end{thm}
Monomial proper maps and their abelian groups of deck transformations occur in number of examples discussed above. We construct projection operator using the characters of these abelian groups. Monomial type maps and decomposition of Bergman space is considered in a recent work, see \cite{NP}. 

We discuss the case of proper holomorphic multipliers from the euclidean unit ball $\mb B_n$ in $\mb C^n.$ Analogous discussion is valid for other instances cited from the literature.
\subsection{Example (Unit ball)} Recall that  $\mb B_n$ is the open unit ball with respect to the $\ell^2$-norm on $\mb C^n$ and $U\in\m U_n$  denote the group of unitary operators on $\mb C^n.$
 %\bea \label{unitary}\langle U\bl z,U\bl w\rangle=\langle \bl z,\bl w\rangle \text{~for~} \bl z, \bl w \in \C^n. \eea %with respect to $\ell^2$-norm 
In \cite[Theorem 1.6, p. 704]{R}, Rudin proved
the following result for proper holomorphic maps from $\mb B_n$ onto a domain $\Omega \subset \mathbb{C}^n, n>1,$ see also \cite[p. 506]{BB}
\begin{thm}\label{Rudin}
Suppose that $\tilde{\bl F} : \mb B_n \to \Omega$ is a proper holomorphic mapping from the open unit ball $\mb B_n$ in $\mb C^n (n>1)$ onto a domain $\Omega$ in $\mb C^n$ with multiplicity $d>1.$ Then there exists a unique finite pseudoreflection group $G$ of order $d$  such that 
\Bea
\tilde{\bl F}= \Phi \circ \bl \theta \circ \Psi ,
\Eea
where $\Psi $ is an automorphism of $\mb B_n$ and $\bl \theta$ is a basic polynomial mapping (hsop) associated to $G.$ Furthermore, $\bl\theta(\mb B_n)$ is a domain in $\mb C^n, \bl\theta:\mb B_n\to\bl\theta(\mb B_n)$ is proper, and   $\Phi:\bl\theta(\mb B_n)\to\Omega$ is  biholomorphic.
\end{thm}
We know that any pseudoreflection group acting on $\mb C^n$ is a subgroup of $\m U_n.$ Since $\m U_n$ leaves $\mb B_n$ invariant,  $\mb B_n$ is $G$-invariant as well and so is the Bergman kernel $B_{\mb B_n}$of $\mb B_n$ and all its positive powers are $G$-invariant positive definite kernels. The corresponding analytic Hilbert modules over $\mb C[\bl z]$ are denoted by $\mb A^{(\l)}(\mb B_n), \l>0,$ known as weighted Bergman modules, see Example \ref{example}.

Consider the multiplication operator $M_{\theta_i} : \mb A^{(\l)}(\mathbb{B}_n) \to \mb A^{(\l)}(\mathbb{B}_n)$ defined by 
\Bea
M_{\theta_i}f=\theta_i f\,\,\text{ for~~ all}\,\, i=1,\ldots,n.
\Eea
%$\norm{\theta_i f}_2^2 \leq \norm{\theta_i}^2_{\infty}\norm{f}_2^2$ for all $i=1,\ldots,n$ shows that each operator $M_{\theta_i}$ is bounded.
Clearly, $\bl{\rm M}_{\bl \theta} = (M_{\theta_1},\ldots,M_{\theta_n})$ defines a commuting tuple of bounded operators on $\mb A^{(\l)}(\mathbb{B}_n),$ see Example \ref{example}. From Corollary \ref{reduce1}, we note that for each  $\varrho\in\w G,\,\,\mathbb{P}_{\varrho}\big(\mb A^{(\l)}(\mathbb{B}_n)\big)$ is a joint reducing subspace of $M_{\bl \theta}.$ 
We set $\bl F:=\bl\theta\circ\Psi={\Phi}^{-1}\circ\tilde{\bl F} : \mb B_n \to \bl \theta(\mb B_n).$ From the definition of representative of a proper holomorphic map, we note that $\bl F=\bl\theta\circ\Psi$ is a representative of $\tilde{\bl F}.$ Suppose $\tilde F_i$ and $F_i$'s $1\leq i\leq n,$ are holomorphic multipliers on $\mb A^{(\l)}(\mathbb{B}_n),$ (in other words, they are in the multiplier algebra of $\mb A^{(\l)}(\mathbb{B}_n)$) and set $\mathbf M_{\tilde{\bl F}} := (M_{\tilde F_1},M_{\tilde F_2},\cdots,M_{\tilde F_n})$ and $\mathbf M_{\bl F} := (M_{F_1},M_{F_2},\cdots,M_{F_n}).$

\begin{prop}\label{eqimodtheta}
The commuting  tuples of bounded operators  $\mathbf M_{\bl F}$ and $\bl{\rm M}_{\bl\theta}$ on $\mb A^{(\l)}(\mathbb{B}_n)$ are unitarily equivalent.
%where $M_{F_i} : $
\end{prop}
\begin{proof}
Define the map $U^{(\l)}(\Psi^{-1}): \mb A^{(\l)}(\mathbb{B}_n) \to\mb A^{(\l)}(\mathbb{B}_n) $ by 
\Bea
U^{(\l)}(\Psi^{-1})f=J_{\Psi}^{\frac{\l}{n+1}} ( f\circ \Psi),
\Eea 
where $J_{\Psi}$ is the complex jacobian of $\Psi.$ It is easy to verify that that  $U^{(\l)}(\Psi)$ is unitary on $\mb A^{(\l)}(\mathbb{B}_n)$ and $M_{\bl F_i}U^{(\l)}(\Psi^{-1}) = U^{(\l)}(\Psi^{-1})M_{\theta_i}$ for all $i=1,2,\ldots,n$  \cite[p.12]{JA}. 
\end{proof}

%The classical change of variables formula reads 
%$$\int_{\mb B_n} f dV = \int_{\mb B_n} |J_{\Psi}|^2 f\circ \Psi dV,$$
%where $dV$ denotes the Lebesgue measure on $\mb B_n$ and $\Psi \in \textnormal{Aut}(\mb B_n).$ 
%Choose $h \in \mb A^{(\l)}(\mathbb{B}_n)$ and denote $f = |h|^2.$ Then $\norm{\Gamma_{\Psi}h}^2 = \norm{h}^2$ for every $h \in \mb A^2(\mathbb{B}_n).$ This proves that $\Gamma_{\Psi}$ is well-defined and isometry. 
%The operator $\Gamma_{\Psi^{-1}} : \mb A^2(\mathbb{B}_n) \longrightarrow \mb A^{(\l)}(\mathbb{B}_n)$ defined by $\Gamma_{\Psi^{-1}}f=J_{\Psi_{-1}} ( f\circ \Psi^{-1}),$ %where $J_{\Psi^{-1}}$ is the complex jacobian of $\Psi^{-1}.$
%has the following property
%\Bea
%\Gamma_{\Psi}\Gamma_{\Psi^{-1}}(f) = \Gamma_{\Psi}(J_{\Psi_{-1}}  f\circ \Psi^{-1})= J_{\Psi} J_{\Psi_{-1}}\circ \Psi  f = J_{\Psi_{-1} \circ \Psi} f = f.
%\Eea
%This implies the surjectivity of $\Gamma_{\Psi}.$ Moreover, $M_{\bl F_i}\Gamma_{\Psi} = \Gamma_{\Psi}M_{\theta_i}$ for all $i=1,2,\ldots,n.$ Hence, the tuples of multiplication operators $M_{\bl F}$ and $M_{\bl \theta}$ are unitarily equivalent on $\mb A^{(\l)}(\mathbb{B}_n).$
%\end{proof}
%Similarly, one can show that $$\Gamma_{\Psi^{-1}}\Gamma_{\Psi}f=f.$$
%In a bit different context, We also have the following result.
Thus from the Theorem \ref{fd1}, we conclude the following.

\begin{cor}
For every $\varrho\in\w G$, $U^{(\l)}(\Psi^{-1})\mb P_{\varrho}\big(\mb A^{(\l)}(\mb B^n)\big)$ is a joint reducing subspace of the multiplication operator $\mathbf M_{\mathbf F}$ acting on $\mb A^{(\l)}(\mb B^n).$ 
\end{cor}

\begin{rem}\label{multbyco}
In Corollary \ref{modtheta},  we specialize $\Omega$ to be $\mb B_n$ and $\mathcal H$ to be $\mb A^{(\l)} (\mathbb{B}_n).$ This ensures the existence of a reproducing kernel Hilbert space $\wi{\mathcal H}\big(\bl \theta(\mb B_n)\big)$ consisting of $\mb C^d$-valued holomorphic functions on $\bl \theta(\mb B_n)$ with the property that $\mathbf M$ on $\wi{\mathcal H}\big(\bl \theta(\mb B_n)\big)$ is unitarily equivalent to $\mathbf M_{\bl \theta}$ on $\mb A^{(\l)} (\mathbb{B}_n),$ where $\mathbf M=(M_1,\ldots, M_n)$ denotes the tuple of multiplication operators by the coordinate functions of $\bl\theta(\mb B_n)$ on $\wi{\mathcal H}\big(\bl \theta(\mb B_n)\big).$ 
Combining this observation with the Proposition \ref{eqimodtheta}, we have that the tuple of multiplication operators $\mathbf M_{\bl F}$ on $\mb A^{(\l)} (\mathbb{B}_n)$ is unitarily equivalent to $\mathbf M$ on $\wi{\mathcal H}\big(\bl \theta(\mb B_n)\big).$ Equivalently, the multiplication operator tuple by a representative of the proper map  $\tilde{\bl F}$ on $\mb B_n$ is unitarily equivalent to $\mathbf M$ on $\wi{\mathcal H}\big(\bl \theta(\mb B_n)\big).$ Analogous result remains true for the earlier example of Reinhardt domains as well.
\end{rem}
%There exists a \rk Hilbert space $\m H_{\varrho}\big(\bl \theta(\mb B_n)\big)$ corresponding to each $\varrho \in \widehat{G}$ such that $\mathbb{P}_{\varrho}\big(\mb A^2(\mathbb{B}_n)\big)$ and $\m H_{\varrho}\big(\bl \theta(\mb B_n)\big)$ are equivalent as module over the ring of $G$-invariant polynomials. Moreover, the way we have constructed each $\m H_{\varrho}\big(\bl \theta(\mb B_n)\big)$
%Again 
%This implies that $M_{\bl F}$ on $\mathbb{P}_{\varrho}\big(\mb A^2(\mathbb{B}_n)\big)$ is unitarily equivalent to $M_{\bl z}$ on $\m H_{\varrho}\big(\bl \theta(\mb B_n)\big).$

%This shows that the reducing subspaces of $M_{\bl F}$ can be realized as Hilbert spaces on $\bl \theta (\mb B_n).$
It is natural to ask  if two proper maps are represented by the same basic polynomials, whether the corresponding tuples of multiplication operators are unitarily equivalent. This will enable us to describe the joint reducing subspaces of tuple of multiplication operators by $\tilde F$ from that of the associated basic polynomials. We recall a definition from the literature \cite{BM}.
\begin{defn}
Given a domain $D$ in $\mb C^n$ and analytic Hilbert module $\m H\subseteq \mathcal O(D)$ over $\C[\bl z]$, the tuple of multiplication operators $\mathbf M:=(M_1,\ldots, M_n)$ on  $\m H$ is said to be \emph{homogeneous} with respect to a subgroup $H$ of ${\rm Aut}(D)$ if 
\begin{enumerate}[leftmargin=*]
    \item the Taylor spectrum $\sigma(\mathbf M)\subseteq\ov{D}$;
    \item each element of $H$ is holomorphic on a neighbourhood of $\sigma(\mathbf M)$ and
    \item $\Phi(\mathbf M)=(M_{\phi_1},\ldots, M_{\phi_n})$ is unitarily equivalent to $\mathbf M$ for all $\Phi=(\phi_1,\ldots,\phi_n)\in H$,
    
\end{enumerate}
  where $M_i$ is the multiplication operator by the $i$-th coordinate function of $D$.
\end{defn}

We specialize to $\Omega = \bl\theta(\mathbb B_n)$ and assume that the multiplication operator $\mathbf M$ on $\wi{\mathcal H}\big(\bl \theta(\mb B_n)\big)$ is homogeneous with respect to $\mathrm{Aut}\big(\bl\theta(\mathbb B_n)\big)$. Recall that $\bl F = \bl\theta\circ\Psi$ and $\tilde{\bl F} = \Phi\circ\bl {F}.$ An application of holomorphic functional calculus to $\mathbf M_{\bl F}$ and $\mathbf M$  together with  Remark \ref{multbyco} imply that $\mathbf M_{\tilde{\bl F}}$ is unitarily equivalent to $\Phi(\mathbf M).$ Therefore, from homogeneity of $\mathbf M$, it follows that $\mathbf M_{\tilde{\bl F}}$ is unitarily equivalent to $\mathbf M.$ Thus any two tuples of multiplication operators induced by two proper maps with the same basic polynomials (hsop) are unitarily equivalent. Conversely, note that the multiplication operators induced by two proper maps from $\mathbb B_n$ to $\bl\theta(\mathbb B_n)$ are unitarily equivalent, then, in particular, for every proper map $\tilde{\bl F}:\mathbb B_n\ra \bl\theta(\mathbb B_n)$, $\mathbf M_{\tilde{\bl F}}$ is unitarily equivalent to $\mathbf M_{\bl\theta}$. Consequently, by holomorphic functional calculus, $\Phi(\mathbf M)$ is unitarily equivalent to $\mathbf M$ for all $\Phi\in \mathrm{Aut}\big(\bl\theta(\mathbb B_n)\big).$ Thus, the multiplication operator $\mathbf M$ becomes homogeneous with respect to $\mathrm{Aut}\big(\bl\theta(\mathbb B_n)\big)$ and we have the following theorem. Let $\mathrm{Prop}_{\bl\theta}\big(\mathbb B_n, \bl \theta(\mb B_n)\big)$ be the class of all proper holomorphic maps from  $\mb B_n$ onto $\bl \theta(\mb B_n)$ whose associated basic polynomial map is $\bl\theta$. 

\begin{thm}

%Further consider $\l>0$ such that $\mathrm{Prop}_{\bl\theta}\big(\mathbb B_n, \bl \theta(\mb B_n)\big)\subseteq \mathrm{mult}\big(\mb A^{(\l)}(\mathbb{B}_n)\big).$ 
For $\l\geq n$, any two $\tilde{\bl F_1}$ and $\tilde{\bl F_2}$ in $\mathrm{Prop}_{\bl\theta}\big(\mathbb B_n, \bl \theta(\mb B_n)\big),$ the tuples of multiplication operators $\mathbf M_{\tilde{\bl F_1}}$ and $\mathbf M_{\tilde{\bl F_2}}$ are unitarily equivalent on $\mb A^{(\l)} (\mathbb{B}_n)$ if and only if the multiplication operator $\mathbf M$ induced by $\bl\theta$ is homogeneous with respect to $\mathrm{Aut}\big(\bl\theta(\mathbb B_n)\big).$
\end{thm}
\begin{rem}
The hypothesis $\l\geq n$ ensures that $\mathrm{Prop}_{\bl\theta}\big(\mathbb B_n, \bl \theta(\mb B_n)\big)$ is a subset of the multiplier algebra of 
%$\mathrm{mult}\big(\mb A^{(\l)}(\mathbb{B}_n)\big)$ of  
$\mb A^{(\l)} (\mathbb{B}_n)$ \cite[Theorem 1.3]{BM}.
%is satisfied in the case when $\lambda\geq n$.
\end{rem}
The theorem above leads us to the problem of characterizing  operators which are homogeneous under the action of ${\rm Aut}\big(\bl\theta(\mb B_n)\big).$ The domains  
\Bea
\{\bl\theta_G(\mb B_n): G \,\, \mbox{is a pseudoreflection group}\,\,\}
\Eea
are not in general homogeneous \cite[Lemma 4.3]{OY}, where $\bl\theta_G$ is a hsop associated to $G.$ Classifying  homogeneous operators on such domains indicates to a new direction of research.

\medskip \textit{Acknowledgment}.
We express our sincere thanks to Prof. Gadadhar Misra for several fruitful discussions and many useful suggestions, which resulted in a considerable refinement of the original draft.

%\begin{prop}
%Let ${\bl F} : \mb B_n \to \bl \theta(\mb B_n)$ be a proper holomorphic map. Then $M_{{\bl F}}$ and $M_{\bl \theta}$ are unitarily equivalent on $\mb A^{(\l)}
%(\mathbb{B}_n)$ if and only if $M_{\bl \theta}$ is homogeneous under the action of the automorphism group of $\bl \theta(\mb B_n).$
%\end{prop}

\end {document}